\documentclass{article}
\usepackage{arxiv}
\usepackage{comment}
\usepackage[utf8]{inputenc} % allow utf-8 input
\usepackage[T1]{fontenc}    % use 8-bit T1 fonts
\usepackage{hyperref}       % hyperlinks
\usepackage{url} 
\usepackage{booktabs}       % professional-quality tables
\usepackage{amsmath,amssymb}
\usepackage{lipsum}
\allowdisplaybreaks
\usepackage{amsthm}
\usepackage{amsfonts} 
\usepackage{multirow}% blackboard math symbols
\usepackage{nicefrac} 
\usepackage{mathtools}
\DeclarePairedDelimiter\abs{\lvert}{\rvert}
\DeclarePairedDelimiter\norm{\lVert}{\rVert}
\usepackage{microtype}      % microtypography
\usepackage{xcolor}
\usepackage{graphicx}
\usepackage{multicol}
\graphicspath{ {./figs/} }
\usepackage{stmaryrd}
\usepackage{bm}
\usepackage[normalem]{ulem}
\usepackage[tickmarkheight=0.1cm]{todonotes}
\usepackage[ruled,vlined]{algorithm2e}
\usepackage{caption}
\usepackage{subcaption}
\usepackage[toc]{appendix}
\usepackage{enumerate}
\usepackage{enumitem}
\usepackage{orcidlink}

\newtheorem{hypx}{}

\numberwithin{equation}{section}
\newcounter{example} % Create a counter for examples
\renewcommand{\theexample}{\arabic{example}} % Number examples sequentially

\newenvironment{example}[1][]{%
    \refstepcounter{example} % Increment the example counter
    \noindent\textbf{Example \theexample.} % Example label in bold
    \ifx&#1&%
    \else \label{#1}%
    \fi
}{\par}

\newtheorem{theorem}{Theorem}[section]
\newtheorem{definition}{Definition}[section]
\newtheorem{lemma}{Lemma}[section]
\newtheoremstyle{iremark}
  {\topsep}   % ABOVESPACE
  {\topsep}   % BELOWSPACE
  {\upshape}  % BODYFONT
  {0pt}       % INDENT (empty value is the same as 0pt)
  {\itshape}  % HEADFONT
  {.}         % HEADPUNCT
  {5pt plus 1pt minus 1pt} % HEADSPACE
  {\thmname{#1}\thmnumber{ \itshape#2}\thmnote{ (#3)}} % CUSTOM-HEAD-SPEC

\theoremstyle{iremark}
\newtheorem{remark}{Remark}
\newenvironment{nalign}{
    \begin{equation}
    \begin{aligned}
}{
    \end{aligned}
    \end{equation}
    \ignorespacesafterend
}

\hypersetup{
    colorlinks=true, % make the links colored
    linkcolor=blue, % color TOC links in blue
    urlcolor=red, % color URLs in red
    citecolor=magenta,
    linktoc=all % 'all' will create links for everything in the TOC
}

\newcommand{\jinz}{{j \in \mathbb{Z}}}
\newcommand{\spc}{{\hspace{0.5cm}}}

\newcommand{\jmh}{{j-\frac{1}{2}}}
\newcommand{\jph}{{j+\frac{1}{2}}}
\newcommand{\nph}{{n+\frac{1}{2}}}
\newcommand{\jmtbt}{{j-\frac{3}{2}}}
\newcommand{\jpo}{{j+1}}
\newcommand{\half}{\frac{1}{2}}

\usepackage{caption}
\usepackage{subcaption}

\newcommand*\dif{\mathop{}\!\mathrm{d}}

\newcommand{\cblue}[1]{\textcolor{black}{#1}}

\newenvironment{frcseries}{\fontfamily{frc}\selectfont}{}

\newcommand{\no}{\nonumber}
\title{A  MUSCL-Hancock scheme for non-local conservation laws}

\author{
Nikhil Manoj\orcidlink{0009-0003-6812-8633}\\
School of Mathematics\\
Indian Institute of Science Education and Research\\
Trivandrum --695551\\
\texttt{nikhilmanoj2020@iisertvm.ac.in} \\
\And
G. D. Veerappa Gowda\orcidlink{0009-0009-8525-4790} \\
Centre for Applicable Mathematics\\
Tata Institute of Fundamental Research\\
Bangalore -- 560065\\
\texttt{gowda@tifrbng.res.in} \\
\And
Sudarshan~Kumar~K. \orcidlink{0009-0005-1186-4167} \thanks{Corresponding author} \\
School of Mathematics\\
Indian Institute of Science Education and Research\\
Trivandrum --695551\\
\texttt{sudarshan@iisertvm.ac.in} \\
}
\begin{document}
\maketitle

\begin{abstract}
In this article, we propose a MUSCL-Hancock-type second-order scheme  for the discretization of a general class of non-local conservation laws and present its convergence analysis. The main difficulty in designing a MUSCL-Hancock-type scheme for non-local equations lies in the discretization of the convolution term, which we carefully formulate to ensure second-order accuracy and facilitate rigorous convergence analysis. We derive several essential estimates  including $\mathrm{L}^\infty,$ bounded variation $(\mathrm{BV})$ and $\mathrm{L}^1-$ Lipschitz continuity in time,  which together with the Kolmogorov’s compactness theorem yield the convergence of the approximate solutions to a weak solution. Further, by incorporating a mesh-dependent modification in the slope limiter, we establish convergence to the entropy solution. Numerical experiments are provided to validate the theoretical results and to demonstrate the improved accuracy of the proposed scheme over its first-order counterpart.
\end{abstract}

\section{Introduction}
Conservation laws with non-local flux have emerged as powerful tools for modeling a wide range of flow-like physical processes that involve interactions over spatial neighborhoods. In contrast to local conservation laws, where the flux at a point depends solely on the value of the unknown at that point, non-local models incorporate surrounding information, typically through a convolution of the unknown with a given kernel function. The non-local term in the flux provides a more realistic description of physical scenarios by capturing phenomena such as the impact of surrounding traffic density on driver behavior in traffic flow \cite{bayen2021, blandin2016, chiarello-globalentropy, chiarellotosin2023,ciotirfayad2021, SopasakisKatsoulakis2006},  the influence of nearby pedestrians in crowd dynamics \cite{burgergoatin2020,colomboGaravelloMercier2012, ColomboHertyMercier2011, ColomboMercier2012}, or the effect of neighboring particle concentrations in sedimentation \cite{betancourt2011}.
These models have been extensively studied from a theoretical standpoint \cite{keimer2017, Keimer2018, keimer2019}, and several numerical schemes have been proposed for their approximation. In particular, a  variety of first-order finite volume schemes have been developed, see \cite{betancourt2011, amorim2015, aggarwal2015, chiarello2018, friedrich2018, friedrich2023}. While first-order methods are robust and aid in establishing well-posedness, they tend to be overly diffusive, making them inadequate for applications requiring high accuracy. This has led to a growing interest in the design of second- and higher-order schemes that can offer improved resolution. For non-local problems, however, the development of high-order numerical methods is still in its early stages; see~\cite{gowda2023, friedrich2019CWENO, chalons2018, manoj2024} for some recent advances.

In this work, we present a provably convergent, single-stage, second-order MUSCL–Hancock (MH)-type scheme for a broad class of non-local conservation laws. The MH method, originally proposed in \cite{vanleer1984}, combines MUSCL-type spatial reconstruction \cite{van1979towards} with a predictor–corrector time integration based on a midpoint rule. Starting from the reconstructed values at time $t^n$, the predictor step uses the Taylor expansion to compute the interface values at $(t^\nph)$. These are then used in the corrector step to evaluate the numerical fluxes. Compared to the conventional two-stage Runge–Kutta-based MUSCL schemes, the MH approach is computationally more efficient, as it requires only one spatial reconstruction and one numerical flux evaluation per time step.  The method has seen widespread use due to its simplicity and efficiency; see \cite{ berthon2006,waagan2009, guinot2012, chandrasekhar2020, tong2023, tong2024} for recent advances and applications of this scheme in various contexts.

The main difficulty in designing a MUSCL–Hancock-type scheme for non-local conservation laws lies in the accurate discretization of the non-local term in the flux, which involves the convolution of the unknown with a given kernel function. Achieving second-order accuracy while enabling a rigorous convergence analysis necessitates a careful treatment of this term in both the predictor and corrector steps of the scheme. In \cite{gowda2023}, a MUSCL- Hancock-type scheme was proposed for a class of non-local traffic flow models with piecewise smooth and non-increasing convolution kernels. Although the method was found to be effective in computations, a rigorous convergence proof for the specific  MH scheme proposed in \cite{gowda2023} remains out of reach. In this work, we formulate  novel discretizations of the convolution term  and  propose a new MH-type scheme for the problem originally studied in \cite{amorim2015, aggarwal2015}. Specifically, in the predictor step, we compute the interface convolution values using a piecewise linear reconstruction of the convolution term, where the slope in each cell is suitably determined. In the corrector step, we discretize the convolutions at the intermediate time level  $t^\nph$  using an appropriate quadrature rule.

As a major contribution of this work, we establish the convergence of the proposed scheme to the unique entropy solution of the problem under consideration. To this end, we first reformulate the scheme in a suitable form  and prove the positivity-preserving property and $\mathrm{L}^\infty$ stability under an appropriate CFL condition. We  then derive a total variation bound and a time-continuity estimate. These analyses present several difficulties, which we overcome using the properties of the slope-limiter and the novel discrete convolutions in the predictor step. Leveraging  the derived estimates, Kolmogorov’s compactness theorem provides us the existence of a convergent subsequence of approximate solutions. To ensure that the entire scheme converges to the entropy solution of the underlying problem, we introduce a mesh-dependent modification to the slope limiter, following the strategy in \cite{vila1988, vila1989, gowda2023}. The resulting scheme is then shown to converge to the unique entropy solution. We also provide numerical examples to support the theoretical results and to compare the proposed method with a first-order scheme and a conventional second-order  MUSCL-type scheme based on Runge-Kutta time integration. 

% Furthermore, the proposed MUSCL-Hancock scheme seems to be more computationally efficient than the  Runge-Kutta scheme.

The remainder of the paper is structured as follows. In Section~\ref{section: preliminaries}, we present the necessary preliminaries for the class of non-local conservation laws considered in this work. Section~\ref{section:mhscheme} describes the proposed numerical scheme in detail. Uniform a priori estimates for the scheme are established in Section~\ref{section:estimates}. In Section~\ref{section:entropy}, we prove the convergence of the scheme to the unique entropy solution. Numerical experiments are presented in Section~\ref{section:numerics}.

\section{Non-local conservation laws}\label{section: preliminaries}

We are interested in the initial value problem for one dimensional  non-local conservation laws considered in \cite{aggarwal2015, amorim2015}: 
\begin{nalign}\label{eq:problem}
    \partial_{t}\rho + \partial_{x}f(\rho, A(t,x)) &= 0, \quad \quad \quad (t,x) \in  (0,\infty) \times \mathbb{R},\\
    \rho(0, x) &= \rho_{0}(x), \quad \,\,\, x \in \mathbb{R}, 
\end{nalign}
where $\mu:\mathbb{R}\rightarrow \mathbb{R}, f:\mathbb{R}\times\mathbb{R}\rightarrow \mathbb{R}$ and $\rho:(0,\infty)\times\mathbb{R}\rightarrow \mathbb{R}$ denote the convolution kernel, the given flux function and the unknown quantity, respectively. Here, the convolution term $\mu*\rho$ is defined as \begin{align}\label{eq:conv_defn}
A(t,x) := \mu*\rho(t,x) = \int_{-\infty}^\infty \mu(x- y) \rho(t, y) \dif y.
\end{align}

\subsection{Notations}
We use the following notations throughout the paper:
$\mathbb{R}_+:=[0,\infty),$ denotes the set of non-negative real numbers. For $a,b \in \mathbb{R},$ denote $\mathcal{I}(a,b):= (\min(a,b), \max(a,b))$ and $\norm{\cdot} :=\norm{\cdot}_{\mathrm{L}^{\infty} }.$ Finally, for a function $g:\mathbb{R} \rightarrow \mathbb{R}$ we denote by $\mathrm{TV}(g)$ the total variation seminorm of $g.$

\subsection{Hypotheses}
The functions $f$ and $ \mu$ in \eqref{eq:problem} and \eqref{eq:conv_defn} are assumed to satisfy the following hypothesis:
\begin{hypx}\label{hyp:H1}
For all \( A \in \mathbb{R} \),
\begin{enumerate}[label=(\roman*), nosep, leftmargin=2cm]
    \item \( f \in \mathrm{C}^{2}(\mathbb{R} \times \mathbb{R}; \mathbb{R}),\quad
        \partial_{\rho}f \in \mathrm{L}^{\infty}( \mathbb{R} \times \mathbb{R}; \mathbb{R}) \),
    \item \( f(0, A) = 0. \)
\end{enumerate}
\end{hypx}

\begin{hypx}\label{hyp:H2}
There exists an \( M > 0 \) such that for all \( \rho, A \) in the respective domains,
\[
    \lvert \partial_{A}f  \rvert, \ \lvert \partial^{2}_{AA}f \rvert \leq M \lvert \rho \rvert.
\]
\end{hypx}

\begin{hypx}\label{hyp:H3}
\( \partial_{\rho}f \in \mathrm{W}^{1, \infty}(\mathbb{R} \times \mathbb{R}; \mathbb{R}) \).
\end{hypx}

\begin{hypx}\label{hyp:H4}
\( \mu \in (\mathrm{C}_{c}^{2} \cap \mathrm{W}^{2, \infty})(\mathbb{R}; \mathbb{R}). \)
\end{hypx}

Solutions to non-linear conservation laws need not be smooth in general, even in the case when the initial datum is smooth. Therefore, we consider the weak/entropy formulations of the solutions to \eqref{eq:problem} defined below.
\subsection{Weak and entropy solutions}
\begin{definition}\label{defn:weaksoln}(Weak solution)
A function $\rho \in ({\rm L}^{\infty} \cap {\rm L}^{1})$$ ([0,T) \times \mathbb{R}; \mathbb{R})$, $T >0$, is a weak solution of (\ref{eq:problem}) if
    \begin{align}\label{eq:weaksoln}
        \int_{0}^{T}\int_{-\infty}^{+\infty}\bigl(\rho\partial_{t}\varphi + f(\rho, A(t,x)) \partial_{x}\varphi\bigr)(t,x) \dif x \dif t + \int_{-\infty}^{+\infty}\rho_{0}(x)\varphi(0, x) \dif x = 0
    \end{align} 
$\mbox{for all} \ \varphi \in {\rm C}_{c}^{1}([0, T) \times \mathbb{R}; \mathbb{R}).$    
\end{definition}
Next, following the definition in \cite{betancourt2011}, we  define an entropy solution to the problem \eqref{eq:problem} as follows:
\begin{definition}\label{defn: entsoln}(Entropy solution)
    A function $\rho \in ({\rm L}^{\infty} \cap {\rm L}^{1})$$ ([0,T) \times \mathbb{R}; \mathbb{R})$, $T >0$, is an entropy weak solution of (\ref{eq:problem}) if
    \begin{nalign}\label{eq:entropy_ineq}
        &\int_{0}^{T}\int_{-\infty}^{+\infty}\Big(\lvert\rho-\kappa \rvert \partial_{t}\varphi + \mathrm{sgn}(\rho- \kappa)\bigl(f(\rho, A(t,x)) - f(\kappa, A(t,x))\bigr)\partial_{x}\varphi \\& - \mathrm{sgn}(\rho -\kappa)\partial_{A}f(\kappa, A(t,x))\partial_{x}(A(t,x)) \varphi \Big)(t, x) \dif x \dif t + \int_{-\infty}^{+\infty}\lvert \rho_{0}(x) - \kappa\rvert \varphi(0, x) \dif x \geq 0
    \end{nalign}
    $\mbox{for all} \ \varphi \in {\rm C}_{c}^{1}([0, T) \times \mathbb{R}; \mathbb{R}^{+}) \ \mbox{and} \  \kappa \in \mathbb{R}, \ \mbox{where $\mathrm{sgn}$ is the sign function.}$
\end{definition}
We note that the weak solution to \eqref{eq:problem} that satisfies the entropy inequality \eqref{eq:entropy_ineq} is unique, as observed in \cite{betancourt2011, blandin2016}.
\section{A MUSCL-Hancock-type second-order scheme}\label{section:mhscheme}
We discretize the spatial domain into Cartesian grids with a uniform mesh of size $\Delta x$. The spatial domain is now a union of cells of the form $[x_{j-\frac{1}{2}}, x_{j+\frac{1}{2}}]$ where $x_{j+\frac{1}{2}}- x_{j-\frac{1}{2}}= \Delta x$ and $x_j = j\Delta x.$ We fix $T>0$ and the time-step is denoted by $\Delta t$ and $t^{n}= n\Delta t \ \mbox{for} \ n \in \mathbb{N}$, $\lambda= \frac{\Delta t}{\Delta x}.$ The initial datum $\rho_{0}$ is discretized as
\begin{align*}
    \rho_{j}^{0}= \frac{1}{\Delta x } \int_{x_{j-\frac{1}{2}}}^{x_{j+\frac{1}{2}}}\rho_{0}(x) \dif x  \quad \mbox{for} \ j \in \mathbb{Z}.
\end{align*}
Given the cell-average solutions $\{\rho_{j}^n\}_{\jinz}$ at the $n-$th time-level, the first step is to obtain a piecewise linear reconstruction as follows
\begin{align}\label{eq:reconstruction_pl}
    \tilde{\rho}_{\Delta}^n(x) &:= \rho_{j}^n + \frac{(x-x_{j})}{\Delta x}\sigma_{j}^n \quad \mbox{for} \, x \in (x_{\jmh}, x_{\jph}),
\end{align}
where the slopes $\sigma_j^n$ are computed as 
\begin{align}
    \sigma^{n}_{j}& = 2 \theta \textrm{minmod}\left((\rho^{n}_{j}-\rho^{n}_{j-1}), \  \frac{1}{2}(\rho^{n}_{j+1}-\rho^{n}_{j-1}), \  (\rho^{n}_{j+1}-\rho^{n}_{j})\right), \label{slope-1}
% &sgn( \rho^{n}_{i,j+1}-\rho^{n}_{i,jcbl} )=sgn( \rho^{n}_{i,j}-\rho^{n}_{i,j-1} )=sgn( \rho^{n}_{i+1,j}-\rho^{n}_{i,j} )=sgn( \rho^{n}_{i,j}-\rho^{n}_{i-1,j} )\no
  \end{align}
with $\theta \in {[0,0.5]}.$ The left and right face values (at the interface $x=x_{\jph}$) of the reconstructed polynomial are given by
\begin{align}\label{eq:reconvalues}
\rho_{j+\frac{1}{2}}^{n,-}= \rho^{n}_{j}+\frac{\sigma_{j}^{n}}{2}, \ \ \ \rho_{j+\frac{1}{2}}^{n,+}= \rho^{n}_{j+1}-\frac{\sigma_{j+1}^{n}}{2}. 
\end{align} Next, a finite volume integration of the conservation law \eqref{eq:problem} in the domain $[t^n, t^{n+1}] \times [x_{\jmh}, x_{\jph}]$ yields
\begin{nalign}\label{eq:finitevolume}
    \int_{x_{\jmh}}^{x_{\jph}}\rho(t^{n+1},x) \dif x &=\int_{x_{\jmh}}^{x_{\jph}}\rho(t^{n},x) \dif x \\ & \spc -\left(\int_{t^n}^{t^{n+1}}f(\rho(t,x_{\jph}),A(t, x_{\jph})) \dif t -\int_{t^n}^{t^{n+1}}f(\rho(t,x_{\jmh}),A(t, x_{\jmh}))\dif t \right)\\
    & \approx  \int_{x_{\jmh}}^{x_{\jph}}\rho(t^{n},x) \dif x \\ & \spc -\Delta t \left(f(\rho(t^\nph,x_{\jph}),A(t^{\nph}, x_{\jph}))  -f(\rho(t^\nph,x_{\jmh}),A(t^{\nph}, x_{\jmh}))\right),
\end{nalign}
where we have used midpoint quadrature rule in approximating the flux integral.
From \eqref{eq:finitevolume}, we formulate a MUSCL-Hancock type finite volume scheme as follows  \begin{nalign}\label{eq:mh}
    \rho^{n+1}_{j} &= \rho^{n}_{j} - \lambda\left[F^{n+\frac{1}{2}}_{j+\frac{1}{2}} - F^{n+\frac{1}{2}}_{j-\frac{1}{2}}\right],
\end{nalign} 
where for an  appropriately chosen numerical flux $F,$ we define $\displaystyle F_{\jph}^\nph := F(\rho_{j+\frac{1}{2}}^{n+\frac{1}{2},-}, \rho_{j+\frac{1}{2}}^{n+\frac{1}{2}, +}, A^{\nph}_{\jph}), \; j \in \mathbb{Z},$ and the mid-time density values are obtained using Taylor series expansions as follows
\begin{nalign}\label{eq:midtimevalues}
\rho_{j+\frac{1}{2}}^{n+\frac{1}{2},-}&=\rho_{j+\frac{1}{2}}^{n,-}-\frac{\lambda}{2} \left (f(\rho_{j+\frac{1}{2}}^{n,-},A^{n,-}_{j+\frac{1}{2}})-f(\rho_{j-\frac{1}{2}}^{n,+},A^{n,+}_{j-\frac{1}{2}}) \right), \\
\rho_{j-\frac{1}{2}}^{n+\frac{1}{2},+}&=\rho_{j-\frac{1}{2}}^{n,+}-\frac{\lambda}{2} \left (f(\rho_{j+\frac{1}{2}}^{n,-},A^{n,-}_{j+\frac{1}{2}})-f(\rho_{j-\frac{1}{2}}^{n,+},A^{n,+}_{j-\frac{1}{2}}) \right), \;\;\;j\; \in \mathbb{Z}.
\end{nalign}
The terms $A^{\nph}_{j \pm \half}$ in \eqref{eq:mh} and  $A^{n,\pm}_{j\mp\frac{1}{2}}$ in \eqref{eq:midtimevalues} are suitable approximations of the convolution terms     $A(t^{\nph},{x_{j \pm \half}})$ and $A(t^n, x_{j\mp \half}^{n, \pm})$, respectively. These approximations are elaborated in the following section.

\subsection{Approximation of convolution terms}
% The predictor step \eqref{eq:midtimevalues} needs the approximate convolution values, $A^{n,+}_{\jmh}$ and $A^{n,-}_{\jph}$, at the left and right interfaces respectively, of each cell $[x_{\jmh}, x_{\jph}],$ $\jinz.$
To begin with, using Taylor series expansions we write 
\begin{nalign}\label{eq:conv_lr}
    A (t^n, x_{\jmh}) &= A (t^n, x_{j}) - \frac{\Delta x}{2} \partial_{x}A(t^n, x_{j}) + \mathcal{O}(\Delta x^2),\\ 
    A (t^n, x_{\jph}) &= A (t^n, x_{j}) + \frac{\Delta x}{2} \partial_{x}A(t^n, x_{j}) + \mathcal{O}(\Delta x^2).
\end{nalign}
Next, we approximate the quantities on the right hand side of \eqref{eq:conv_lr} using suitable quadrature rules. To this end, using midpoint quadrature rule, we first write
\begin{nalign}\label{eq:conv_appr}
    A(t^n,x_{j})&= \int_{\mathbb{R}} \mu(x_{j}- y) \rho(t^n,y) \dif y  \\
& = \sum_{l \in \mathbb{Z}} \int_{x_{l-\frac{1}{2}}}^{x_{l+\frac{1}{2}}} \mu(x_{j}- y) \rho(t^n, y) \dif y  \approx
\Delta x \sum_{l \in \mathbb{Z}}\mu_{j-l}        \ \rho^{n}_{l} =: A^{n}_{j},
\end{nalign}
where $\mu_{j}:= \mu(j\Delta x), \, \jinz.$
Further, using the central difference approximation to the derivative, we write
\begin{nalign}\label{eq:der_approx}
   \Delta x \partial_{x}A(t^n, x_{j}) &= \frac{1}{2}\left(A (t^n, x_{j+1})- A (t^n, x_{j-1})\right) + \mathcal{O}(\Delta x^3)\\
   & \approx \frac{1}{2}(A_{j+1}^n-A_{j-1}^n) + \mathcal{O}(\Delta x^3).
\end{nalign}
Plugging in the approximations \eqref{eq:conv_appr} and \eqref{eq:der_approx} to \eqref{eq:conv_lr}, we obtain
\begin{nalign}\label{eq:lrapprox}
  A (t^n, x_{\jmh}) 
  &\approx A_{j}^{n} - \frac{1}{4}(A_{j+1}^n-A_{j-1}^n) =: \hat{A}_{\jmh}^{n,+},\\
  A (t^n, x_{\jph}) 
  &\approx A_{j}^{n}+ \frac{1}{4}(A_{j+1}^n-A_{j-1}^n) =: \hat{A}_{\jph}^{n,-}.
\end{nalign}
Motivated by \eqref{eq:lrapprox} and using the parameter $\theta$ from \eqref{slope-1}, we redefine the left and right approximate convolutions in the cell $[x_{\jmh}, x_{\jph}]$ as follows
\begin{nalign}\label{eq:lrconv}
     A_{\jph}^{n,-}&:= A_{j}^{n}+ \frac{s_{j}^n}{2} , \quad 
A_{\jmh}^{n,+} := A_{j}^{n}- \frac{s_{j}^n}{2}, \end{nalign}
$\displaystyle \mbox{where} \, s_{j}^n := 2\theta\frac{(A_{j+1}^n-A_{j-1}^n)}{2} \,\, \mbox{for} \,\, \theta \in  [0,0.5].$
% To summarize, we now have a piecewise linear reconstruction of the convolution term $A(t,x)= \mu*\rho(t,x)$ of the form:
% \begin{align*}
%     \tilde{A}(t^n, x)
%  := A_{j}^{n} + {s_{j}^n}\frac{(x-x_{j})}{\Delta x} \quad \mbox{for} \quad x \in (x_{\jmh}, x_{\jph}), \quad j \in \mathbb{Z},
% \end{align*}
% where $\displaystyle \frac{1}{\Delta x}s_{j}^n$ is the slope of the linear reconstruction in the cell $[x_{\jmh}, x_{\jph}].$

\begin{remark}\label{remark:lrconv_}
    If we choose $\theta = 0.5,$ the approximations \eqref{eq:lrconv} and \eqref{eq:lrapprox} coincide -- i.e., $A_{\jph}^{n,-}= \hat{A}_{\jph}^{n,-} , \, \mbox{and} \, 
A_{\jmh}^{n,+}= \hat{A}_{\jmh}^{n,+}.$ On the other hand, if we choose $\theta =0,$ then $s_j^n=0$ for all $j \in \mathbb{Z}$ and consequently $A_{\jmh}^{n,+}= A_{\jph}^{n,-}= A_{j}^n.$
\end{remark}
% $A^n_j$ denote the approximation of the convolution terms $A(t^n,x_{j}),$
Once the mid-time density values $\rho^{\nph,\mp}_{j \pm\half}$ are computed using \eqref{eq:midtimevalues}, the mid-time convolution approximation $A^{(n+\frac{1}{2})}_{j+\frac{1}{2}} \approx A(t^{\nph}, x_{\jph})$ in \eqref{eq:mh} is approximated using the trapezoidal quadrature rule as follows
\begin{nalign}\label{eq:midtimeconv}
     A(t^\nph, x_{\jph}) &=\sum_{l \in \mathbb{Z}} \int_{x_{l-\frac{1}{2}}}^{x_{l+\frac{1}{2}}} \mu(x_{\jph}- y) \rho(t^{n+\frac{1}{2}}, y) \dif y\\
     &\approx \frac{\Delta x}{2}  \sum_{l \in \mathbb{Z}}\left[\mu_{j+1-l}        \ \rho^{n+\frac{1}{2},+}_{l-\frac{1}{2}}+ \mu_{j-l} \ \rho^{n+\frac{1}{2},-}_{l+\frac{1}{2}}\right] =: A^{n+\frac{1}{2}}_{j+\frac{1}{2}}.
\end{nalign}

\subsection{Numerical flux}
Throughout this article, we consider the scheme \eqref{eq:mh} with a Lax-Friedrichs type numerical flux, which we define following the formulation in \cite{aggarwal2015, amorim2015, aggarwal_holden_vaidya2024} as described below
\begin{nalign}\label{eq:midtimenumflux}
F(u, v, A) = \frac{ f(u, A) + f(v, A)}{2} - \frac{\alpha(v-u)}{2 \lambda}, \quad \mbox{for} \,\, u, v \in \mathbb{R},
\end{nalign}
and for a fixed coefficient $\alpha \in (0, \frac{8}{27})$ .

%Here
% \begin{align*}
%     f^{n+\frac{1}{2}}_{i+\frac{1}{2}}(\rho) \coloneqq f(\rho, A^{(n+\frac{1}{2})}_{i+\frac{1}{2}}).
%    \end{align*}

\par
Finally, for a fixed mesh size $\Delta x,$ we denote by $\rho_{\Delta}(t, x)$  the piecewise constant approximate solution  obtained from the scheme \eqref{eq:mh}:
\begin{align}\label{eq:approxsoln}
    \rho_{\Delta}(t, x) := \rho_{j}^{n} \quad \mbox{for} \quad (t,x) \in [t^{n}, t^{n+1}) \times [x_{j-\frac{1}{2}}, x_{j+\frac{1}{2}}) \quad \mbox{for} \ n \in \mathbb{N} \ \mbox{and} \ j \in \mathbb{Z}.
\end{align}

\begin{remark}
    Setting $\theta = 0$ in \eqref{slope-1} and \eqref{eq:lrconv} leads to $\sigma_{j}^n = s_{j}^n = 0,$ which implies that $\rho^{n,-}_{\jph} =\rho^{n,+}_{\jmh}=\rho_{j}^{n}$ for all $j \in \mathbb{Z}.$ Using this, along with  Remark \ref{remark:lrconv_}, we further obtain $\rho^{n+\frac{1}{2},-}_{\jph} =\rho^{n+\frac{1}{2},+}_{\jmh}=\rho_{j}^{n}$ for all $j \in \mathbb{Z}.$ Consequently, the scheme \eqref{eq:mh} with numerical flux \eqref{eq:midtimenumflux} reduces to a first-order Lax-Friedrichs type scheme:
    \begin{nalign}\label{eq:lxf}
    \rho^{n+1}_{j} &= \rho^{n}_{j} - \lambda\left[F(\rho_{j}^{n},\rho_{j+1}^{n}, A^{n}_{\jph}) - F(\rho_{j-1}^{n}, \rho_{j}^{n}, A^{n}_{\jmh})\right], \;\;\;j\; \in \mathbb{Z},
\end{nalign}
where $ \displaystyle A^{n}_{\jph} := \frac{\Delta x}{2}\sum_{\jinz} (\mu_{j+1-l}+\mu_{j-l})\rho_{l}^n \approx A(t^n,x_{\jph}).$
\end{remark}

\section{Estimates on the numerical solutions}\label{section:estimates}
In this section, we establish certain essential estimates on the approximate solutions generated by the scheme \eqref{eq:mh}, namely $\mathrm{L}^\infty,$ $\mathrm{BV}$ and $\mathrm{L}^{1}$ time continuity estimates, which are required for the convergence analysis. Before proceeding into that, we start with some preliminary structural properties of the scheme. In the following remark, we obtain an estimate on the difference between the two consecutive left/right interface density values in \eqref{eq:reconvalues}. 
\begin{remark}\label{remark:ratioslopes}
From \eqref{slope-1}, we have   $\sigma_{j}^{n}\sigma_{j+1}^{n} \geq 0$ for $j \in \mathbb{Z}.$ This in turn implies that $\abs{\sigma_{j+1}^{n}-\sigma_{j}^{n}} \leq \max\{\abs{\sigma_{j+1}^{n}},\abs{\sigma_{j}^{n}}\}.$ As a result, we obtain
\begin{align}\label{eq:ratioslope}
    \abs[\Bigg]{\frac{\sigma_{j+1}^{n}-\sigma_{j}^{n}}{\rho_{j+1}^{n}-\rho_{j}^{n}}} &\leq \frac{\max\{\abs{\sigma_{j+1}^{n}},\abs{\sigma_{j}^{n}}\}}{\abs{\rho_{j+1}^{n}-\rho_{j}^{n}}} \leq 2\theta.
\end{align}
Consequently, using \eqref{eq:reconvalues}, for $\theta \in [0, 0.5],$ we write
\begin{nalign}\label{eq:ratiobd}
    \frac{\left(\rho_{j+\frac{3}{2}}^{n,-}-\rho_{j+\frac{1}{2}}^{n,-}\right)}{\rho_{j+1}^{n}-\rho_{j}^{n}} &= \frac{\rho_{j+1}^{n}-\rho_{j}^{n}}{\rho_{j+1}^{n}-\rho_{j}^{n}} + \frac{\left(\sigma_{j+1}^{n}-\sigma_{j}^{n}\right)}{2\left(\rho_{j+1}^{n}-\rho_{j}^{n}\right)} \geq 1 - \theta \geq \frac{1}{2},\\
     \frac{\left(\rho_{j+\frac{3}{2}}^{n,-}-\rho_{j+\frac{1}{2}}^{n,-}\right)}{\rho_{j+1}^{n}-\rho_{j}^{n}} &= \frac{\rho_{j+1}^{n}-\rho_{j}^{n}}{\rho_{j+1}^{n}-\rho_{j}^{n}} + \frac{\left(\sigma_{j+1}^{n}-\sigma_{j}^{n}\right)}{2\left(\rho_{j+1}^{n}-\rho_{j}^{n}\right)} \leq 1 + \theta \quad \mbox{for all} \, \jinz.
 \end{nalign}
Similarly, we obtain 
\begin{align}
    \frac{1}{2} \leq \frac{\left(\rho_{j+\frac{1}{2}}^{n,+}-\rho_{j-\frac{1}{2}}^{n,+}\right)}{\rho_{j+1}^{n}-\rho_{j}^{n}} \leq 1+\theta \quad \mbox{for all} \, \jinz.
\end{align}
% Hence, it follows that
% \begin{align}
%     \abs[\Bigg]{\frac{\rho_{j+\frac{1}{2}}^{n,+}-\rho_{j-\frac{1}{2}}^{n,+}}{\rho_{j+1}^{n}-\rho_{j}^{n}}}, \abs[\Bigg]{\frac{\rho_{j+\frac{3}{2}}^{n,-}-\rho_{j+\frac{1}{2}}^{n,-}}{\rho_{j+1}^{n}-\rho_{j}^{n}}} & \leq (1+\theta). 
% \end{align}
\end{remark}

Next, we present a bound on the reconstructed density values \eqref{eq:reconvalues}, in the following lemma.
\begin{lemma}\label{lemma:reconvaluebd} (Bound on reconstructed values)
Suppose that the piecewise constant approximate solution at the time level $t^n$ given by the sequence $\{\rho^n_j\}_{\jinz},$ satisfies $\rho_{j}^n \geq 0$ for all $\jinz.$ Then, for each $\jinz,$ the left and right interface values defined in \eqref{eq:reconvalues} are estimated as follows \begin{align}\label{eq:recon_bda}
    \abs{\rho^{n,-}_{\jph}}, \abs{\rho^{n,+}_{\jmh}} \leq (1+\theta)\rho_{j}^n.
\end{align}
\end{lemma}
\begin{proof}
    We split the proof into two cases: \\
    \textbf{Case (i)} $(\sigma_{j}^n \geq 0):$ In this case, we observe that
    \begin{align}
        \rho^{n,-}_{\jph} &= \rho_{j}^n+ \frac{1}{2}\sigma_{j}^n \geq \rho_{j}^n,
    \end{align} and
\begin{nalign}\label{eq:rightval_ub}
    \rho^{n,-}_{\jph} &= \rho_{j}^n+ \frac{1}{2}\sigma_{j} \leq  \rho_{j}^n+ \frac{1}{2}2 \theta (\rho_{j}^n-\rho_{j-1}^n)
     = (1+\theta)\rho_j^n - \theta\rho_{j-1}^n  \leq (1+\theta)\rho_{j}^n. 
\end{nalign}
\textbf{Case (ii)} $(\sigma_{j}^n < 0):$ In this case, we have \begin{align}
    \rho^{n,-}_{\jph} = \rho_{j}^n+ \frac{1}{2}\sigma_{j} \leq \rho_{j}^n,
\end{align} and
\begin{nalign}\label{eq:rightval_lb}
    \rho^{n,-}_{\jph} &=\rho_{j}^n - 2\theta  \frac{(\rho_{j}^n- \rho_{j+1}^n)}{2} \geq (1-\theta)\rho_{j}^n +\theta \rho_{j+1}^n \geq 0.
\end{nalign}
By combining both the cases, it follows that $\abs{\rho^{n,-}_{\jph}} \leq (1+\theta)\rho_{j}^n$ for all $\jinz.$ In a similar way, we obtain $\abs{\rho^{n,+}_{\jmh}} \leq (1+\theta)\rho_{j}^n$ for all $\jinz.$
\end{proof}

\begin{remark}  
For each $n \in \mathbb{N} \cup \{0\},$ the piecewise linear reconstruction $\tilde{\rho}^{n}_{\Delta}$ in \eqref{eq:reconstruction_pl} satisfies the following property on its total variation (see Lemma 3.1, Chapter 4, \cite{godlewski1991hyperbolic}) 
\begin{align}\label{eq:tv_recomstruction}
    \mathrm{TV}(\tilde{\rho}^{n}_{\Delta}) & \leq  \sum_{\jinz}\abs{\rho_{j+1}^n- \rho_{j}^n}.
\end{align}
\end{remark}

\begin{lemma}\label{lemma:midtimebd}
    (Bound on mid-time density values)
 Suppose that the piecewise constant approximate solution at the time level $t^n$ given by the sequence $\{\rho^n_j\}_{\jinz},$ satisfies $\rho_{j}^n \geq 0$ for all $\jinz.$ Then, the mid-time density values defined in \eqref{eq:midtimevalues} can be estimated as follows \begin{align}\label{eq:recon_bd}
    \abs{\rho^{\nph,-}_{\jph}}, \abs{\rho^{\nph,+}_{\jmh}} \leq (1+\theta)(1+\lambda\norm{\partial_{\rho}f})\rho_{j}^n \quad \mbox{for all} \, \jinz.
    \end{align}
\end{lemma}

\begin{proof}
    Using  \eqref{eq:midtimevalues} and the hypothesis \eqref{hyp:H1}, we estimate:
\begin{nalign}\label{eq:midvalue_bd}
    \abs{\rho^{\nph,-}_{\jph}}     &\leq \abs{\rho^{n,-}_{\jph}} + \frac{\lambda}{2}\abs{ f(\rho^{n,-}_{\jph}, A_{\jph}^{n,-})-f(0, A_{\jph}^{n,-}) +f(0, A_{\jmh}^{n,+}) - f(\rho^{n,+}_{\jmh}, A_{\jmh}^{n,+})}\\
    & \leq \abs{\rho^{n,-}_{\jph}} + \frac{\lambda}{2}\abs{ \partial_{\rho}f(\bar{\rho}^{n,-}_{\jph}, A_{\jph}^{n,-})}\abs{\rho^{n,-}_{\jph}} +\frac{\lambda}{2} \abs{\partial_{\rho}f(\bar{\rho}^{n,+}_{\jmh}, A_{\jmh}^{n,+})}\abs{\rho^{n,+}_{\jmh}}\\
    & \leq (1+\theta)(1+\lambda\norm{\partial_{\rho}f})\rho_{j}^n,
\end{nalign}
where $\bar{\rho}^{n,-}_{\jph} \in \mathcal{I}(0,{\rho}^{n,-}_{\jph})$ and $\bar{\rho}^{n,+}_{\jmh} \in \mathcal{I}(0, {\rho}^{n,+}_{\jmh}).$
Similarly, we obtain $\abs{\rho^{\nph,+}_{\jmh}} \leq (1+\theta)(1+\lambda\norm{\partial_{\rho}f})\rho_{j}^n.$
\end{proof}

Now, we establish that the proposed MUSCL-Hancock type scheme \eqref{eq:mh} gives non-negative solutions when the initial data is non-negative. 
\begin{theorem}\label{theorem:positivity}\label{theorem:pstivity}
    (Positivity-preserving property) Let the initial datum $\rho_0 \in \mathrm{L}^{\infty}(\mathbb{R};\mathbb{R}_{+}).$ Then, the approximate solutions  $\rho_{\Delta}$ obtained from the proposed scheme \eqref{eq:mh} satisfies $\rho_\Delta(t,x) \geq 0$ for a.e. $(t,x) \in \mathbb{R} \times \mathbb{R}_{+},$
provided the CFL condition
    \begin{align}\label{eq:cfl_psty}
         \frac{\Delta t}{\Delta x} &\leq \min\left\{\frac{8-27\alpha}{27\norm{\partial_\rho f}}, \frac{2}{27\norm{\partial_{\rho}f}}, \frac{\alpha}{\norm{\partial_\rho f}}\right\},
    \end{align} holds.
\end{theorem}
\begin{proof}
To prove this result, we use the principle of mathematical induction. For the base case $n=0,$  we have $\rho^{0}_{j} \geq 0$ for all $\jinz$ by assumption. Now, for any $n \in \mathbb{N}\cup \{0\},$ assume that $\rho_{j}^n \geq 0$ for all $\jinz.$ We show that $\rho_{j}^{n+1} \geq 0 $ for all $\jinz.$ To begin with, we write the scheme \eqref{eq:mh} in the form 
\begin{align}\label{eq:positivityscheme}
       \rho^{n+1}_{j} &= \rho_{j}^{n} - a^{\nph}_{j-\frac{1}{2}} (\rho_{j}^{n}-\rho_{j-1}^{n})          + b^{\nph}_{j+\frac{1}{2}}(\rho^{n}_{j+1} - \rho^{n}_{j})\\
   & \hspace{0.5cm}+ \lambda\left[F(\rho_{j+\frac{1}{2}}^{\nph,-}, \rho_{j-\frac{1}{2}}^{\nph, +}, A^{\nph}_{j-\frac{1}{2}}) - 
 F(\rho_{j+\frac{1}{2}}^{\nph,-}, \rho_{j-\frac{1}{2}}^{\nph,+}, A^{\nph}_{j+\frac{1}{2}})\right], \no  
\end{align}
where we define
\begin{nalign}\label{eq:aandb}
    a_{j-\frac{1}{2}}^{\nph} &\coloneqq  \lambda\frac{\left[F(\rho_{j+\frac{1}{2}}^{\nph,-}, \rho_{j-\frac{1}{2}}^{\nph, +}, A^{\nph}_{j-\frac{1}{2}}) - F(\rho_{j-\frac{1}{2}}^{\nph,-}, \rho_{j-\frac{1}{2}}^{\nph,+}, A^{\nph}_{j-\frac{1}{2}})\right]}{(\rho_{j+\frac{1}{2}}^{\nph, -}-\rho_{j-\frac{1}{2}}^{\nph,-})} \left(\frac{\rho_{j+\frac{1}{2}}^{\nph, -}-\rho_{j-\frac{1}{2}}^{\nph,-}}{\rho_{j}^{n}-\rho_{j-1}^{n}  } \right),\\
    b_{j+\frac{1}{2}}^{\nph} &\coloneqq - \lambda\frac{\left[F(\rho_{j+\frac{1}{2}}^{\nph,-}, \rho_{j+\frac{1}{2}}^{\nph, +}, A^{\nph}_{j+\frac{1}{2}}) - F(\rho_{j+\frac{1}{2}}^{\nph,-}, \rho_{j-\frac{1}{2}}^{\nph,+}, A^{\nph}_{j+\frac{1}{2}})\right]}{(\rho_{j+\frac{1}{2}}^{\nph, +}-\rho_{j-\frac{1}{2}}^{\nph,+})} \left(\frac{\rho_{j+\frac{1}{2}}^{\nph, +}-\rho_{j-\frac{1}{2}}^{\nph,+}}{\rho_{j+1}^{n}-\rho_{j}^{n}  } \right). 
\end{nalign}
Now, using the definition \eqref{eq:midtimenumflux}, we observe that 
\begin{nalign}\label{eq:Fdifratio_psty}
    &\frac{\left[F(\rho_{j+\frac{1}{2}}^{\nph,-}, \rho_{j-\frac{1}{2}}^{\nph, +}, A^{\nph}_{j-\frac{1}{2}}) - F(\rho_{j-\frac{1}{2}}^{\nph,-}, \rho_{j-\frac{1}{2}}^{\nph,+}, A^{\nph}_{j-\frac{1}{2}})\right]}{(\rho_{j+\frac{1}{2}}^{\nph, -}-\rho_{j-\frac{1}{2}}^{\nph,-})} \\& \spc = \frac{1}{2}\frac{\left(f(\rho_{\jph}^{\nph,-}, A^{\nph}_{\jmh})- f(\rho_{\jmh}^{\nph,-}, A^{\nph}_{\jmh})\right)}{(\rho_{\jph}^{\nph,-}- \rho_{\jmh}^{\nph,-})} + \frac{1}{2\lambda}\alpha\\
    &\spc = \frac{1}{2}\partial_{\rho}f(\bar{\rho}_{j}^{\nph,-},A^{\nph}_{\jmh})+ \frac{1}{2\lambda}\alpha,
\end{nalign} for some $\bar{\rho}_{j}^{\nph,-} \in \mathcal{I}(\rho_{\jmh}^{\nph,-},\rho_{\jph}^{\nph,-}).$
% Further, using  \eqref{eq:midtimevalues}, we write 
% \begin{nalign}\label{eq:midtimratio_psty}
%   \frac{\rho_{j+\frac{1}{2}}^{\nph, -}-\rho_{j-\frac{1}{2}}^{\nph,-}}{\rho_{j}^{n}-\rho_{j-1}^{n}}  &= \frac{\rho_{j+\frac{1}{2}}^{n, -}-\rho_{j-\frac{1}{2}}^{n,-}}{\rho_{j}^{n}-\rho_{j-1}^{n}  }   -\frac{\lambda}{2(\rho_{j}^{n}-\rho_{j-1}^{n})}\left(f(\rho_{\jph}^{n,-}, A_{\jph}^{n,-}- f(\rho_{\jmh}^{n,+}, A_{\jmh}^{n,+})\right)  \\& \spc +\frac{\lambda}{2(\rho_{j}^{n}-\rho_{j-1}^{n})}\left(f(\rho_{\jmh}^{n,-}, A_{\jmh}^{n,-})- f(\rho_{\jmtbt}^{n,+}, A_{\jmtbt}^{n,+})\right).
% \end{nalign}
Upon expanding the term $\displaystyle \frac{\rho_{j+\frac{1}{2}}^{\nph, -}-\rho_{j-\frac{1}{2}}^{\nph,-}}{\rho_{j}^{n}-\rho_{j-1}^{n}}$ using \eqref{eq:midtimevalues} and further using  \eqref{eq:Fdifratio_psty}, we can write
\begin{align}\label{eq:arefor_psty}
     a_{j-\frac{1}{2}}^{\nph} & =\bar{a}_{j-\frac{1}{2}}^{\nph} + \frac{c_{\jmh}^{n}}{2(\rho_{j}^{n}-\rho_{j-1}^{n})}{\left(\lambda\partial_{\rho}f(\bar{\rho}_{j}^{\nph,-},A^{\nph}_{\jmh})+ \alpha\right)} , 
\end{align}
where we define
\begin{nalign}\label{eq:abarandc}
 \bar{a}_{j-\frac{1}{2}}^{\nph} &:= \frac{1}{2}\left(\frac{\rho_{j+\frac{1}{2}}^{n, -}-\rho_{j-\frac{1}{2}}^{n,-}}{\rho_{j}^{n}-\rho_{j-1}^{n}  } \right) \left(\lambda\partial_{\rho}f(\bar{\rho}_{j}^{\nph,-},A^{\nph}_{\jmh})+ \alpha\right),
\\
    c_{\jmh}^{n} &\coloneqq  -\frac{\lambda}{2}\left(f(\rho_{\jph}^{n,-}, A_{\jph}^{n,-})- f(\rho_{\jmh}^{n,+}, A_{\jmh}^{n,+})\right)  +\frac{\lambda}{2}\left(f(\rho_{\jmh}^{n,-}, A_{\jmh}^{n,-})- f(\rho_{\jmtbt}^{n,+}, A_{\jmtbt}^{n,+})\right).
\end{nalign}
Now,  proceeding as in  \eqref{eq:midvalue_bd} to apply the mean value theorem and subsequently employing the estimate \eqref{eq:recon_bd}, the term $c_{\jmh}^{n}$ can be estimated as \begin{align}\label{eq:cbd}
    \abs{c_{\jmh}^{n}} &\leq {\lambda}\norm{\partial_{\rho}f}(1+\theta)(\rho_{j}^n+ \rho_{j-1}^{n}).
\end{align}
Through similar arguments,  the term $b_{\jph}^{\nph}$ in \eqref{eq:aandb} also can be written as \begin{nalign}\label{eq:b_refor_psty}
     b_{j+\frac{1}{2}}^{\nph} & =\bar{b}_{j+\frac{1}{2}}^{\nph}+ \frac{1}{2} \frac{d_{\jph}^{n}}{(\rho_{j+1}^{n}-\rho_{j}^{n})}\left(\alpha-\lambda\partial_{\rho}f(\bar{\rho}_{j}^{\nph,+},A^{\nph}_{\jph})\right),
\end{nalign}
where $\bar{\rho}_{j}^{\nph,+} \in \mathcal{I}({\rho}_{\jmh}^{\nph,+}, {\rho}_{\jph}^{\nph,+})$ and we define
\begin{align}
 \bar{b}_{j+\frac{1}{2}}^{\nph}&:= \frac{1}{2}\left(\frac{\rho_{j+\frac{1}{2}}^{n, +}-\rho_{j-\frac{1}{2}}^{n,+}}{\rho_{j+1}^{n}-\rho_{j}^{n}  } \right)\left(\alpha-\lambda\partial_{\rho}f(\bar{\rho}_{j}^{\nph,+},A^{\nph}_{\jph})\right),\\
    d_{\jph}^{n} &\coloneqq -\frac{\lambda}{2}\left(f(\rho_{j+\frac{3}{2}}^{n,-}, A_{j+\frac{3}{2}}^{n,-})- f(\rho_{\jph}^{n,+}, A_{\jph}^{n,+})\right)  +\frac{\lambda}{2}\left(f(\rho_{\jph}^{n,-}, A_{\jph}^{n,-})- f(\rho_{\jmh}^{n,+}, A_{\jmh}^{n,+})\right).
\end{align}
Furthermore, similar to \eqref{eq:cbd}, the term $d_{\jph}^{n}$ can be estimated as\begin{align}\label{eq:dbd}
    \abs{d_{\jph}^{n}} \leq {\lambda}\norm{\partial_{\rho}f}(1+\theta)(\rho_{j}^n+ \rho_{j+1}^{n}) .
\end{align}
In view of \eqref{eq:arefor_psty} and \eqref{eq:b_refor_psty}, the  expression \eqref{eq:positivityscheme} now reduces to
\begin{nalign}\label{eq:mh_refor}
      \rho_{j}^{n+1} 
 &= \rho_{j}^{n}(1-\bar{a}^{\nph}_{j-\frac{1}{2}}-\bar{b}^{\nph}_{j+\frac{1}{2}}) + \bar{a}^{\nph}_{j-\frac{1}{2}} \rho_{j-1}^{n}          + \bar{b}^{\nph}_{j+\frac{1}{2}}\rho^{n}_{j+1}\no\\& \spc - \frac{c_{\jmh}^{n}}{2}  \left(\lambda\partial_{\rho}f(\bar{\rho}_{j}^{\nph,-},A^{\nph}_{\jmh})+ \alpha\right)  + \frac{d_{\jph}^{n}}{2} \left(\alpha-\lambda\partial_{\rho}f(\bar{\rho}_{j}^{\nph,+},A^{\nph}_{\jph})\right)\\& \hspace{0.5cm}+ \lambda\left[F(\rho_{j+\frac{1}{2}}^{\nph,-}, \rho_{j-\frac{1}{2}}^{\nph, +}, A^{\nph}_{j-\frac{1}{2}}) 
 -F(\rho_{j+\frac{1}{2}}^{\nph,-}, \rho_{j-\frac{1}{2}}^{\nph,+}, A^{\nph}_{j+\frac{1}{2}})\right]. \no  
\end{nalign}
As a consequence of \eqref{eq:cbd} and \eqref{eq:dbd} and the CFL condition \eqref{eq:cfl_psty} ($\lambda \norm{\partial_\rho f} \leq \alpha$), we subsequently obtain
\begin{nalign}\label{eq:positivitylb}
     \rho_{j}^{n+1} 
 & \geq \left(1-\bar{a}^{\nph}_{j-\frac{1}{2}}-\bar{b}^{\nph}_{j+\frac{1}{2}}-\frac{1}{2}\lambda\norm{\partial_{\rho}f}(1+\theta){\left(\lambda\partial_{\rho}f(\bar{\rho}_{j}^{\nph,-},A^{\nph}_{\jmh})+ \alpha\right)}\right.\\ & \spc \spc \left. -\frac{1}{2}\lambda\norm{\partial_{\rho}f}(1+\theta){\left(\alpha-\lambda\partial_{\rho}f(\bar{\rho}_{j}^{\nph,+},A^{\nph}_{\jph})\right)}\right)\rho_{j}^{n}\\& \spc + \left(\bar{a}^{\nph}_{j-\frac{1}{2}}-\frac{1}{2}\lambda\norm{\partial_{\rho}f}(1+\theta)\left(\lambda\partial_{\rho}f(\bar{\rho}_{j}^{\nph,-},A^{\nph}_{\jmh})+ \alpha\right)\right) \rho_{j-1}^{n}   \\& \spc + \left(\bar{b}^{\nph}_{j+\frac{1}{2}}- \frac{1}{2}\lambda\norm{\partial_{\rho}f}(1+\theta) \left(\alpha-\lambda\partial_{\rho}f(\bar{\rho}_{j}^{\nph,+},A^{\nph}_{\jph})\right)\right)\rho^{n}_{j+1}\\& \hspace{0.5cm}- \lambda\left[F(\rho_{j+\frac{1}{2}}^{\nph,-}, \rho_{j-\frac{1}{2}}^{\nph, +}, A^{\nph}_{j-\frac{1}{2}}) 
 -F(\rho_{j+\frac{1}{2}}^{\nph,-}, \rho_{j-\frac{1}{2}}^{\nph,+}, A^{\nph}_{j+\frac{1}{2}})\right].    
\end{nalign}
Next, invoking Remark \ref{remark:ratioslopes} (equation \eqref{eq:ratiobd}), the CFL condition \eqref{eq:cfl_psty}  and the fact that $\theta \in [0,0.5],$ we obtain an estimate 
\begin{nalign}\label{eq:abar}
& \bar{a}^{\nph}_{j-\frac{1}{2}}-\frac{1}{2}\lambda\norm{\partial_{\rho}f}(1+\theta)\left(\lambda\partial_{\rho}f(\bar{\rho}_{j}^{\nph,-},A^{\nph}_{\jmh})+ \alpha\right)\\&\spc= \frac{1}{2}\left(\lambda\partial_{\rho}f(\bar{\rho}_{j}^{\nph,-},A^{\nph}_{\jmh})+ \alpha\right) \left(\left(\frac{\rho_{j+\frac{1}{2}}^{n, -}-\rho_{j-\frac{1}{2}}^{n,-}}{\rho_{j}^{n}-\rho_{j-1}^{n}  } \right) - \lambda\norm{\partial_{\rho}f}(1+\theta) \right)\\
& \spc \geq 0.
\end{nalign}
% where we have used the facts that $\displaystyle \lambda\partial_{\rho}f(\bar{\rho}_{j}^{\nph,-},A^{\nph}_{\jmh})+ \alpha \geq 0$ and $\displaystyle \lambda \norm{\partial_{\rho}f}(1+\theta) \leq \frac{2}{27}(1+\theta) \leq \frac{1}{9}$ by the CFL condition \eqref{eq:cfl_psty}
% Since $ \displaystyle \abs[\bigg]{\frac{\sigma_{j}^n-\sigma_{j-1}^n}{\rho_{j}^{n}-\rho_{j-1}^{n}}} \leq 2\theta$ by \eqref{eq:ratioslope},
% we have the bound
% \begin{align}\label{eq:recvaluesratio}
%   \frac{\rho_{j+\frac{1}{2}}^{n, -}-\rho_{j-\frac{1}{2}}^{n,-}}{\rho_{j}^{n}-\rho_{j-1}^{n}  }  &= 1 + \frac{1}{2}\frac{(\sigma_{j}^n-\sigma_{j-1}^n)}{(\rho_{j}^{n}-\rho_{j-1}^{n})} \geq 1 -\theta \geq \frac{1}{2}  \quad \quad \mbox{(since $\theta \in [0,0.5]$).}
% \end{align}

Similarly, we deduce that
\begin{align}\label{eq:bbarminus_lb}
    \bar{b}^{\nph}_{j+\frac{1}{2}}- \frac{1}{2}\lambda\norm{\partial_{\rho}f}(1+\theta) \left(\alpha-\lambda\partial_{\rho}f(\bar{\rho}_{j}^{\nph,+},A^{\nph}_{\jph})\right) \geq 0.
\end{align}
Furthermore, using the definition of $\bar{a}_{\jmh}^{\nph}$ from \eqref{eq:abarandc} and subsequently applying  Remark \ref{remark:ratioslopes} (equation \eqref{eq:ratiobd}), 
we obtain
\begin{nalign} \label{eq:abarplus_ub}
     &\bar{a}^{\nph}_{j-\frac{1}{2}}+\frac{1}{2}\lambda\norm{\partial_{\rho}f}(1+\theta)\left(\lambda\partial_{\rho}f(\bar{\rho}_{j}^{\nph,-},A^{\nph}_{\jmh})+ \alpha\right) \\ &\spc = \frac{1}{2} \left(\left(\frac{\rho_{j+\frac{1}{2}}^{n, -}-\rho_{j-\frac{1}{2}}^{n,-}}{\rho_{j}^{n}-\rho_{j-1}^{n}  } \right) + \lambda\norm{\partial_{\rho}f}(1+\theta) \right) \left(\lambda\partial_{\rho}f(\bar{\rho}_{j}^{\nph,-},A^{\nph}_{\jmh})+ \alpha\right)\\
     & \spc \leq \frac{1}{2} (1+\theta)\left( 1 + \lambda\norm{\partial_{\rho}f}\right)\left(\lambda\norm{\partial_{\rho}f}+ \alpha\right)   
      \leq  \frac{3}{4} \times \frac{29}{27} \times \frac{8}{27}  \leq \frac{1}{3},
\end{nalign}
where the last inequality is obtained from the CFL assumption \eqref{eq:cfl_psty}.
Analogously, we obtain
\begin{align}\label{eq:bbarplus_ub}
\bar{b}^{\nph}_{j+\frac{1}{2}}+\frac{1}{2} \lambda\norm{\partial_{\rho}f} \left(\alpha-\lambda\partial_{\rho}f(\bar{\rho}_{j}^{\nph,+},A^{\nph}_{\jph})\right) \leq \frac{1}{3}.
\end{align}
Now, simplifying the last term in \eqref{eq:positivitylb}, using hypothesis \eqref{hyp:H1} along similar lines to \eqref{eq:midvalue_bd}, and subsequently using \eqref{eq:recon_bd} and the CFL condition \eqref{eq:cfl_psty}, we obtain the estimate
\begin{nalign}\label{eq:Fdifbd}
    & \lambda \abs{F(\rho_{j+\frac{1}{2}}^{\nph,-}, \rho_{j-\frac{1}{2}}^{\nph, +}, A^{\nph}_{j-\frac{1}{2}}) 
 -F(\rho_{j+\frac{1}{2}}^{\nph,-}, \rho_{j-\frac{1}{2}}^{\nph,+}, A^{\nph}_{j+\frac{1}{2}})} \\ & \spc = \frac{\lambda}{2} \abs{f(\rho^{\nph,-}_{\jph}, A_{\jmh}^{\nph})+ f(\rho^{\nph,+}_{\jmh}, A_{\jmh}^{\nph}) - f(\rho^{\nph,-}_{\jph}, A_{\jph}^{\nph})- f(\rho^{\nph,+}_{\jmh}, A_{\jph}^{\nph})} \\& \spc \leq  \frac{\lambda}{2}\left( \abs{\partial_{\rho}f(\tilde{\rho}^{\nph,-}_{\jph}, A_{\jmh}^{\nph})}\abs{{\rho}^{\nph,-}_{\jph}}+ \abs{\partial_{\rho}f(\tilde{\rho}^{\nph,+}_{\jmh}, A_{\jmh}^{\nph})}\abs{\rho^{\nph,+}_{\jmh}} \right.\\ &\spc \spc \left.+ \abs{\partial_{\rho}f(\bar{\rho}^{\nph,-}_{\jph}, A_{\jph}^{\nph})}\abs{\rho^{\nph,-}_{\jph}}+ \abs{\partial_{\rho}f(\hat{\rho}^{\nph,+}_{\jmh}, A_{\jph}^{\nph})}\abs{{\rho}^{\nph,+}_{\jmh}} \right)\\ & \spc \leq 2(1+\theta)\lambda\norm{\partial_{\rho}f}(1+\lambda \norm{\partial_{\rho}f})\rho_{j}^n \leq  2\times \frac{3}{2}\times  \frac{2}{27} \times \frac{29}{27} \rho_{j}^n \leq   \frac{1}{3}\rho_{j}^n.
\end{nalign}
Finally, invoking the estimates \eqref{eq:abar},  \eqref{eq:bbarminus_lb}, \eqref{eq:abarplus_ub}, \eqref{eq:bbarplus_ub} and \eqref{eq:Fdifbd} in \eqref{eq:positivitylb}, we arrive at
\begin{nalign}
     \rho_{j}^{n+1} \geq 0.
\end{nalign}
This concludes the proof.
\end{proof}

\begin{theorem}\label{theorem:L1stability}($\mathrm{L}^1$-stability)
Let the initial datum $\rho_0 \in \mathrm{L}^{\infty}(\mathbb{R};\mathbb{R}_{+}).$  If the CFL condition \eqref{eq:cfl_psty} holds,  then the approximate solution  $\rho_{\Delta}$  computed using the proposed scheme \eqref{eq:mh} is $\mathrm{L}^1-$stable:
\begin{nalign}
 \norm{\rho_\Delta(t,\cdot)}_{\mathrm{L}^{1}(\mathbb{R})} & =   \norm{\rho_0}_{\mathrm{L}^{1}(\mathbb{R})} \quad \mbox{for all} \,\, t > 0.
\end{nalign}
\end{theorem}
\begin{proof}
From the definition of the numerical scheme \eqref{eq:mh} and using the positivity- preserving property (Theorem \ref{theorem:positivity}), for $t \in [t^n, t^{n+1}),$ $n \in \mathbb{N}\cup \{0\},$  we write \begin{nalign}
  \norm{\rho_\Delta(t,\cdot)}_{\mathrm{L}^{1}(\mathbb{R})}=  \norm{\rho_\Delta(t^n,\cdot)}_{\mathrm{L}^{1}(\mathbb{R})}  = \Delta x \sum_{\jinz}\abs{\rho_{j}^n} &= \Delta x \sum_{\jinz}\rho_{j}^n\\
  &=  \Delta x \sum_{\jinz}\rho_{j}^{n-1} -\Delta t\sum_{\jinz} F_{\jph}^{n-\half} + \Delta t\sum_{\jinz} F_{\jmh}^{n-\half}  \\
    & = \Delta x \sum_{\jinz}\rho_{j}^{n-1} = \cdots = \Delta x \sum_{\jinz}\rho_{0}^n = \norm{\rho_{0}}_{\mathrm{L}^{1}(\mathbb{R})}.
\end{nalign}
\end{proof}
\begin{theorem}\label{theorem:Linfty}($\mathrm{L^\infty}$- stability)
    Let the initial datum $\rho_0 \in \mathrm{L}^{\infty}(\mathbb{R};\mathbb{R}_{+}).$ If the CFL condition \eqref{eq:cfl_psty} holds, then there exists a constant $C$  such that the approximate solutions $\rho_\Delta$  computed using the scheme \eqref{eq:mh}  satisfy the $\mathrm{L}^{\infty}-$ estimate  \begin{nalign}\label{eq:Linfbd}
    \norm{\rho_\Delta(t,\cdot)} \leq C,
\end{nalign} for all $t \in [0,T].$  
\end{theorem}
\begin{proof}
Recall the formulation \eqref{eq:positivityscheme} of the scheme \eqref{eq:mh}:
\begin{align}\label{eq:linfscheme} 
       \rho^{n+1}_{j} &= \rho_{j}^{n} - a^{\nph}_{j-\frac{1}{2}} (\rho_{j}^{n}-\rho_{j-1}^{n})          + b^{\nph}_{j+\frac{1}{2}}(\rho^{n}_{j+1} - \rho^{n}_{j})\\
   & \hspace{0.5cm}+ \lambda\left[F(\rho_{j+\frac{1}{2}}^{\nph,-}, \rho_{j-\frac{1}{2}}^{\nph, +}, A^{\nph}_{j-\frac{1}{2}}) - 
 F(\rho_{j+\frac{1}{2}}^{\nph,-}, \rho_{j-\frac{1}{2}}^{\nph,+}, A^{\nph}_{j+\frac{1}{2}})\right]. \no  
\end{align}

First, using the definition \eqref{eq:midtimevalues}, adding and subtracting suitable terms and applying the mean  value theorem, we write
\begin{nalign}\label{eq:lmidvaluedif_ratio}
  \frac{\rho_{j+\frac{1}{2}}^{\nph, -}-\rho_{j-\frac{1}{2}}^{\nph,-}}{\rho_{j}^{n}-\rho_{j-1}^{n}  } 
  & = \frac{\rho_{j+\frac{1}{2}}^{n, -}-\rho_{j-\frac{1}{2}}^{n,-}}{\rho_{j}^{n}-\rho_{j-1}^{n}  } \left(1-\frac{\lambda}{2}\partial_{\rho}f(\bar{\rho}_{j}^{n,-}, A_{\jph}^{n,-})\right) \\& \spc + \frac{\rho_{\jmh}^{n,+}-\rho_{\jmtbt}^{n,+}}{\rho_{j}^{n}-\rho_{j-1}^{n}}\left(\frac{\lambda}{2}\partial_{\rho}f(\bar{\rho}_{j-1}^{n,+}, A_{\jmh}^{n,+})\right) \\& \spc  -\frac{\lambda}{2(\rho_{j}^{n}-\rho_{j-1}^{n})}\partial_{A}f(\rho^{n,-}_{\jmh}, \bar{A}^{n,-}_{j})\left(A_{\jph}^{n,-}- A_{\jmh}^{n,-}\right)\\ & \spc + \frac{\lambda}{2(\rho_{j}^{n}-\rho_{j-1}^{n})} \partial_{A}f(\rho^{n,+}_{\jmtbt}, \bar{A}^{n,+}_{j-1})\left(A_{\jmh}^{n,+}- A_{\jmtbt}^{n,+}\right),
\end{nalign}

where $\bar{\rho}_{j}^{n,-} \in \mathcal{I}({\rho}_{\jmh}^{n,-}, {\rho}_{\jph}^{n,-}), \bar{\rho}_{j-1}^{n,+} \in \mathcal{I}({\rho}_{\jmtbt}^{n,+}, \rho_{\jmh}^{n,+}), \bar{A}^{n,-}_{j} \in \mathcal{I}({A}^{n,-}_{\jmh}, {A}^{n,-}_{\jph})$ and $\bar{A}^{n,+}_{j-1} \in \mathcal{I}({A}^{n,+}_{\jmtbt}, {A}^{n,+}_{\jmh}),$ for $\jinz.$
Now using \eqref{eq:lmidvaluedif_ratio}, the term $ a_{j-\frac{1}{2}}^{\nph}$ in \eqref{eq:aandb} can be expressed as
\begin{align}\label{eq:a_refor_maxpri}
     a_{j-\frac{1}{2}}^{\nph} & =\tilde{a}_{j-\frac{1}{2}}^{\nph} + \frac{1}{2} \frac{\tilde{c}_{\jmh}^{n}}{(\rho_{j}^{n}-\rho_{j-1}^{n})}\left(\lambda\partial_{\rho}f(\bar{\rho}_{j}^{\nph,-},A^{\nph}_{\jmh})+ \alpha\right),
\end{align}
where 
\begin{align*}
 \tilde{a}_{j-\frac{1}{2}}^{\nph} &:= \frac{1}{2}\hat{a}_{j-\frac{1}{2}}^{\nph}\left(\lambda\partial_{\rho}f(\bar{\rho}_{j}^{\nph,-},A^{\nph}_{\jmh})+ \alpha\right),
\\
    \tilde{c}_{\jmh}^{n} &\coloneqq  -\frac{\lambda}{2}\left(\partial_{A}f(\rho^{n,-}_{\jmh}, \bar{A}^{n,-}_{j})\left(A_{\jph}^{n,-}- A_{\jmh}^{n,-}\right)- \partial_{A}f(\rho^{n,+}_{\jmtbt}, \bar{A}^{n,+}_{j-1})\left(A_{\jmh}^{n,+}- A_{\jmtbt}^{n,+}\right)\right),\\
    \hat{a}_{j-\frac{1}{2}}^{\nph}&:=  \left(\frac{\rho_{j+\frac{1}{2}}^{n, -}-\rho_{j-\frac{1}{2}}^{n,-}}{\rho_{j}^{n}-\rho_{j-1}^{n}  }\right) \left(1-\frac{\lambda}{2}\partial_{\rho}f(\bar{\rho}_{j}^{n,-}, A_{\jph}^{n,-})\right) + \left(\frac{\rho_{\jmh}^{n,+}-\rho_{\jmtbt}^{n,+}}{\rho_{j}^{n}-\rho_{j-1}^{n}}\right) \frac{\lambda}{2}\partial_{\rho}f(\bar{\rho}_{j-1}^{n,+}, A_{\jmh}^{n,+}). 
\end{align*}
Next, we show that the terms $\tilde{a}_{j-\frac{1}{2}}^{\nph}$ satisfy
\begin{align}\label{eq:tildea-bd}
     \quad 0 \leq \tilde{a}_{j-\frac{1}{2}}^{\nph} \leq \frac{1}{2}, \quad \mbox{for all} \,  \jinz.
\end{align}
First, using the estimate \eqref{eq:ratiobd}, the CFL condition \eqref{eq:cfl_psty} and the assumption that $\theta \in [0,0.5],$ we have 
\begin{nalign}\label{eq:atildelb}
     \hat{a}_{j-\frac{1}{2}}^{\nph} &  \geq \frac{1}{2} \left(1-\frac{\lambda}{2}\norm{\partial_{\rho}f}\right) - \frac{\lambda}{2}\norm{\partial_{\rho}f}(1+\theta)\\
   & = \frac{1}{2} - \lambda\norm{\partial_\rho f} \left(\frac{1}{4}+ \frac{1+\theta}{2}\right) \geq 0. 
\end{nalign}
The estimate \eqref{eq:atildelb} together with the CFL condition \eqref{eq:cfl_psty} yields
$\tilde{a}_{j-\frac{1}{2}}^{\nph} \geq 0.$
Further, using \eqref{eq:ratiobd} and the CFL condition \eqref{eq:cfl_psty}, we derive an upper bound
\begin{nalign}\label{eq:tildea_bd}
    \tilde{a}_{j-\frac{1}{2}}^{\nph} &\leq \frac{1}{2}(1+\theta) \left(1+{\lambda}\norm{\partial_{\rho}f}\right)\left(\lambda\norm{\partial_{\rho}f}+ \alpha\right)\\& \leq \frac{1}{2} \times\frac{3}{2} \times \frac{29}{27} \times \frac{8}{27}     \leq \frac{1}{2},
\end{nalign} thereby verifying \eqref{eq:tildea_bd}.
Next, Theorems \ref{theorem:positivity} and \ref{theorem:L1stability} imply that
\begin{nalign}
    \abs{A_{j}^n- A_{j-1}^n} &\leq \Delta x \sum_{l \in \mathbb{Z}} \abs{\mu_{j-l}- \mu_{j-1-l}}\rho_{l}^n \leq \Delta x \norm{\mu^{\prime}}\norm{\rho_{\Delta}(t^{n}, \cdot)}_{\mathrm{L}^{1}(\mathbb{R})} \leq \Delta x \norm{\mu^{\prime}}\norm{\rho_{0}}_{\mathrm{L}^{1}(\mathbb{R})}. 
\end{nalign}
Hence, from \eqref{eq:lrconv}, we obtain
\begin{nalign}
    \abs{s_{j}^n- s_{j-1}^n} &\leq 2\theta \Delta x \norm{\mu^{\prime}}\norm{\rho_{0}}_{\mathrm{L}^{1}(\mathbb{R})}. 
\end{nalign}
As a result, we arrive at
\begin{nalign}\label{eq:lrvaluediff_bd}
    \abs{A_{\jmh}^{n,+}- A_{\jmtbt}^{n,+}}, \abs{A_{\jph}^{n,-}- A_{\jmh}^{n,-}} &\leq \abs{A_{j}^n- A_{j-1}^n}+ \frac{1}{2}\abs{s_{j}^n- s_{j-1}^n} \\&\leq (1+\theta)\Delta x \norm{\mu^{\prime}}\norm{\rho_{0}}_{\mathrm{L}^{1}(\mathbb{R})},
\end{nalign} which together with hypothesis \eqref{hyp:H2} and the estimate \eqref{eq:recon_bda} yields 
\begin{align}\label{eq:tildec_bd}
    \abs{ \tilde{c}_{\jmh}^{n}} &\leq \lambda M(1+\theta)^{2}\Delta x \norm{\mu^{\prime}}\norm{\rho_{0}}_{\mathrm{L}^{1}(\mathbb{R})} \rho_{j-1}^n.
\end{align}

\cblue{In an analogous way, the term $b_{j+\frac{1}{2}}^{\nph}$ in \eqref{eq:aandb} can be expressed as 
\begin{nalign}\label{eq:b_refor_maxpri}
     b_{j+\frac{1}{2}}^{\nph} & =\tilde{b}_{j+\frac{1}{2}}^{\nph}+ \frac{1}{2} \frac{\tilde{d}_{\jph}^{n}}{(\rho_{j+1}^{n}-\rho_{j}^{n})}\left(\alpha-\lambda\partial_{\rho}f(\bar{\rho}_{j}^{\nph,+},A^{\nph}_{\jph})\right),
\end{nalign}
where 
\begin{nalign}
 \tilde{b}_{j+\frac{1}{2}}^{\nph}&:= \frac{1}{2}\hat{b}_{\jph}^{\nph}\left(\alpha-\lambda\partial_{\rho}f(\bar{\rho}_{j}^{\nph,+},A^{\nph}_{\jph})\right),
    \\ \tilde{d}_{\jph}^{n} &\coloneqq -\frac{\lambda}{2}\left(\partial_{A}f(\rho^{n,-}_{\jph}, \bar{A}^{n,-}_{j+1})\left(A_{j+\frac{3}{2}}^{n,-}- A_{\jph}^{n,-}\right)- \partial_{A}f(\rho^{n,+}_{\jmh}, \bar{A}^{n,+}_{j})\left(A_{\jph}^{n,+}- A_{\jmh}^{n,+}\right)\right),\\
    \hat{b}_{\jph}^{\nph} &:= \left(\frac{\rho_{j+\frac{1}{2}}^{n, +}-\rho_{j-\frac{1}{2}}^{n,+}}{\rho_{j+1}^{n}-\rho_{j}^{n}  } \right)(1+\frac{\lambda}{2}\partial_{\rho}f(\bar{\rho}_{j}^{n,+}, A_{\jph}^{n,+}))-\left(\frac{\rho_{j+\frac{3}{2}}^{n, -}-\rho_{j+\frac{1}{2}}^{n,-}}{\rho_{j+1}^{n}-\rho_{j}^{n}} \right) \frac{\lambda}{2}\partial_{\rho}f(\bar{\rho}_{j+1}^{n,-}, A_{j+\frac{3}{2}}^{n,-}),
\end{nalign}
where $\bar{\rho}_{j}^{n,-} \in \mathcal{I}({\rho}_{\jmh}^{n,-}, {\rho}_{\jph}^{n,-}), \bar{\rho}_{j-1}^{n,+} \in \mathcal{I}({\rho}_{\jmtbt}^{n,+}, \rho_{\jmh}^{n,+}), \bar{A}^{n,-}_{j} \in \mathcal{I}({A}^{n,-}_{\jmh}, {A}^{n,-}_{\jph})$ and $\bar{A}^{n,+}_{j-1} \in \mathcal{I}({A}^{n,+}_{\jmtbt}, {A}^{n,+}_{\jmh}),$ for $\jinz.$}
 Further, we obtain the following bounds:   
\begin{align}\label{eq:tildeb_bd}
 \quad 0 \leq \tilde{b}_{j+\frac{1}{2}}^{\nph} \leq \frac{1}{2},\end{align} and 
\begin{align}\label{eq:tilded_bd}
     \abs{ \tilde{d}_{\jph}^{n}} &\leq \lambda M(1+\theta)^{2}\Delta x \norm{\mu^{\prime}}\norm{\rho_{0}}_{\mathrm{L}^{1}(\mathbb{R})} \rho_{j}^n.
\end{align}
Using \eqref{eq:a_refor_maxpri} and \eqref{eq:b_refor_maxpri}, the scheme \eqref{eq:linfscheme} now reads as
\begin{align}\label{eq:mhstability}
      \rho_{j}^{n+1} 
 &= \rho_{j}^{n}(1-\tilde{a}^{\nph}_{j-\frac{1}{2}}-\tilde{b}^{\nph}_{j+\frac{1}{2}}) + \tilde{a}^{\nph}_{j-\frac{1}{2}} \rho_{j-1}^{n}          + \tilde{b}^{\nph}_{j+\frac{1}{2}}\rho^{n}_{j+1}\\& \spc - \frac{1}{2} \tilde{c}_{\jmh}^{n}\left(\lambda\partial_{\rho}f(\bar{\rho}_{j}^{\nph,-},A^{\nph}_{\jmh})+ \alpha\right)  + \frac{1}{2} \tilde{d}_{\jph}^{n}\left(\alpha-\lambda\partial_{\rho}f(\bar{\rho}_{j}^{\nph,+},A^{\nph}_{\jph})\right)\no\\& \hspace{0.5cm}+ \lambda\left[F(\rho_{j+\frac{1}{2}}^{\nph,-}, \rho_{j-\frac{1}{2}}^{\nph, +}, A^{\nph}_{j-\frac{1}{2}}) 
 -F(\rho_{j+\frac{1}{2}}^{\nph,-}, \rho_{j-\frac{1}{2}}^{\nph,+}, A^{\nph}_{j+\frac{1}{2}})\right]. \no  
\end{align}
Lemma \ref{lemma:midtimebd} (equation \eqref{eq:recon_bd}) and Theorems \ref{theorem:positivity} and \ref{theorem:L1stability} together yield \begin{nalign}\label{eq:midtimeconvdiff}
    \abs{A^{\nph}_{\jmh}-A^{\nph}_{\jph}} &\leq \frac{\Delta x}{2}\sum_{l \in \mathbb{Z}}\abs{\mu_{j+1-l}-\mu_{j-l}}\abs{\rho^{\nph,+}_{l-\half}} +  \frac{\Delta x}{2}\sum_{l \in \mathbb{Z}}\abs{\mu_{j-l}-\mu_{j-l-1}}\abs{\rho^{\nph,-}_{l+\half}}\\
    & \leq \frac{\Delta x^2}{2}\norm{\mu^{\prime}}\sum_{l \in \mathbb{Z}}(\abs{\rho^{\nph,+}_{l-\half}} +\abs{\rho^{\nph,-}_{l+\half}}) \\
    & \leq {\Delta x}\norm{\mu^{\prime}}(1+\theta)  (1+ {\lambda}\norm{\partial_{\rho}f})\norm{\rho_{0}}_{\mathrm{L}^{1}(\mathbb{R})}.
\end{nalign}
Furthermore, using the mean value theorem,  hypothesis \eqref{hyp:H2} and the estimates \eqref{eq:midtimeconvdiff}  and  \eqref{eq:recon_bd}, we obtain
\begin{nalign}\label{eq:Fdif_bound}
    &\lambda\abs[\big]{F(\rho_{j+\frac{1}{2}}^{\nph,-}, \rho_{j-\frac{1}{2}}^{\nph, +}, A^{\nph}_{j-\frac{1}{2}}) 
 -F(\rho_{j+\frac{1}{2}}^{\nph,-}, \rho_{j-\frac{1}{2}}^{\nph,+}, A^{\nph}_{j+\frac{1}{2}})} \\&\spc = \frac{\lambda}{2}\abs[\bigg]{\left(\partial_{A}f(\rho_{\jph}^{\nph,-},\bar{A}_{j}^{\nph})+\partial_{A}f(\rho_{\jmh}^{\nph,+},\hat{A}_{j}^{\nph})\right) (A^{\nph}_{\jmh}-A^{\nph}_{\jph})}\\ & \spc \leq \frac{\lambda}{2}M\left(\abs{\rho_{\jph}^{\nph,-}}+ \abs{\rho_{\jmh}^{\nph,+}}\right) \abs{A^{\nph}_{\jmh}-A^{\nph}_{\jph}}\\
 & \spc \leq \Delta t M (1+\theta)^{2}(1+\lambda \norm{\partial_{\rho}f})^{2}\norm{\mu^{\prime}}\norm{\rho_0}_{\mathrm{L}^1(\mathbb{R})}\norm{\rho_{\Delta}(t^n, \cdot)}.
\end{nalign}
Finally, combining the estimates \eqref{eq:tildea-bd}, \eqref{eq:tildeb_bd}, \eqref{eq:tildec_bd}, \eqref{eq:tilded_bd} and \eqref{eq:Fdif_bound} in \eqref{eq:mhstability}, we arrive at
\begin{nalign}
    \abs{\rho_{j}^{n+1}} &\leq \norm{\rho_{\Delta}(t^{n}, \cdot)}+ \Delta t(\alpha+\lambda\norm{\partial_{\rho}f}) M(1+\theta)^{2} \norm{\mu^{\prime}}\norm{\rho_{0}}_{\mathrm{L}^{1}(\mathbb{R})} \norm{\rho_{\Delta}(t^{n}, \cdot)}\\
    & \spc + \Delta t M (1+\theta)^{2}(1+\lambda\norm{\partial_{\rho}f})^{2}\norm{\mu^\prime}\norm{\rho_{0}}_{\mathrm{L}^1(\mathbb{R})}\norm{\rho_{\Delta}(t^n, \cdot)}\\
    & \leq \norm{\rho_{\Delta}(t^n, \cdot)}(1+\tilde{\mathcal{C}}\Delta t) \\
    & \leq (1+\tilde{\mathcal{C}}\Delta t)^{n+1}\norm{\rho_0} \leq \exp(\tilde{\mathcal{C}}(n+1)\Delta t)\norm{\rho_0},
\end{nalign} where $\tilde{\mathcal{C}}:=\left[\alpha+\lambda \norm{\partial_{\rho}f}+ (1+\lambda\norm{\partial_{\rho}f})^2 \right]M(1+\theta)^{2}\norm{\mu^{\prime}}\norm{\rho_{0}}_{\mathrm{L}^{1}(\mathbb{R})}.$
Now, choosing $ C := \exp(\tilde{\mathcal{C}}T)\norm{\rho_0},$ the result \eqref{eq:Linfbd} follows.
\end{proof}

\begin{theorem}\label{theorem:TVbound}(Total variation estimate)
If the initial datum $\rho_0 \in \mathrm{L}^{\infty} \cap \mathrm{BV}(\mathbb{R}; \mathbb{R}_{+})$ and the CFL condition \eqref{eq:cfl_psty} holds, then  the approximate solutions $\rho_{\Delta}$ computed using the scheme \eqref{eq:mh} satisfy
\begin{nalign}\label{eq:tvbestimate}
\mathrm{TV}(\rho_{\Delta}(t,\cdot)) 
    &\leq \exp(\mathcal{A}T)\left(  \sum_{\jinz}\abs{\rho^{0}_{j+1} -\rho^{0}_{j}} \right) + \mathcal{B}(\exp(\mathcal{A}T) - 1),
\end{nalign} for all $t \in [0,T]$ and for some constants $\mathcal{A}, \mathcal{B} > 0.$
\end{theorem}
\begin{proof}
From the scheme \eqref{eq:mh}, computing the difference $\rho^{n+1}_{j+1} -\rho^{n+1}_{j}$ and subsequently adding and subtracting suitable terms, we obtain
\begin{align}\label{eq:bvdiff}
    \rho^{n+1}_{j+1} -\rho^{n+1}_{j} 
&=C_{\jph}^n-\lambda D_{\jph}^n,
\end{align}
where
\begin{nalign}\label{eq:Cjphdefn}
C_{\jph}^n&:=(\rho^{n}_{j+1} -\rho^{n}_{j})-\lambda \left(F(\rho_{j+\frac{3}{2}}^{\nph,-}, \rho_{j+\frac{3}{2}}^{\nph, +}, A^{\nph}_{j+\frac{3}{2}}) - F(\rho_{j+\frac{1}{2}}^{\nph,-}, \rho_{j+\frac{1}{2}}^{\nph,+}, A^{\nph}_{j+\frac{3}{2}})\right) \\      
&\spc+\lambda \left(F(\rho_{j+\frac{1}{2}}^{\nph,-}, \rho_{j+\frac{1}{2}}^{\nph, +}, A^{\nph}_{j+\frac{1}{2}}) - F(\rho_{j-\frac{1}{2}}^{\nph,-}, \rho_{j-\frac{1}{2}}^{\nph,+}, A^{\nph}_{j+\frac{1}{2}}) \right), 
\end{nalign}
and 
\begin{nalign}\label{eq:D_defn}
D_{\jph}^n &:= \left( F(\rho_{j+\frac{1}{2}}^{\nph,-}, \rho_{j+\frac{1}{2}}^{\nph,+}, A^{\nph}_{j+\frac{3}{2}})-F(\rho_{j+\frac{1}{2}}^{\nph,-}, \rho_{j+\frac{1}{2}}^{\nph, +}, A^{\nph}_{j+\frac{1}{2}})\right) \\
&\spc - \left( F(\rho_{j-\frac{1}{2}}^{\nph,-}, \rho_{j-\frac{1}{2}}^{\nph,+}, A^{\nph}_{j+\frac{1}{2}})-F(\rho_{j-\frac{1}{2}}^{\nph,-}, \rho_{j-\frac{1}{2}}^{\nph, +}, A^{\nph}_{j-\frac{1}{2}})\right).
\end{nalign}
Now, we expand the term $C_{\jph}^{n}$ in \eqref{eq:Cjphdefn} by adding and subtracting appropriate terms as follows
\begin{nalign}\label{eq:Cnjph}
    C_{\jph}^{n} &:= (\rho^{n}_{j+1} -\rho^{n}_{j})&\\
&\hspace{.5cm}-\lambda \left(\frac{F(\rho_{j+\frac{3}{2}}^{\nph,-}, \rho_{j+\frac{3}{2}}^{\nph, +}, A^{\nph}_{j+\frac{3}{2}}) - F(\rho_{j+\frac{3}{2}}^{\nph,-}, \rho_{j+\frac{1}{2}}^{\nph, +}, A^{\nph}_{j+\frac{3}{2}})}{\rho_{j+\frac{3}{2}}^{\nph,+}-\rho_{j+\frac{1}{2}}^{\nph,+}} \right) \left( \frac{   \rho_{j+\frac{3}{2}}^{\nph,+}-\rho_{j+\frac{1}{2}}^{\nph,+}} {\rho_{j+2}^{n}-\rho_{j+1}^{n}} \right) (\rho_{j+2}^{n}-\rho_{j+1}^{n}) \\      
&\hspace{.5cm}-\lambda \left(\frac{F(\rho_{j+\frac{3}{2}}^{\nph,-}, \rho_{j+\frac{1}{2}}^{\nph, +}, A^{\nph}_{j+\frac{3}{2}}) - F(\rho_{j+\frac{1}{2}}^{n,-}, \rho_{j+\frac{1}{2}}^{\nph,+}, A^{\nph}_{j+\frac{3}{2}})}{ \rho_{j+\frac{3}{2}}^{\nph,-}-  \rho_{j+\frac{1}{2}}^{\nph,-}} \right) 
 \left( \frac{   \rho_{j+\frac{3}{2}}^{\nph,-}-\rho_{j+\frac{1}{2}}^{\nph,-}} {\rho_{j+1}^{n}-\rho_{j}^{n}} \right) (\rho_{j+1}^{n}-\rho_{j}^{n}) \\    
&\hspace{.5cm}+\lambda \left(\frac{F(\rho_{j+\frac{1}{2}}^{\nph,-}, \rho_{j+\frac{1}{2}}^{\nph, +}, A^{\nph}_{j+\frac{1}{2}}) - F(\rho_{j+\frac{1}{2}}^{\nph,-}, \rho_{j-\frac{1}{2}}^{\nph,+}, A^{\nph}_{j+\frac{1}{2}})}{\rho_{j+\frac{1}{2}}^{\nph,+}- \rho_{j-\frac{1}{2}}^{\nph,+}} \right) 
 \left( \frac{   \rho_{j+\frac{1}{2}}^{\nph,+}-\rho_{j-\frac{1}{2}}^{\nph,+}} {\rho_{j+1}^{n}-\rho_{j}^{n}} \right) (\rho_{j+1}^{n}-\rho_{j}^{n}) \\    
&\hspace{.5cm}+\lambda \left(\frac{F(\rho_{j+\frac{1}{2}}^{\nph,-}, \rho_{j-\frac{1}{2}}^{\nph, +}, A^{\nph}_{j+\frac{1}{2}}) - F(\rho_{j-\frac{1}{2}}^{\nph,-}, \rho_{j-\frac{1}{2}}^{\nph,+}, A^{\nph}_{j+\frac{1}{2}})}{ \rho_{j+\frac{1}{2}}^{\nph,-}-  \rho_{j-\frac{1}{2}}^{\nph,-}} \right) 
 \left( \frac{\rho_{j+\frac{1}{2}}^{\nph,-}-\rho_{j-\frac{1}{2}}^{\nph,-}} {\rho_{j}^{n}-\rho_{j-1}^{n}} \right) (\rho_{j}^{n}-\rho_{j-1}^{n}) \\    
&=(1-\ell_{j+\frac{1}{2}}^n-{k}_{j+\frac{1}{2}}^n)(\rho_{j+1}^{n}-\rho_{j}^{n}) +
{\ell}_{j+\frac{3}{2}}^n(\rho_{j+2}^{n}-\rho_{j+1}^{n}) +{k}_{j-\frac{1}{2}}^n (\rho_{j}^{n}-\rho_{j-1}^{n}), 
\end{nalign}
where \begin{nalign}\label{eq:kandl}
    k_{\jph}^n &:= \lambda \left(\frac{F(\rho_{j+\frac{3}{2}}^{\nph,-}, \rho_{j+\frac{1}{2}}^{\nph, +}, A^{\nph}_{j+\frac{3}{2}}) - F(\rho_{j+\frac{1}{2}}^{\nph,-}, \rho_{j+\frac{1}{2}}^{\nph,+}, A^{\nph}_{j+\frac{3}{2}})}{ \rho_{j+\frac{3}{2}}^{\nph,-}-  \rho_{j+\frac{1}{2}}^{\nph,-}} \right) 
 \left(\frac{\rho_{j+\frac{3}{2}}^{\nph,-}-\rho_{j+\frac{1}{2}}^{\nph,-}} {\rho_{j+1}^{n}-\rho_{j}^{n}} \right),\\ 
 \ell_{\jph}^n &:= -\lambda \left(\frac{F(\rho_{j+\frac{1}{2}}^{\nph,-}, \rho_{j+\frac{1}{2}}^{\nph, +}, A^{\nph}_{j+\frac{1}{2}}) - F(\rho_{j+\frac{1}{2}}^{\nph,-}, \rho_{j-\frac{1}{2}}^{\nph,+}, A^{\nph}_{j+\frac{1}{2}})}{\rho_{j+\frac{1}{2}}^{\nph,+}- \rho_{j-\frac{1}{2}}^{\nph,+}} \right) 
 \left( \frac{   \rho_{j+\frac{1}{2}}^{\nph,+}-\rho_{j-\frac{1}{2}}^{\nph,+}} {\rho_{j+1}^{n}-\rho_{j}^{n}} \right). 
\end{nalign}
Starting from \eqref{eq:midtimevalues}, we consider the difference $\rho_{j+\frac{3}{2}}^{\nph,-}-\rho_{j+\frac{1}{2}}^{\nph,-}.$ Then we rearrange the terms, add and subtract suitable terms and finally apply the mean value theorem to obtain
\begin{nalign}\label{eq:leftmidvldif}
    \rho_{j+\frac{3}{2}}^{\nph,-}-\rho_{j+\frac{1}{2}}^{\nph,-} &= \left(\rho_{j+\frac{3}{2}}^{n,-}-\rho_{j+\frac{1}{2}}^{n,-}\right) - \frac{\lambda}{2}\left(f(\rho_{j+\frac{3}{2}}^{n,-}, A_{j+\frac{3}{2}}^{n,-})-f(\rho_{\jph}^{n,-}, A_{\jph}^{n,-})\right)\\
    & \spc + \frac{\lambda}{2}\left(f(\rho_{j+\frac{1}{2}}^{n,+}, A_{j+\frac{1}{2}}^{n,+})-f(\rho_{\jmh}^{n,+}, A_{\jmh}^{n,+})\right)\\
    & = \left(\rho_{j+\frac{3}{2}}^{n,-}-\rho_{j+\frac{1}{2}}^{n,-}\right)\left(1-\frac{\lambda}{2}\partial_{\rho}f(\bar{\rho}_{j+1}^{n,-}, A_{j+\frac{3}{2}}^{n,-})\right) \\
    & \spc + \left(\rho_{\jph}^{n,+}- \rho_{\jmh}^{n,+}\right)\left(\frac{\lambda}{2}\partial_{\rho}f(\bar{\rho}^{n,+}_{j}, A_{\jph}^{n,+})\right)\\ & \spc - \frac{\lambda}{2}\partial_{A}f(\rho_{\jph}^{n,-}, \bar{A}_{j+1}^{n,-})  \left(A_{j+\frac{3}{2}}^{n,-}-A_{\jph}^{n,-}\right) + \frac{\lambda}{2}\partial_{A}f(\rho_{\jmh}^{n,+}, \bar{A}_{j}^{n,+})  \left(A_{\jph}^{n,+}-A_{\jmh}^{n,+}\right),
\end{nalign}
where $\bar{\rho}_{j}^{n,-} \in \mathcal{I}({\rho}_{\jmh}^{n,-},$ ${\rho}_{\jph}^{n,-}), \bar{\rho}_{j-1}^{n,+} \in \mathcal{I}({\rho}_{\jmtbt}^{n,+},$ $\rho_{\jmh}^{n,+}), \bar{A}^{n,-}_{j} \in \mathcal{I}({A}^{n,-}_{\jmh}, {A}^{n,-}_{\jph})$ and $\bar{A}^{n,+}_{j-1} \in \mathcal{I}({A}^{n,+}_{\jmtbt}, {A}^{n,+}_{\jmh}),$ for $\jinz.$
Noting that 
\begin{align*}
    \lambda \left(\frac{F(\rho_{j+\frac{3}{2}}^{\nph,-}, \rho_{j+\frac{1}{2}}^{\nph, +}, A^{\nph}_{j+\frac{3}{2}}) - F(\rho_{j+\frac{1}{2}}^{\nph,-}, \rho_{j+\frac{1}{2}}^{\nph,+}, A^{\nph}_{j+\frac{3}{2}})}{ \rho_{j+\frac{3}{2}}^{\nph,-}-  \rho_{j+\frac{1}{2}}^{\nph,-}} \right) &= \frac{1}{2}\left(\lambda\partial_{\rho}f(\bar{\rho}_{\jpo}^{\nph,-},A_{j+\frac{3}{2}}^{\nph})+\alpha\right),
\end{align*}
where $\bar{\rho}_{\jpo}^{\nph,-} \in \mathcal{I}({\rho}_{j+\frac{1}{2}}^{\nph,-},{\rho}_{j+\frac{3}{2}}^{\nph,-})$
and in view of \eqref{eq:leftmidvldif}, we rewrite the term $k_{\jph}^n$ in \eqref{eq:kandl} as
\begin{nalign}\label{eq:k_rw}
    k_{\jph}^n &= \tilde{k}_{\jph}^n+ \frac{ \hat{k}_{\jph}^n}{\rho_{j+1}^{n}-\rho_{j}^{n}}, 
\end{nalign}
where 
\begin{nalign}\label{eq:tiildehatkdefn}
 \tilde{k}_{\jph}^n &:= \frac{1}{2}\left(\lambda\partial_{\rho}f(\bar{\rho}_{\jpo}^{\nph,-},A_{j+\frac{3}{2}}^{\nph})+\alpha\right)
 \left(\left(\frac{\rho_{j+\frac{3}{2}}^{n,-}-\rho_{j+\frac{1}{2}}^{n,-}}{\rho_{j+1}^{n}-\rho_{j}^{n}}\right)\left(1-\frac{\lambda}{2}\partial_{\rho}f(\bar{\rho}_{j+1}^{n,-}, A_{j+\frac{3}{2}}^{n,-})\right)\right.\\ & \left.\spc \spc  + \left(\frac{\rho_{\jph}^{n,+}- \rho_{\jmh}^{n,+}}{\rho_{j+1}^{n}-\rho_{j}^{n}}\right)\left(\frac{\lambda}{2}\partial_{\rho}f(\bar{\rho}^{n,+}_{j}, A_{\jph}^{n,+})\right)\right), 
\\
    \hat{k}_{\jph}^n &:= \frac{1}{2}\left(\lambda\partial_{\rho}f(\bar{\rho}_{\jpo}^{\nph,-},A_{j+\frac{3}{2}}^{\nph})+\alpha\right) \left(- \frac{\lambda}{2}\partial_{A}f(\rho_{\jph}^{n,-}, \bar{A}_{j+1}^{n,-})  \left(A_{j+\frac{3}{2}}^{n,-}-A_{\jph}^{n,-}\right) \right.\\ & \spc \spc \left. + \frac{\lambda}{2}\partial_{A}f(\rho_{\jmh}^{n,+}, \bar{A}_{j}^{n,+})  \left(A_{\jph}^{n,+}-A_{\jmh}^{n,+}\right)\right).
\end{nalign}
Analogously, we can write $\ell_{\jph}^n$  \eqref{eq:kandl} as \begin{nalign}\label{eq:l_rw}
    \ell_{\jph}^n &= \tilde{\ell}_{\jph}^n+ \frac{ \hat{\ell}_{\jph}^n}{\rho_{j+1}^{n}-\rho_{j}^{n}},
\end{nalign}
where 
\begin{nalign}\label{eq:ltildehat}
    \tilde{\ell}_{\jph}^n &:= \frac{1}{2}\left(\alpha-\lambda\partial_{\rho}f(\bar{\rho}_{j}^{\nph,+},A_{j+\frac{1}{2}}^{\nph})\right)
 \left(\left(\frac{\rho_{j+\frac{3}{2}}^{n,-}-\rho_{j+\frac{1}{2}}^{n,-}}{\rho_{j+1}^{n}-\rho_{j}^{n}}\right)\left(-\frac{\lambda}{2}\partial_{\rho}f(\bar{\rho}_{j+1}^{n,-}, A_{j+\frac{3}{2}}^{n,-})\right)\right.\\ & \left.\spc \spc  + \left(\frac{\rho_{\jph}^{n,+}- \rho_{\jmh}^{n,+}}{\rho_{j+1}^{n}-\rho_{j}^{n}}\right)\left(1+\frac{\lambda}{2}\partial_{\rho}f(\bar{\rho}^{n,+}_{j}, A_{\jph}^{n,+})\right)\right), 
\\
    \hat{\ell}_{\jph}^n &:= \frac{1}{2}\left(\alpha-\lambda\partial_{\rho}f(\bar{\rho}_{j}^{\nph,+},A_{j+\frac{1}{2}}^{\nph})\right) \left(- \frac{\lambda}{2}\partial_{A}f(\rho_{\jph}^{n,-}, \bar{A}_{j+1}^{n,-})  \left(A_{j+\frac{3}{2}}^{n,-}-A_{\jph}^{n,-}\right) \right.\\ & \spc \spc \left. + \frac{\lambda}{2}\partial_{A}f(\rho_{\jmh}^{n,+}, \bar{A}_{j}^{n,+})  \left(A_{\jph}^{n,+}-A_{\jmh}^{n,+}\right)\right).
\end{nalign}
\par
In the light of  \eqref{eq:k_rw} and \eqref{eq:l_rw}, the term $C_{\jph}^{n}$ in \eqref{eq:Cjphdefn} can be expressed as
\begin{align}\label{eq:c-refor}
     C_{\jph}^{n} =  \tilde{C}_{\jph}^{n} +  \hat{C}_{\jph}^{n},
\end{align}
where
\begin{nalign}\label{eq:tildehatCs}
     \tilde{C}_{\jph}^{n} &:=(1-\tilde{\ell}_{j+\frac{1}{2}}^n-\tilde{k}_{j+\frac{1}{2}}^n)(\rho_{j+1}^{n}-\rho_{j}^{n}) +
\tilde{\ell}_{j+\frac{3}{2}}^n(\rho_{j+2}^{n}-\rho_{j+1}^{n}) +\tilde{k}_{j-\frac{1}{2}}^n (\rho_{j}^{n}-\rho_{j-1}^{n}),\\  \hat{C}_{\jph}^{n}& :=  -\hat{\ell}_{j+\frac{1}{2}}^n-\hat{k}_{j+\frac{1}{2}}^n +
\hat{\ell}_{j+\frac{3}{2}}^n +\hat{k}_{j-\frac{1}{2}}^n. \end{nalign}
From \eqref{eq:bvdiff} and \eqref{eq:c-refor}, it follows that
\begin{nalign}\label{eq:tvb_initalform}
\sum_{\jinz}\abs{\rho^{n+1}_{j+1} -\rho^{n+1}_{j}} &\leq  \sum_{\jinz}\abs{\tilde{C}_{\jph}^n}+\sum_{\jinz}\abs{\hat{C}_{\jph}^n}+ \lambda\sum_{\jinz} \abs{D_{\jph}^n}.
\end{nalign}
The terms on the right-hand side of \eqref{eq:tvb_initalform} are estimated  (derived in Appendix \ref{section:tvbappendix}) as follows \begin{nalign}\label{eq:tvb_component_bds}
    \sum_{j \in \mathbb{Z}} \abs{\tilde{C}_{\jph}^{n}} &\leq \sum_{j \in \mathbb{Z}} \abs{\rho_{j+1}^n- \rho_{j}^n},\\
\sum_\jinz \abs{\hat{C}_{\jph}^{n}} 
    & \leq 4\mathcal{K}_{1}\Delta t + 4\mathcal{K}_{2} \Delta t \sum_{\jinz}\abs{\rho^n_{j+1} - \rho_{j}^n},\\
    \lambda\sum_{\jinz} \abs{D_{\jph}^n} & \leq \Delta t \mathcal{K}_{7} + \Delta t \mathcal{K}_{8}\sum_{\jinz}\abs{\rho_{j+1}^n-\rho_{j}^n},
\end{nalign}
where $\mathcal{K}_{1},  \mathcal{K}_{2}, \mathcal{K}_{7},  \mathcal{K}_{8} $ are constants independent of $\Delta x$ defined in \eqref{eq:K1K2} and \eqref{eq:K7K8}.
Now, invoking the estimates  \eqref{eq:tvb_component_bds} in \eqref{eq:tvb_initalform} yields
\begin{nalign}\label{eq:tvb_update_bd}
\sum_{\jinz}\abs{\rho^{n+1}_{j+1} -\rho^{n+1}_{j}}
& \leq \Delta t \mathcal{K}_{9} + (1+\mathcal{K}_{10}\Delta t )\sum_{\jinz}\abs{\rho_{j+1}^n-\rho_{j}^n},
\end{nalign}
where
\begin{align*}
    \mathcal{K}_{9} &:= 4\mathcal{K}_{1}+ \mathcal{K}_{7},\\
    \mathcal{K}_{10} &:= 4\mathcal{K}_{2}+ \mathcal{K}_{8}. 
\end{align*}
 Starting from \eqref{eq:tvb_update_bd} and proceeding recursively, we obtain
\begin{nalign}\label{eq:tv_exp}
    \sum_{\jinz}\abs{\rho^{n+1}_{j+1} -\rho^{n+1}_{j}}
    &\leq (1+\mathcal{K}_{10}\Delta t)^{n+1} \left(\sum_{\jinz}\abs{\rho^{0}_{j+1} -\rho^{0}_{j}} \right) + \mathcal{N}(n),
\end{nalign}
where $\mathcal{N}(n):={\mathcal{K}_9}\Delta t \left( 1+ (1+\mathcal{K}_{10}\Delta t) + \dots + (1+\mathcal{K}_{10}\Delta t)^{n-1} + (1+\mathcal{K}_{10}\Delta t)^{n} \right).$

Noting that $(1+\mathcal{K}_{10}\Delta t)^{n+1} \leq \exp(\mathcal{K}_{10}T)$ for $(n+1)\Delta t \leq T,$ the term $\mathcal{N}(n)$ in \eqref{eq:tv_exp} can be simplified as
\begin{nalign}\label{eq:kbound}
   \mathcal{N}(n)
    &= \frac{\mathcal{K}_9}{\mathcal{K}_{10}}\left((1+\mathcal{K}_{10}\Delta t)-1\right) \left[ 1+ (1+\mathcal{K}_{10}\Delta t) + \dots + (1+\mathcal{K}_{10}\Delta t)^{n-1} + (1+\mathcal{K}_{10}\Delta t)^{n}  \right]\\
    &= \frac{\mathcal{K}_9}{\mathcal{K}_{10}} \left[(1+\mathcal{K}_{10}\Delta t) + \dots+  (1+\mathcal{K}_{10}\Delta t)^{n}+  (1+\mathcal{K}_{10}\Delta t)^{n+1}   \right. \\ & \spc \spc \spc \left.  - 1 - (1+\mathcal{K}_{10}\Delta t)   \dots- (1+\mathcal{K}_{10}\Delta t)^{n}  \right]\\
    &= \frac{\mathcal{K}_{9}}{\mathcal{K}_{10}} \left((1+\mathcal{K}_{10}\Delta t)^{n+1}-1\right) \leq \frac{\mathcal{K}_9}{\mathcal{K}_{10}}(\exp(\mathcal{K}_{10}T) - 1).
\end{nalign}
Finally, in light of  \eqref{eq:kbound} from \eqref{eq:tv_exp}, we deduce \begin{nalign}\label{eq:tv_exp_b}
\mathrm{TV}(\rho_{\Delta}^{n+1}) =  \sum_{\jinz}\abs{\rho^{n+1}_{j+1} -\rho^{n+1}_{j}} 
    &\leq \exp(\mathcal{K}_{10}T)\left(  \sum_{\jinz}\abs{\rho^{0}_{j+1} -\rho^{0}_{j}} \right) + \frac{\mathcal{K}_{9}}{\mathcal{K}_{10}}(\exp(\mathcal{K}_{10}T) - 1),
\end{nalign} for $(n+1)\Delta t \leq T.$ Now, upon choosing $\mathcal{A}:=\mathcal{K}_{10}$ and $\displaystyle \mathcal{B}:=\frac{\mathcal{K}_{9}}{\mathcal{K}_{10}},$ the result \eqref{eq:tvbestimate} is immediate.
\end{proof}

\begin{theorem}\label{theorem:L1timecty}($\mathrm{L}^1$-Lipschitz continuity in time)
  Let the initial datum $\rho_0 \in \mathrm{L}^{\infty} \cap \mathrm{BV}(\mathbb{R};\mathbb{R}_{+}).$ If the CFL condition  \eqref{eq:cfl_psty} holds, then there exists a constant $\kappa$ such that the approximate solutions  generated by the proposed scheme \eqref{eq:mh} satisfy
    \begin{nalign}\label{eq:timecty}
\norm{\rho_{\Delta}(t_1,\cdot)-\rho_{\Delta}(t_2,\cdot)}_{\mathrm{L}^{1}(\mathbb{R})} & \leq \kappa (\abs{t_1 - t_2}+\Delta t),
\end{nalign}
for any $t_{1}, t_{2} \in [0,T].$
\end{theorem}
\begin{proof}

Recalling the scheme \eqref{eq:mh} written in the form \eqref{eq:mhstability}:
\begin{nalign}
    \rho_{j}^{n+1} 
 &= \rho_{j}^{n}(1-\tilde{a}^{\nph}_{j-\frac{1}{2}}-\tilde{b}^{\nph}_{j+\frac{1}{2}}) + \tilde{a}^{\nph}_{j-\frac{1}{2}} \rho_{j-1}^{n}          + \tilde{b}^{\nph}_{j+\frac{1}{2}}\rho^{n}_{j+1}\\& \spc - \frac{1}{2} \tilde{c}_{\jmh}^{n}\left(\lambda\partial_{\rho}f(\bar{\rho}_{j}^{\nph,-},A^{\nph}_{\jmh})+ \alpha\right)  + \frac{1}{2} \tilde{d}_{\jph}^{n}\left(\alpha-\lambda\partial_{\rho}f(\bar{\rho}_{j}^{\nph,+},A^{\nph}_{\jph})\right)\no\\& \hspace{0.5cm}+ \lambda\left[F(\rho_{j+\frac{1}{2}}^{\nph,-}, \rho_{j-\frac{1}{2}}^{\nph, +}, A^{\nph}_{j-\frac{1}{2}}) 
 -F(\rho_{j+\frac{1}{2}}^{\nph,-}, \rho_{j-\frac{1}{2}}^{\nph,+}, A^{\nph}_{j+\frac{1}{2}})\right], \no 
\end{nalign}
we write
\begin{nalign}\label{eq:timecty_initial}
    \norm{\rho_{\Delta}(t^{n+1},\cdot)-\rho_{\Delta}(t^{n},\cdot)} _{\mathrm{L}^{1}(\mathbb{R})} &= \Delta x \sum_{\jinz}\abs{\rho_{j+1}^{n+1}-\rho_{j}^{n}} \\ &\leq  \Delta x \sum_{\jinz}\abs{\rho_{j-1}^{n}-\rho_{j}^{n}}\abs{\tilde{a}^{\nph}_{j-\frac{1}{2}}}+ \Delta x \sum_{\jinz}\abs{\rho_{j+1}^{n}-\rho_{j}^{n}}\abs{\tilde{b}^{\nph}_{j+\frac{1}{2}}} \\& \spc + \frac{1}{2}\Delta x  \left(\lambda\norm{\partial_{\rho}f}+ \alpha\right)\sum_{\jinz} \left(\abs{\tilde{c}_{\jmh}^{n}}+\abs{\tilde{d}_{\jph}^{n}} \right) \\& \spc +\Delta t \sum_{\jinz} \abs[\Big]{F(\rho_{j+\frac{1}{2}}^{\nph,-}, \rho_{j-\frac{1}{2}}^{\nph, +}, A^{\nph}_{j-\frac{1}{2}}) 
 -F(\rho_{j+\frac{1}{2}}^{\nph,-}, \rho_{j-\frac{1}{2}}^{\nph,+}, A^{\nph}_{j+\frac{1}{2}})}.
\end{nalign}

In order to estimate the last term in \eqref{eq:timecty_initial}, we apply the mean value theorem and subsequently use the estimate  \eqref{eq:midtimeconvdiff}, hypothesis \eqref{hyp:H2} and Lemma \ref{lemma:midtimebd} to yield 
\begin{nalign}\label{eq:Fdifbd_timecty}   &\sum_{\jinz}\abs{F(\rho_{j+\frac{1}{2}}^{\nph,-}, \rho_{j-\frac{1}{2}}^{\nph, +}, A^{\nph}_{j-\frac{1}{2}}) 
 -F(\rho_{j+\frac{1}{2}}^{\nph,-}, \rho_{j-\frac{1}{2}}^{\nph,+}, A^{\nph}_{j+\frac{1}{2}})}\\%%%%%%%%%%%%%%%%%%%%%%%%%%%%%%%%%%%%%%%%%%%%%%%%%%%%%%%%%%%%%%%%%%%%%%%%%%%%%
 & \spc \leq \frac{1}{2}\sum_{\jinz}\left(\abs[\big]{\partial_{A}f(\rho_{\jph}^{\nph,-}, \bar{A}_{j}^{\nph})} + \abs[\Big]{\partial_{A}f(\rho_{\jmh}^{\nph,+}, \tilde{A}_{j}^{\nph})} \right) \abs[big]{A_{\jmh}^{\nph}- A_{\jph}^{\nph}}\\& \spc \leq M\norm{\mu^{\prime}}(1+\theta)^2  (1+ {\lambda}\norm{\partial_{\rho}f})^{2}\norm{\rho_{0}}_{\mathrm{L}^{1}(\mathbb{R})} {\Delta x} \sum_{\jinz}\rho_{j}^n  \\
 &\spc \leq M\norm{\mu^{\prime}}(1+\theta)^2  (1+ {\lambda}\norm{\partial_{\rho}f})^{2}\norm{\rho_{0}}_{\mathrm{L}^{1}(\mathbb{R})}^{2}.  
\end{nalign}

Now, invoking the estimates \eqref{eq:tildea-bd},  \eqref{eq:tildeb_bd}, \eqref{eq:tildec_bd}, \eqref{eq:tilded_bd} (derived in the proof of Theorem \ref{theorem:Linfty}) and \eqref{eq:Fdifbd_timecty}, the $\mathrm{L}^1$ distance in \eqref{eq:timecty_initial} can be estimated as
\begin{nalign}\label{eq:timecty_1step}
      \norm{\rho_{\Delta}(t^{n+1},\cdot)-\rho_{\Delta}(t^{n},\cdot)} _{\mathrm{L}^{1}(\mathbb{R})}
  & \leq \Delta t \left(\frac{1}{\lambda} \sum_{\jinz}\abs{\rho_{j+1}^n-\rho_{j}^n} + \mathcal{J}\right)\\
& \leq  \Delta t \kappa,
\end{nalign}
for $(n+1)\Delta t \leq T,$  where  
\begin{align*}
\mathcal{J}&:=  M\norm{\mu^{\prime}}(1+\theta)^2  \norm{\rho_{0}}_{\mathrm{L}^{1}(\mathbb{R})}^{2}(\alpha+ {\lambda}\norm{\partial_{\rho}f}+(1+{\lambda}\norm{\partial_{\rho}f})^2), \\ \kappa &:=  \frac{1}{\lambda} \left(\exp(\mathcal{A}T)\mathrm{TV}(\rho_{0}) + \mathcal{B}(\exp(\mathcal{A}T) - 1)\right) + \mathcal{J},
    \end{align*} and  the last inequality follows from Theorem \ref{theorem:TVbound}.
The result \eqref{eq:timecty}  is now an immediate consequence of \eqref{eq:timecty_1step}.

% Further, for any $t_1, t_2 \in [0,T]$ with $t_1 > t_2,$ let $m_{1}, m_{2} \in \mathbb{N}$ be such that $t_1 \in [t^{m_{1}}, t^{m_{1}+1})$ and $t_2 \in [t^{m_{2}}, t^{m_{2}+1}).$ Then, employing  \eqref{eq:timecty_1step}, we obtain the desired estimate:  
% \begin{nalign}
%     \norm{\rho_{\Delta}(t_1,\cdot)-\rho_{\Delta}(t_2,\cdot)}_{\mathrm{L}^{1}(\mathbb{R})} & = \norm{\rho_{\Delta}(t^{m_1},\cdot)-\rho_{\Delta}(t^{m_2},\cdot)} _{\mathrm{L}^{1}(\mathbb{R})}\\
%     & \leq \norm{\rho_{\Delta}(t^{m_1},\cdot)-\rho_{\Delta}(t^{{m_1}-1},\cdot)} _{\mathrm{L}^{1}(\mathbb{R})} + \norm{\rho_{\Delta}(t^{{m_1}-1},\cdot)-\rho_{\Delta}(t^{{m_1}-2},\cdot)} _{\mathrm{L}^{1}(\mathbb{R})} \\
%     & \spc  + \cdots + \norm{\rho_{\Delta}(t^{{m_2}+1},\cdot)-\rho_{\Delta}(t^{{m_2}},\cdot)} _{\mathrm{L}^{1}(\mathbb{R})} \\
%     &\leq (m_1- m_2) \Delta t \kappa_T \leq (\abs{t_1 - t_2}+ \Delta t)\kappa_T.
% \end{nalign}
% This completes the proof. 
\end{proof}

\section{Convergence of the numerical scheme}\label{section:entropy}
In this section, we show that the numerical scheme \eqref{eq:mh} converges to the unique entropy solution of the problem \eqref{eq:problem}. To begin with, we recall a result originally established in \cite{vila1988}, and later adapted to the case of non-local conservation laws in \cite{gowda2023} (Theorem 5.1).

\begin{theorem}\label{theorem:entconver} Suppose that a numerical scheme that approximates \eqref{eq:problem} can be written in the form:
\begin{equation}
\rho_j^{n+1}=\tilde{ \rho}_{j}^{n+1} - e_{\jph}^{n+1}+e_{\jmh}^{n+1},
\label{second order}
\end{equation}
 where\\
 (i) $\tilde{ \rho}_{j}^{n+1}$ is computed from $\rho_j^n$ using a scheme which yields a  sequence  of approximate solutions converging in $\mathrm{L}^1_{\mathrm{loc}}$ to the entropy solution of \eqref{eq:problem}.\\
 (ii) $|e^{n+1}_{\jph}| \leq K{\Delta x}^\delta$ for  $(n+1)\Delta t \leq T,$ where  $K>0$  and $\delta \in (0,1)$ are some constants independent of $\Delta x.$ \\
 (iii) The approximate solutions $ \rho_{\Delta}$ obtained using \eqref{second order} are  in $\mathrm{L}^{\infty},$ $\mathrm{BV},$  and admits $\mathrm{L}^{1}$- Lipschitz continuity in time.\\
Then the approximate solutions generated by the scheme (\ref{second order}) converges in $\mathrm{L}^1_{\mathrm{loc}}$  to the entropy solution of \eqref{eq:problem}. 
\end{theorem}
\cblue{The proof of this result follows along similar lines to that of Theorem 5.1 from \cite{gowda2023}.}
Now our goal is to write the scheme \eqref{eq:mh} in the form \eqref{second order} and show that it satisfies all the hypotheses of Theorem \ref{theorem:entconver}. To this end, we first observe that the scheme \eqref{eq:mh} can be written in the form
\begin{nalign}\label{eq:mhpredcorr}
    \tilde{\rho}^{n+1}_{j} &= \rho^{n}_{j} - \lambda\left[F(\rho_{j}^{n},\rho_{j+1}^{n}, A^{n}_{\jph}) - F(\rho_{j-1}^{n}, \rho_{j}^{n}, A^{n}_{\jmh})\right],\\
    \rho^{n+1}_{j}&=\tilde{\rho}^{n+1}_{j}- e_{\jph}^{n+1} +  e_{\jmh}^{n+1}, 
\end{nalign}
where $ \displaystyle A^{n}_{\jph} $ is as defined in \eqref{eq:lxf} and \begin{nalign}\label{eq:ejph}
    e_{\jph}^{n+1} &:= \lambda \left( F(\rho_{j+\frac{1}{2}}^{n+\frac{1}{2},-}, \rho_{j+\frac{1}{2}}^{n+\frac{1}{2}, +}, A^{\nph}_{\jph})- F(\rho_{j}^{n},\rho_{j+1}^{n}, A^{n}_{\jph})\right).
    \end{nalign}
Additionally, we modify the scheme \eqref{eq:mhpredcorr} by redefining the slopes \eqref{slope-1} as follows 
\begin{nalign}\label{eq:mod_slope}
    \sigma^{n}_{j}& = 2 \theta \textrm{minmod}\left((\rho^{n}_{j}-\rho^{n}_{j-1}), \  \frac{1}{2}(\rho^{n}_{j+1}-\rho^{n}_{j-1}), \  (\rho^{n}_{j+1}-\rho^{n}_{j}), \mathrm{sgn}(\rho^{n}_{j+1}-\rho^{n}_{j}) \mathcal{K}(\Delta x)^{\delta} \right),
\end{nalign} for some $\mathcal{K}>0$ and some $\delta \in (0,1).$

\begin{lemma}\label{lemma:corrtermbd}
    There exists a constant $K>0$ such that  the correction terms $\{e_\jph^{n+1}\}_{\jinz}$ in the scheme \eqref{eq:mhpredcorr} with the modified slopes \eqref{eq:mod_slope} satisfy 
    \begin{align}\label{eq:corrtermbd}
        \abs{e_\jph^{n+1}} & \leq K(\Delta x)^{\delta} \quad \mbox{for} \,\, \jinz \,\, \mbox{and} \,\, (n+1) \Delta t \leq T.
    \end{align}
\end{lemma}
\begin{proof}
We expand $e_{\jph}^{n+1}$, add and subtract suitable terms and subsequently use the mean value theorem to write 
\begin{nalign}
    e_{\jph}^{n+1} 
    & =\frac{\lambda}{2}\left(\partial_{\rho}f(\bar{\rho}_{j+\frac{1}{2}}^{n+\frac{1}{2},-},A^{\nph}_{\jph})(\rho_{j+\frac{1}{2}}^{n+\frac{1}{2},-} -\rho_{j}^{n}) + \partial_{A}f(\rho_{j}^{n}, \bar{A}^{n}_{\jph})(A^{\nph}_{\jph}- A^{n}_{\jph}) \right) \\ & \spc  + \frac{\lambda}{2}\left(\partial_{\rho}f(\bar{\rho}_{j+\frac{1}{2}}^{n+\frac{1}{2},+},A^{\nph}_{\jph})({\rho}_{j+\frac{1}{2}}^{n+\frac{1}{2},+}-\rho_{j+1}^n) + \partial_{A}f(\rho_{j+1}^{n},\tilde{A}^{n}_{\jph})({A}^{\nph}_{\jph}- {A}^{n}_{\jph}) \right)\\
    & \spc -\frac{1}{2}\alpha  \left(\rho_{j+\frac{1}{2}}^{n+\frac{1}{2},+}-\rho_{j+1}^n\right) + \frac{1}{2}\alpha \left(\rho_{j+\frac{1}{2}}^{n+\frac{1}{2},-}-\rho_{j}^n\right) \\
    & = \left(\frac{\lambda}{2}\partial_{\rho}f(\bar{\rho}_{j+\frac{1}{2}}^{n+\frac{1}{2},-},A^{\nph}_{\jph}) + \frac{1}{2}\alpha\right)(\rho_{j+\frac{1}{2}}^{n+\frac{1}{2},-} -\rho_{j}^{n}) \\ & \spc + \left( \frac{\lambda}{2}\partial_{\rho}f(\bar{\rho}_{j+\frac{1}{2}}^{n+\frac{1}{2},+},A^{\nph}_{\jph}) -\frac{1}{2}\alpha \right) ({\rho}_{j+\frac{1}{2}}^{n+\frac{1}{2},+}-\rho_{j+1}^n)\\
    & \spc + \left(\partial_{A}f(\rho_{j}^{n}, \bar{A}^{n}_{\jph}) + \partial_{A}f(\rho_{j+1}^{n},\tilde{A}^{n}_{\jph}) \right)(A^{\nph}_{\jph}- A^{n}_{\jph}),
\end{nalign}
where $\bar{\rho}_{j+\frac{1}{2}}^{n+\frac{1}{2},-} \in \mathcal{I}(\rho_{j}^n, \rho_{j+\frac{1}{2}}^{n+\frac{1}{2},-}),$ $\bar{\rho}_{j+\frac{1}{2}}^{n+\frac{1}{2},+} \in \mathcal{I}(\rho_{j+1}^n, \rho_{j+\frac{1}{2}}^{n+\frac{1}{2},+}),$ and $\bar{A}^{n}_{\jph}, \tilde{A}^{n}_{\jph} \in \mathcal{I}({A}^{n}_{\jph}, {A}^{\nph}_{\jph}).$ In order to obtain a bound on $e^{n+1}_{\jph},$ it is sufficient to estimate the terms $ \abs{\rho_{j+\frac{1}{2}}^{n+\frac{1}{2},-} -\rho_{j}^{n}},$ $\abs{\rho_{j-\frac{1}{2}}^{n+\frac{1}{2},+} -\rho_{j}^{n}}$  and $\abs{A^{\nph}_{\jph}- A^{n}_{\jph}}.$ To this end, we first note that
\begin{align}\label{eq:sjbd}
    \abs{A^{n,-}_{j+\frac{1}{2}}-A^{n,+}_{j-\frac{1}{2}}} =\abs{s_j^n} &= \theta\abs{A_{j+1}^n - A_{j-1}^n} \leq \theta \Delta x \norm{\mu^\prime}\norm{\rho_{0}}_{\mathrm{L}^1{\mathbb{R}}}.
\end{align}
Now, expanding $\rho_{j+\frac{1}{2}}^{n+\frac{1}{2},-}$ and $\rho_{j-\frac{1}{2}}^{n+\frac{1}{2},+}$ from \eqref{eq:midtimevalues}, adding and subtracting appropriate terms, subsequently applying the mean value theorem and hypothesis \eqref{hyp:H2}, and finally applying Lemma \ref{lemma:reconvaluebd}, Theorem \ref{theorem:Linfty} and the estimate \eqref{eq:sjbd}, we obtain
\begin{nalign}\label{eq:ent_bd_a}
    \abs{\rho_{j+\frac{1}{2}}^{n+\frac{1}{2},-} -\rho_{j}^{n}} & \leq \frac{1}{2}\abs{\sigma_{j}^n}+\frac{\lambda}{2} \abs[\big]{ (f(\rho_{j+\frac{1}{2}}^{n,-},A^{n,-}_{j+\frac{1}{2}})-f(\rho_{j-\frac{1}{2}}^{n,+},A^{n,+}_{j-\frac{1}{2}})}\\
    & \leq \frac{1}{2}\abs{\sigma_{j}^n}+\frac{\lambda}{2} \abs[\big]{ (\partial_{\rho}f(\bar{\rho}_{j}^{n},A^{n,-}_{j+\frac{1}{2}})(\rho_{j+\frac{1}{2}}^{n,-}-\rho_{j-\frac{1}{2}}^{n,+})}\\ & \spc +\frac{\lambda}{2}\abs[\big]{\partial_{A}f(\rho_{j-\frac{1}{2}}^{n,+},\bar{A}^{n}_{j})(A^{n,-}_{j+\frac{1}{2}}-A^{n,+}_{j-\frac{1}{2}})}\\
    & \leq \left(\frac{1}{2}+ \frac{\lambda}{2}  \norm{\partial_{\rho}f}\right)\abs{\sigma_{j}^n} +\frac{\lambda}{2}M\abs{\rho_{j-\frac{1}{2}}^{n,+}}\abs[\big]{A^{n,-}_{j+\frac{1}{2}}-A^{n,+}_{j-\frac{1}{2}}}\\
    & \leq  \left(\frac{1}{2}+ \frac{\lambda}{2}  \norm{\partial_{\rho}f}\right)\mathcal{K}(\Delta x)^{\delta} + \left(\frac{\lambda}{2}M C(1+\theta)\theta  \norm{\mu^\prime}\norm{\rho_{0}}_{\mathrm{L}^1(\mathbb{R})}\right)\Delta x\\
    & \leq \tilde{\mathcal{K}} (\Delta x)^{\delta},
    \end{nalign}
    where $\bar{A}^{n}_{j} \in \mathcal{I}(A^{n,-}_{j+\frac{1}{2}}, A^{n,+}_{j-\frac{1}{2}})$ and $\bar{\rho}_{j}^n \in \mathcal{I}(\rho_{j+\frac{1}{2}}^{n,-}, \rho_{j-\frac{1}{2}}^{n,+})$ and $\displaystyle \tilde{\mathcal{K}}:=\left(\frac{1}{2}+ \frac{\lambda}{2}  \norm{\partial_{\rho}f}\right)\mathcal{K}+ \frac{\lambda}{2}M C(1+\theta)\theta  \norm{\mu^\prime}\norm{\rho_{0}}_{\mathrm{L}^1(\mathbb{R})}.$ Identically, we obtain    \begin{nalign}\label{eq:entropy_est2}
        \abs{\rho_{j-\frac{1}{2}}^{n+\frac{1}{2},+} -\rho_{j}^{n}} &\leq \tilde{\mathcal{K}} (\Delta x)^{\delta}. 
    \end{nalign}
    Further, we write
    \begin{nalign}\label{eq:ent_bd_b}
        \abs{A^{\nph}_{\jph}- A^{n}_{\jph}} & \leq \frac{\Delta x}{2} \sum_{\jinz} \left[\abs{\mu_{j+1-l}}\abs{\rho^{\nph,+}_{l-\frac{1}{2}}-\rho_{l}^n} + \abs{\mu_{j-l}}\abs{\rho^{\nph,-}_{l+\frac{1}{2}}-\rho_{l}^n} \right]  \\
        & \leq \tilde{\mathcal{K}} (\Delta x)^\delta \norm{\mu}L_{\mu},
    \end{nalign}
where $L_{\mu}$ denotes the length of a compact interval which contains the support of the  convolution kernel $\mu.$ Now, invoking the estimates \eqref{eq:ent_bd_a}, \eqref{eq:entropy_est2} and \eqref{eq:ent_bd_b} in \eqref{eq:ejph}, we obtain the desired result \eqref{eq:corrtermbd}
with $\displaystyle K= \left(\lambda\norm{\partial_{\rho}f} + \alpha + 2M \norm{\mu}\cblue{L_{\mu}} \right)\tilde{\mathcal{K}}.$  
\end{proof}
With the necessary estimates at hand, we are now in a position to use Theorem \ref{theorem:entropyconv} to establish the convergence of the numerical scheme to the entropy solution. We present this result in the following theorem.
\begin{theorem}\label{theorem:entropyconv}
Let $\rho_0 \in \mathrm{L}^{\infty} \cap \mathrm{BV}(\mathbb{R};\mathbb{R}_{+}).$ If the CFL condition \eqref{eq:cfl_psty} holds, then the approximate solutions $\rho_{\Delta}$ generated by the scheme \eqref{eq:mh} with the modified slopes \eqref{eq:mod_slope}, converge to the unique entropy solution of the problem \eqref{eq:problem}.    
\end{theorem}
\begin{proof}
    It is sufficient to verify that the scheme \eqref{eq:mh}, with modified slopes \eqref{eq:mod_slope}, satisfies the hypotheses of Theorem \ref{theorem:entconver}. The formulation \eqref{eq:mhpredcorr} shows that the scheme can be expressed in the form \eqref{second order}. Further,  $\tilde{\rho}_{j}^{n+1}$ in \eqref{eq:mhpredcorr} is computed using the first-order Lax-Friedrichs scheme \eqref{eq:lxf}, whose convergence to the entropy solution has been established in \cite{betancourt2011, amorim2015, aggarwal2015}. This verifies hypothesis (i) of Theorem \ref{theorem:entconver}. Hypothesis (ii) holds true due to the result in  Lemma \ref{lemma:corrtermbd}. Finally, we observe that the results in Theorems \ref{theorem:Linfty},  \ref{theorem:TVbound} and  \ref{theorem:L1timecty} hold true for the scheme \eqref{eq:mh} together with the modified slopes \eqref{eq:mod_slope}, as well. This confirms the validity of condition (iii). Hence, the result is proved.
    \end{proof}

\begin{remark}
     While implementing the scheme, the slope modification \eqref{eq:mod_slope} is not really needed. This is because, for any given mesh-size $\Delta x,$ we can choose a sufficiently large constant $\mathcal{K}>0,$ so that the modified slope \eqref{eq:mod_slope} reduces to \eqref{slope-1}. Specifically, for mesh sizes $\Delta x \geq \epsilon$ for some fixed $\epsilon > 0,$ we can choose $\mathcal{K}= 2C\epsilon^{-\delta},$ where $C$ is as in Theorem \ref{theorem:Linfty}.  This has also been observed in  \cite{leroux1981} (p. 158), \cite{vila1988} (p. 68) and \cite{viallon1991} (p. 577).
\end{remark}

\begin{remark}
Note that the slope modification \eqref{eq:mod_slope} is introduced solely to ensure the entropy convergence of the scheme.  Even without the modification \eqref{eq:mod_slope}, the scheme \eqref{eq:mh} can be independently shown to converge to a weak solution \eqref{eq:weaksoln} of the problem \eqref{eq:problem}: Theorems \ref{theorem:Linfty}, \ref{theorem:TVbound} and \ref{theorem:L1timecty} allow us to apply Kolmogorov's compactness theorem (see Theorem 4.1 of \cite{gowda2023}), which guarantees the existence of a subsequence $\Delta_k \to 0$ and a function $\rho \in C([0,T]; \mathrm{L}^{1}_{\mathrm{loc}}(\mathbb{R}))$ such that $\rho_{\Delta_k}$ converges to $\rho$ in $C([0,T]; \mathrm{L}^{1}_{\mathrm{loc}}(\mathbb{R}))$. Further, using a Lax-Wendroff-type argument (see the proof of Theorem 4.2 in \cite{gowda2023}), we can show that $\rho$ is a weak solution of  the problem \eqref{eq:problem}. 
% , thus establishing the convergence of a subsequence of approximate solutions to a weak solution. 
\end{remark}

\section{Numerical experiments}\label{section:numerics}

In this section, we present numerical experiments to illustrate the performance of the proposed scheme~\eqref{eq:mh}, in comparison with the first-order Lax--Friedrichs (FO) scheme ~\eqref{eq:lxf}. To further demonstrate the advantage of this particular second-order MH-type scheme, we also compare it with a standard second-order MUSCL-Runge-Kutta (RK-2) scheme  (see Section \ref{section:rkscheme}). For all schemes, we choose the time step to satisfy the CFL condition~\eqref{eq:cfl_psty} corresponding to the MH scheme, with the coefficient $\alpha$ set to $0.16$. We consider a uniform discretization of the spatial domain $I = [x_{l}, x_{r}]$ into $M$ cells of size $\Delta x = (x_r - x_l)/M$, and denote the cell averages at time $t^n$ by $\{\rho_j^n\}_{j=1}^M$.
Let $[a, b]$ be the smallest compact interval containing the support of the measure $\mu$, and define $N := (b - a)/\Delta x$.

To implement the boundary conditions, we introduce ghost cells on either side of the domain by defining the values $\rho_{-N+1}^n,\rho_{-N+2}^n, \dots, \rho_0^n,$ and $\rho_{M+1}^n, \rho_{M+2}^n, \dots, \rho_{M+N}^n,$ on the left and right of the domain, respectively. We consider two types of boundary conditions:
\begin{enumerate}
    \item {Periodic boundary conditions:}
\begin{align*}
    \rho_{M+i}^n &:= \rho_{i}^n, \quad \quad \quad \,\, \text{for }   i = 1, \dots, N,\\
    \rho_{-i}^n &:= \rho_{M - i + 1}^n, \quad \text{for } i = 0, \dots, N-1.
\end{align*}\\
\item {Absorbing boundary conditions:}
\begin{align*}
    \rho_{M+i}^n &:= \rho_{M}^n, \quad \quad \text{for } i = 1, \dots, N,\\
    \rho_{-i}^n &:= \rho_{1}^n, \quad \quad \,\, \text{for } i = 0, \dots, N-1.
\end{align*}

\end{enumerate}

% \begin{figure}
%     \centering
%     \includegraphics[width=0.5\linewidth]{figs/amorim1/amorim_0.pdf}
%     \caption{Example 1: initial datum as in \eqref{eq:icamorim}.}
%     \label{fig:enter-label}
% \end{figure}
% \begin{figure}
%     \centering
%     \includegraphics[width=0.5\linewidth]{figs/amorim1.pdf}
%     \caption{Example 1: Solutions at time $t=2.5,$ computed with a mesh of size $\Delta x = 6/200.$ Reference solutions computed using $\Delta x = \frac{6}{1800}.$}
%     \label{fig:amorim1}
% \end{figure}

\begin{example}\label{example:2}(Smooth solution case)  As studied in \cite{amorim2015}, we consider an example   of the problem \eqref{eq:problem}, where  the flux  function is given by\begin{align}\label{eq:ex1flux}
        f(\rho, A)= \rho(1-\rho)(1-A)
    \end{align} and the convolution kernel is chosen as
\begin{align}\label{eq:amorimkernel}
    \mu(x) &:= \frac{1}{\mathcal{M}}\left((x-a)(b-x)\right)^{\frac{5}{2}} \chi_{[a,b]}(x), \quad x \in \mathbb{R},
\end{align}
where $\mathcal{M}:= \int_{a}^{b}\left((x-a)(b-x)\right)^{\frac{5}{2}}\dif x$ and $[a,b] \subset \mathbb{R}.$ To evaluate the experimental order of accuracy, we consider the smooth initial condition:
\begin{nalign}\label{eq:icsmooth}
    \rho_{0}(x)= 0.5 + 0.4 \sin(\pi x)
\end{nalign}
in the domain $[-1,1]$ together with three sets of $[a,b]$ given by: $[0.0, 0.25],$ $[-0.125, 0.125]$ and $[-0.25, 0.0]$  corresponding to the  upstream, centered and downstream convolutions, respectively. We evolve the solutions imposing periodic boundary conditions up to time $t = 0.15$, for mesh sizes $\Delta x\in \{0.2, 0.1, 0.05, 0.025, 0.0125\}.$ The MH solution corresponding to the fine mesh size $\Delta x = 2/640$ is taken as the reference solution which we denote by $\rho_{ref}$. The experimental order of accuracy (E.O.A.) is then computed using the formula
\begin{nalign}
    \Theta(\Delta x)= \log_{2}\left(\frac{\norm{{\rho}_{2\Delta x} - {\rho}_{ref}}_{L^{1}}}{\norm{{\rho}_{\Delta x} - {\rho}_{ref}}_{L^{1}}}\right).
\end{nalign}
The results presented in Table \ref{table:ordermh} show that the MH scheme attains the expected order of accuracy. Further, the corresponding $\mathrm{L}^{1}$ error versus CPU time plots for the MH and RK-2 schemes are given in Figure \ref{fig:cpu_time_smooth}. While Table \ref{table:ordermh} indicates that the the RK-2 scheme attains a similar E.O.A. to that of the MH scheme, the results in Figure \ref{fig:cpu_time_smooth} show that the MH scheme is  computationally more efficient. 

\begin{table}[h!]
\centering
\begin{tabular}{|c|l|l|l|l|l|l|l|}
\hline
Kernel support $[a,b]$& & \multicolumn{2}{|c|}{FO}& \multicolumn{2}{|c|}{MH}&\multicolumn{2}{|c|}{RK-2} \\ \hline
& $\Delta x$ & $\mathrm{L}^1-$ error& $\Theta(\Delta x)$ & $\mathrm{L}^1-$ error & $\Theta(\Delta x)$ & $\mathrm{L}^1-$ error & $\Theta(\Delta x)$  \\ \cline{2-8}
%%%%%%%%%%%%%%%%%%%%%%%%%%%%%%%%%%%%%
\multirow{5}{*}{\begin{tabular}{@{}c@{}}$[0.0,0.25]$\\ (upstream convolution)\end{tabular}} & 0.2 &  0.206012   &  - &  0.082694   &   -   & 0.082462   &    -      \\ \cline{2-8}               
& 0.1 & 0.115288     &  0.837490 & 0.027777 &    1.573868  &   0.027355 &  1.591917  \\ \cline{2-8}       
& 0.05 &  0.063176 &  0.867773 &  0.008861  &    1.648333 &  0.008904   &     1.619287 \\ \cline{2-8}     
& 0.025 & 0.033041 &   0.935116  & 0.002471 &  1.842107 &  0.002530 &   1.815118 \\ \cline{2-8} 
& 0.0125 &  0.017026 & 0.956479 & 0.000644 &    1.939505 & 0.000657& 1.943746\\
\hline

%%%%%%%%%%%%%%%%%%%%%%%%%%%%%%%%%%%%%

\multirow{5}{*}{\begin{tabular}{@{}c@{}}$[-0.125, 0.125]$\\ (centered convolution)\end{tabular}} & 0.2 & 0.191868   &    -   & 0.071610  &   -  &  0.071878   &     -    \\ \cline{2-8}               
& 0.1 & 0.108584 & 0.821301 & 0.027806   & 1.364778  &  0.027328  & 1.395156  \\ \cline{2-8}       
& 0.05 &  0.058011 & 0.904406 &  0.008895 &     1.644184 & 0.008844 & 1.627536 \\ \cline{2-8}     
& 0.025 &  0.030147 &  0.944303  & 0.002527 &    1.815667 & 0.002544  & 1.797536 \\ \cline{2-8} 
& 0.0125 & 0.015438  &  0.965463 &0.000655  &  1.946496 & 0.000661  &  1.944145\\
\hline
%%%%%%%%%%%%%%%%%%%%%%%%%%%%%%%%%%%%%%%%%
%%%%%%%%%%%%%%%%%%%%%%%%%%%%%%%%%%%%%%%%%%%%

\multirow{5}{*}{\begin{tabular}{@{}c@{}}$[-0.25, 0.0]$\\ (downstream convolution)\end{tabular}} & 0.2 &  0.179211    &   -  &  0.070700     &  -    & 0.070413   &   - \\ \cline{2-8}               
& 0.1 & 0.102061& 0.812219 &  0.026216   &  1.431257  &  0.025736 &   1.452034  \\ \cline{2-8}       
& 0.05 & 0.054620 & 0.901929 &  0.008712  &   1.589371 & 0.008626   &  1.576991 \\ \cline{2-8}     
& 0.025 & 0.028217  &   0.952864  & 0.002590&  1.749549 & 0.002600  &  1.729973 \\ \cline{2-8} 
& 0.0125 &  0.014402  & 0.970304 & 0.000696  & 1.894568 &  0.000699 &   1.893973\\
\hline
%%%%%%%%%%%%%%%%%%%%%%%%%%%%%%%%%%%%%%%%%%%%%
%%%%%%%%%%%%%%%%%%%%%%%%%%%%%%%%%%%%%%%%%%%%%
\end{tabular}
\vspace{0.3cm} 
\caption{Example \ref{example:2}. $\mathrm{L}^{1}-$ errors and E.O.A. obtained for FO, MH and RK-2 schemes, with $\Delta t  = \Delta x/20,$ computed at time $t = 0.15.$}
\label{table:ordermh}
\end{table}

%%%%%%%%%%%%%%%%%%%%%%%%%%%%%%%%%
%%%%%%%%%%%%%%%%%%%%%%%%%%%%%%%%%
%%CPU TIME:

\begin{figure}[!tbp]
  \centering
  \begin{subfigure}[b]{0.32\textwidth}
    \includegraphics[width=\textwidth]{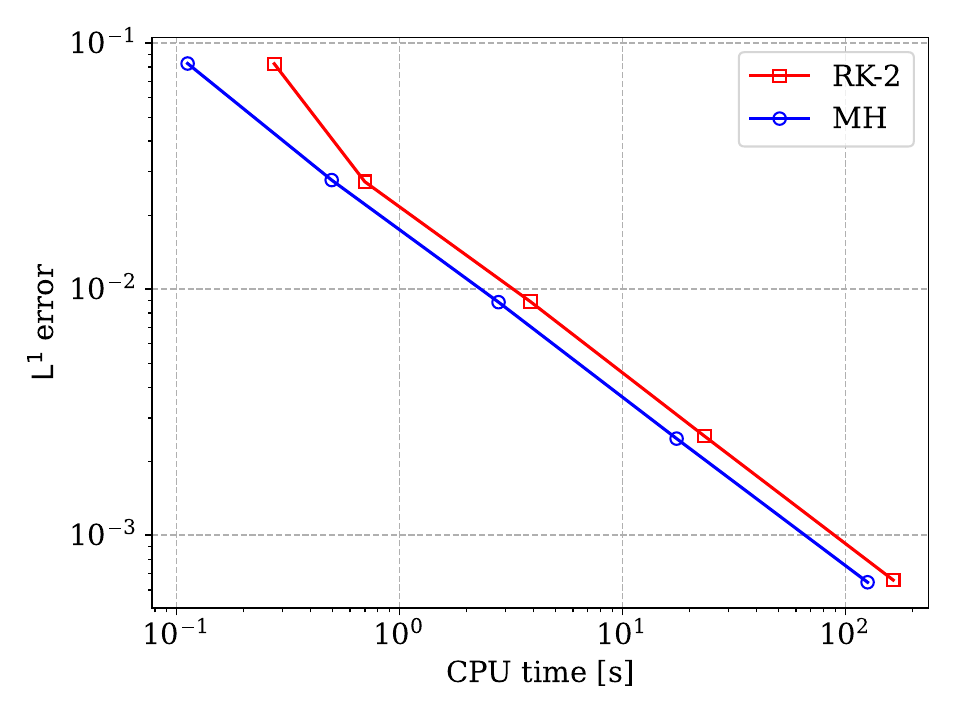}
    \caption*{(a)}
  \end{subfigure}
  \hfill
  \begin{subfigure}[b]{0.32\textwidth}
    \includegraphics[width=\textwidth]{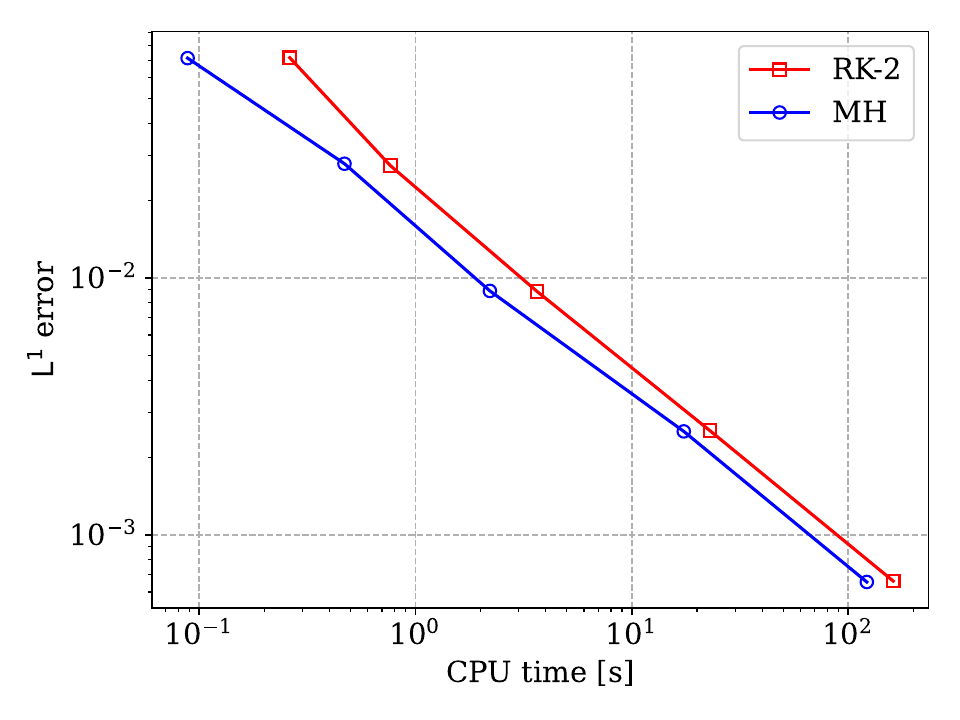}
    \caption*{(b)}
  \end{subfigure}
  \hfill
  \begin{subfigure}[b]{0.32\textwidth}
    \includegraphics[width=\textwidth]{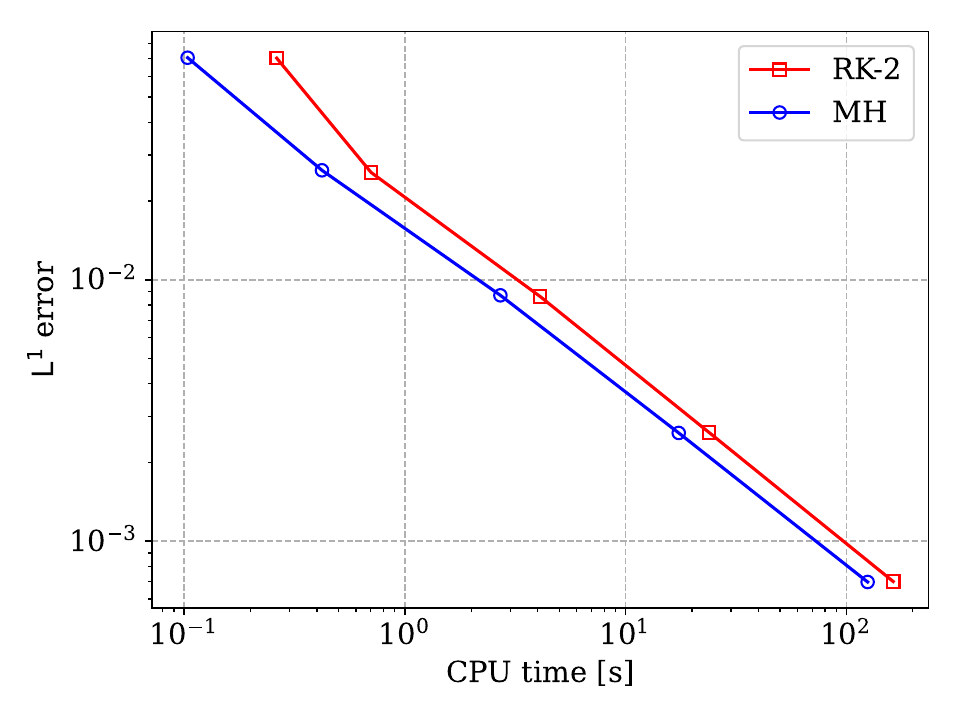}
    \caption*{(c)}
  \end{subfigure}

  \caption{Example \ref{example:2}. Log–log plots of $\mathrm{L}^{1}$ error versus CPU time for the MH and RK-2 schemes applied to \eqref{eq:problem} with the smooth initial condition \eqref{eq:icsmooth} at time $t = 0.15$. Results are displayed for three different choices of the interval $[a,b]$ in  \eqref{eq:amorimkernel}: (a) $[0.0, 0.25]$ (upstream convolution), (b) $[ -0.125,0.125]$ (centered convolution), and (c) $[-0.25,0.0]$ (downstream convolution).}
  \label{fig:cpu_time_smooth}
\end{figure}

\end{example}

\begin{example}\label{example:1}
    We consider the same setup as in Example \ref{example:1} but with a discontinuous initial datum 
\begin{align}\label{eq:icamorim}
    \rho_{0}(x)= \frac{1}{2}\chi_{[-2.8,-1.8]}(x) + \frac{3}{4}\chi_{[-1.2,-0.2]}(x) +\frac{3}{4}\chi_{[0.6,1.0]}(x) + \chi_{[1.5,+\infty)}(x)
\end{align} and $[a,b]=[-0.25, 0.0]$ (downstream convolution).
The numerical solutions are computed in the domain $[-3.0, 3.0]$ up to time $t=2.5$ using absorbing boundary conditions and the results are displayed in Figure \ref{fig:ex1}. Here, the reference solution is computed using the MH scheme with a fine mesh of size $\Delta x = {6}/{900}$ . While  both the MH and RK-2 schemes provide a comparable resolution as seen in Figure \ref{fig:ex1}, the $\mathrm{L}^{1}$ error versus CPU time plots presented in Figure \ref{fig:disc_cpu}(a) , computed for mesh sizes $\Delta x \in \{{6}/{40}, {6}/{80}, {6}/{160}, {6}/{320}\},$ indicate that the MH scheme is  computationally more efficient.
\begin{figure}[!tbp]
  \centering
  %— left —
  \begin{subfigure}[t]{0.49\textwidth}
    \includegraphics[width=\textwidth]{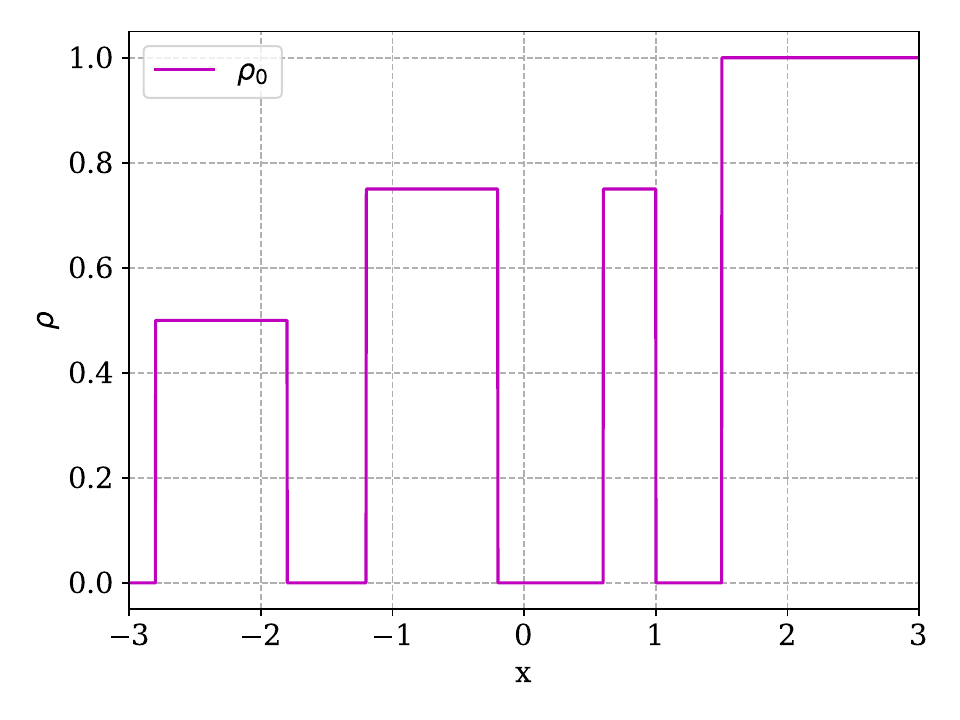}
    \caption*{(a)}
  \end{subfigure}
  \hfill
  %— right —
  \begin{subfigure}[t]{0.49\textwidth}
    \includegraphics[width=\textwidth]{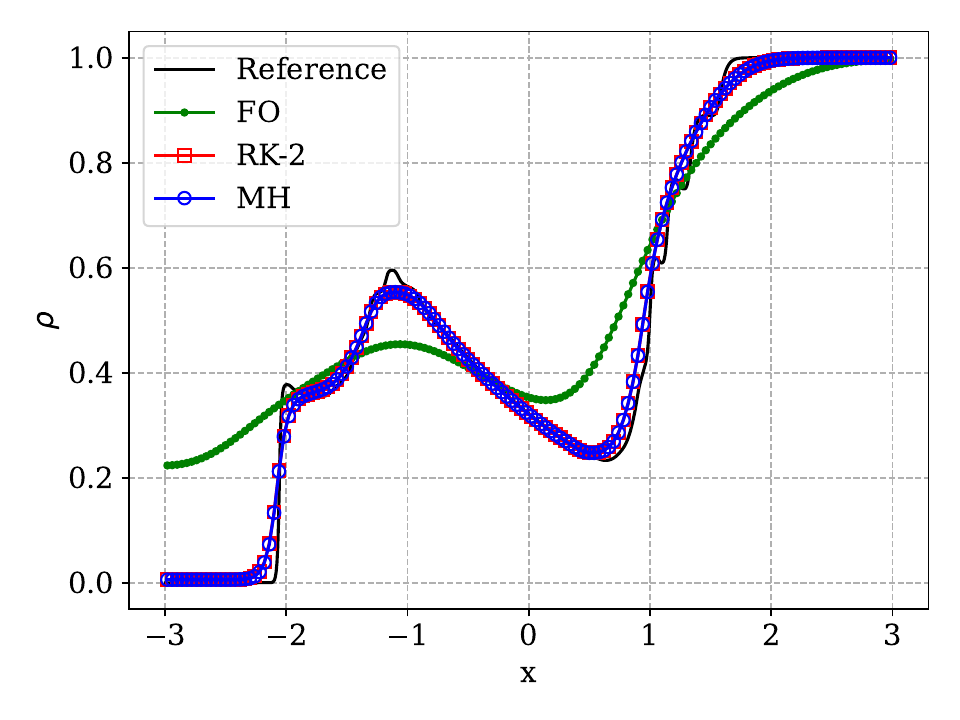}
    \caption*{(b)}
  \end{subfigure}

  \caption{Example \ref{example:1}. (a) Initial datum given in \eqref{eq:icamorim}. (b) Numerical solutions at time $t=2.5$, computed with the kernel function \eqref{eq:amorimkernel} with  $[a, b]=[-0.25,0.0]$  using $\Delta x = 6/150$ and $\Delta t = \Delta x/20.$ } 
  \label{fig:ex1}
\end{figure}
\end{example}

\begin{example}\label{example:3}
In this example, we consider the problem \eqref{eq:problem} with a different flux function given by
\begin{align}
    f(\rho, A) &= \rho(1-A),
\end{align}
and the kernel function $\mu$ given by
\begin{align}\label{eq:aggarwalkernel}
    \mu(x) &:= \frac{1}{\left(\int_{-\eta}^{0}(-x(\eta+x))^{3} \dif x\right)}(-x(\eta+x))^{3} \chi_{[-\eta,0]}, 
\end{align}  considered in \cite{aggarwal_holden_vaidya2024}.
We choose a discontinuous initial datum
\begin{align}\label{eq:ic_aggarwal}
    \rho_0(x) &= \begin{cases}
        0.25 \quad \mbox{for} \,\, -0.9 \leq x \leq 0.1,\\
        0.5 \quad \,\,\, \mbox{for} \,\, 0.1 \leq x \leq 0.3
    \end{cases}
\end{align}
and evolve the numerical solutions in the domain $[-1.5,1.5]$ up to time $t=0.5$ using absorbing boundary conditions and the results are illustrated in Figure \ref{fig:ex3}. Here, the reference solution is computed using the MH scheme with a mesh of size $\Delta x =3/900.$ Additionally, we present the $\mathrm{L}^{1}$ error versus CPU time plots in Figure \ref{fig:disc_cpu}(b), computed for mesh sizes $\Delta x \in \{{6}/{40}, {6}/{80}, {6}/{160}, {6}/{320}\}.$ 
The comparison in Figure \ref{fig:ex3} shows that the MH scheme (and RK-2) outperforms the FO scheme, as expected. As in Example \ref{example:1}, although the MH and RK-2 schemes produce comparable solutions, Figure \ref{fig:disc_cpu}(b) indicates that the MH scheme  is  computationally  more  efficient.

\begin{figure}[!tbp]
  \centering
  %— left —
  \begin{subfigure}[t]{0.49\textwidth}
    \includegraphics[width=\textwidth]{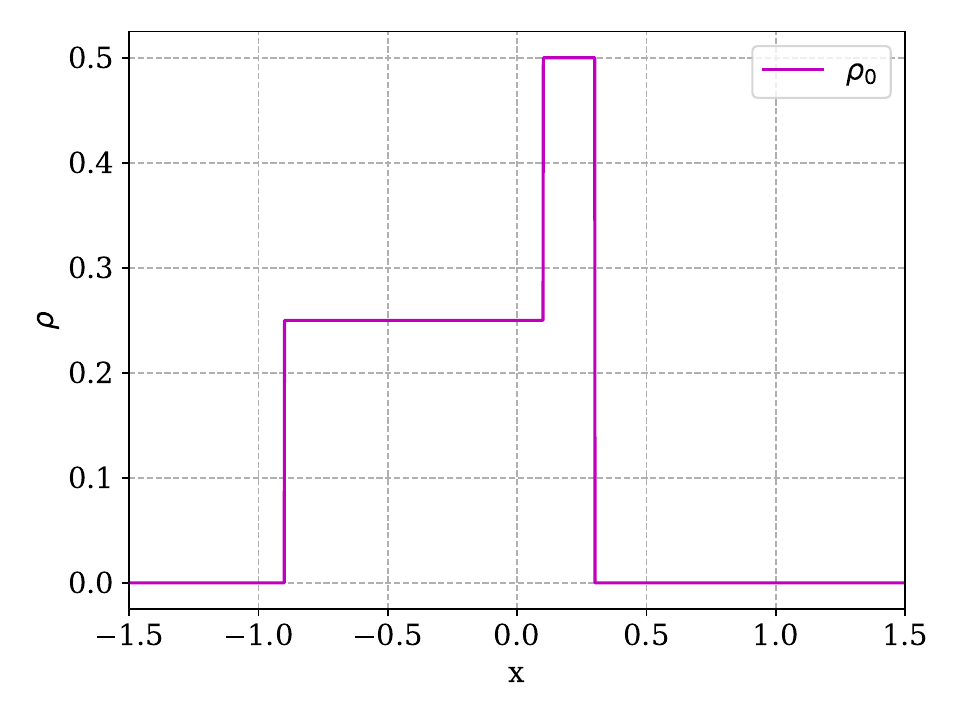}
    \caption*{(a)}
  \end{subfigure}
  \hfill
  %— right —
  \begin{subfigure}[t]{0.49\textwidth}
    \includegraphics[width=\textwidth]{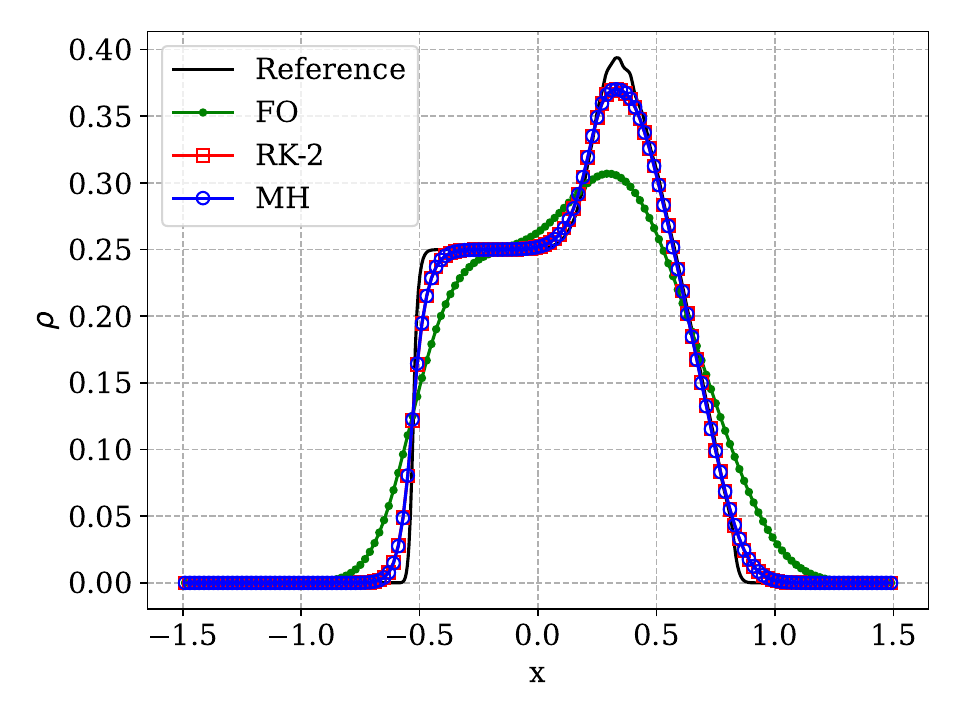}
    \caption*{(b) }
  \end{subfigure}

  \caption{Example \ref{example:3}. (a) Initial datum given in \eqref{eq:ic_aggarwal}. (b) Numerical solutions at time $t=0.5,$ computed with the kernel function \eqref{eq:aggarwalkernel}, using $\Delta x = 3/150$ and $\Delta t =\Delta x/20.$ }
  \label{fig:ex3}
\end{figure}

\begin{figure}[!tbp]
  \begin{subfigure}[t]{0.49\textwidth}
    \includegraphics[width=\textwidth]{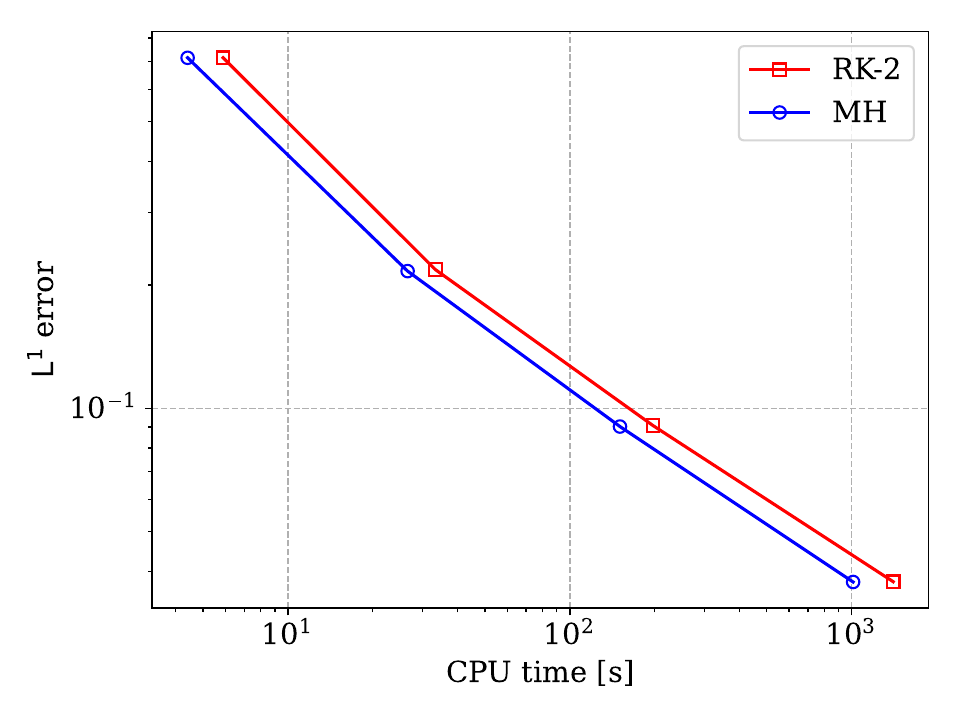}
    \caption*{(a) }
  \end{subfigure}
  \hfill
  \begin{subfigure}[t]{0.49\textwidth}
    \includegraphics[width=\textwidth]{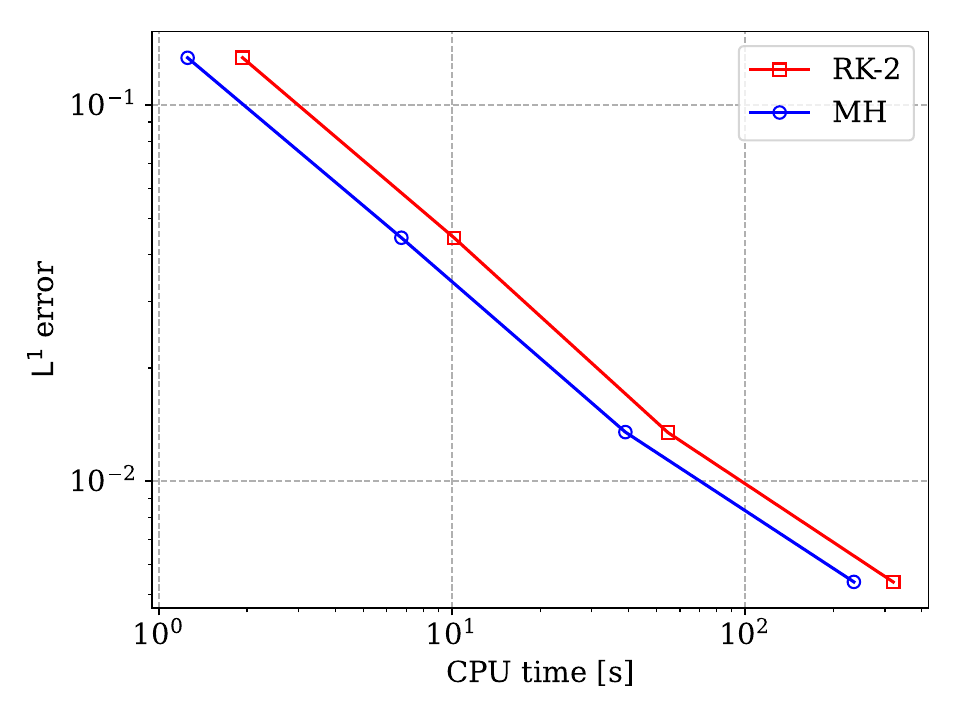}
    \caption*{(b)}
  \end{subfigure}

  \caption{Examples \ref{example:1} and \ref{example:3}. Log–log plots of $\mathrm{L}^{1}$ error versus CPU time for the MH and RK-2 schemes for discontinuous solutions. (a) Example \ref{example:1} and (b) Example \ref{example:3}.} 
\label{fig:disc_cpu}

\end{figure}    
\end{example}

% \section{CPU time}

% \begin{figure}[!tbp]
%   \begin{subfigure}[b]{0.49\textwidth}
%     \includegraphics[width=\textwidth]{figs/CPUtime-oldMHvsRK.pdf}
%     \caption*{(a) Current MH}
%   \end{subfigure}
%   \hfill
%   \begin{subfigure}[b]{0.49\textwidth}
%     \includegraphics[width=\textwidth]{figs/CPUtime_MHlatvsRK.pdf}
%     \caption*{(b) MH with a backward difference modification in the predictor}
%   \end{subfigure}

%   \caption{CPU time comparison for discontinuous solution} 
% \label{fig:smoothcpu}

% \end{figure}

\section{Conclusion}

In this work, we have proposed a single-stage MUSCL–Hancock-type second-order scheme for a general class of non-local conservation laws and established its convergence analysis. The construction of the scheme relies on a careful design of the discrete convolutions, which is crucial not only for achieving second-order accuracy but also for enabling a rigorous convergence analysis. The  numerical results presented in Section~\ref{section:numerics} confirm that the proposed scheme produces numerical solution of  significantly improved accuracy than a typical first-order method.
 To highlight the advantage of the proposed MH-type scheme, we also compare it with a standard second-order MUSCL–RK-2 method. For the test cases in Examples \ref{example:1} and \ref{example:3}, both the MH and RK-2 schemes produce comparable solutions. However, the results shown in Figures~\ref{fig:disc_cpu}(a) and~\ref{fig:disc_cpu}(b) reveal that the MH scheme is computationally more efficient than the RK-2 scheme, underscoring the significance of the proposed method.
We note that this work opens several avenues for further exploration, including the extension to multidimensional problems and the development of a corresponding convergence analysis. We postpone these investigations to future work.

\section*{Acknowledgments} Nikhil Manoj acknowledges financial support in the form of a doctoral fellowship from the Council of Scientific and Industrial Research (CSIR), Government of India.

\appendix

\section{Technical estimates for the total variation bound}\label{section:tvbappendix}
In this section, we derive certain estimates required in the proof of Theorem \ref{theorem:TVbound}.
\subsection{Bound on $\displaystyle \sum_{j \in \mathbb{Z}} \abs{\tilde{C}_{\jph}^{n}} $}\label{subsec:tilde_k_l_bds}
In order to estimate the term $\displaystyle \sum_{j \in \mathbb{Z}} \abs{\tilde{C}_{\jph}^{n}},$ we first show that
$\displaystyle 0 \leq \tilde{k}_{\jph}^n, \tilde{\ell}_{\jph}^n \leq \frac{1}{2},$ for $\jinz.$
Using Remark \ref{remark:ratioslopes} (equation \eqref{eq:ratiobd}) and the CFL condition \eqref{eq:cfl_psty},  it is straightforward to see that
\begin{nalign}\label{eq:ktildelbd}
& \left(\frac{\rho_{j+\frac{3}{2}}^{n,-}-\rho_{j+\frac{1}{2}}^{n,-}}{\rho_{j+1}^{n}-\rho_{j}^{n}}\right)\left(1-\frac{\lambda}{2}\partial_{\rho}f(\bar{\rho}_{j+1}^{n,-}, A_{j+\frac{3}{2}}^{n,-})\right)+ \left(\frac{\rho_{\jph}^{n,+}- \rho_{\jmh}^{n,+}}{\rho_{j+1}^{n}-\rho_{j}^{n}}\right)\left(\frac{\lambda}{2}\partial_{\rho}f(\bar{\rho}^{n,+}_{j}, A_{\jph}^{n,+})\right)\\
& \spc \geq \frac{1}{2}\left(1-\frac{\lambda}{2}\norm{\partial_{\rho}f}\right) - (1+\theta)\frac{\lambda}{2}\norm{\partial_{\rho}f} \geq 0.
\end{nalign}
Further, from the CFL condition \eqref{eq:cfl_psty}, it follows that $\displaystyle\frac{1}{2}\left(\lambda\partial_{\rho}f(\bar{\rho}_{\jpo}^{\nph,-},A_{j+\frac{3}{2}}^{\nph})+\alpha\right) \geq 0,$
which together with \eqref{eq:ktildelbd} implies that $\tilde{k}_{\jph}^n \geq 0.$ Using the CFL condition \eqref{eq:cfl_psty} once again, we obtain 
\begin{align}
    \tilde{k}_{\jph}^n &\leq  \frac{1}{2}\left(\lambda\norm{\partial_{\rho}f}+\alpha\right)
 (1+\theta)\left(1+{\lambda}\norm{\partial_{\rho}f}\right) \leq \frac{1}{2} \times \frac{8}{27} \times \frac{3}{2} \times \frac{29}{27} \leq \frac{1}{2},
 \end{align}
thereby concluding that  $0 \leq \tilde{k}_{\jph}^n \leq \frac{1}{2},$ for all $\jinz.$ Analogously, we 
obtain a bound $0 \leq \tilde{\ell}_{\jph}^n \leq \frac{1}{2}.$
\par
Since $0 \leq \tilde{k}_{\jph}^n, \tilde{\ell}_{\jph}^n \leq \frac{1}{2}$ for $\jinz,$
from \eqref{eq:tildehatCs} we obtain the desired estimate
\begin{align}\label{eq:tildeCbd}
    \sum_{j \in \mathbb{Z}} \abs{\tilde{C}_{\jph}^{n}} &\leq \sum_{j \in \mathbb{Z}} \abs{\rho_{j+1}^n- \rho_{j}^n}.
\end{align}
\subsection{Bound on $\displaystyle \sum_{\jinz}\abs{\hat{C}_{\jph}^n},$}\label{subsec:Cjphsumbd}
Recall from \eqref{eq:tildehatCs} that 
\begin{align*}
    \hat C_\jph =  -\hat{\ell}_{j+\frac{1}{2}}^n-\hat{k}_{j+\frac{1}{2}}^n +
\hat{\ell}_{j+\frac{3}{2}}^n +\hat{k}_{j-\frac{1}{2}}^n.
\end{align*}
To analyze this term, we first focus on the  difference $\hat{\ell}_{j+\frac{1}{2}}^n - \hat{\ell}_{j-\frac{1}{2}}^n.$  Upon expanding it using the definition in \eqref{eq:ltildehat} and subsequently rearranging the terms, we obtain\begin{nalign}\label{eq:hatldif}
     \hat{\ell}_{j+\frac{1}{2}}^n - \hat{\ell}_{j-\frac{1}{2}}^n 
     & = \mathcal{L}^1_{j} +  \mathcal{L}^2_{j},
 \end{nalign}

% = \frac{1}{2}\left(\alpha-\lambda\partial_{\rho}f(\bar{\rho}_{j}^{\nph,+},A_{j+\frac{1}{2}}^{\nph})\right) \left[- \frac{\lambda}{2}\partial_{A}f(\rho_{\jph}^{n,-}, \bar{A}_{j+1}^{n,-})  \left(A_{j+\frac{3}{2}}^{n,-}-A_{\jph}^{n,-}\right) \right.\\ & \spc \spc \left. + \frac{\lambda}{2}\partial_{A}f(\rho_{\jmh}^{n,+}, \bar{A}_{j}^{n,+})  \left(A_{\jph}^{n,+}-A_{\jmh}^{n,+}\right)\right] \\ 
%      & \spc -\frac{\left(\alpha-\lambda\partial_{\rho}f(\bar{\rho}_{j-1}^{\nph,+},A_{j-\frac{1}{2}}^{\nph})\right)}{2} \left[- \frac{\lambda}{2}\partial_{A}f(\rho_{\jmh}^{n,-}, \bar{A}_{j}^{n,-})  \left(A_{j+\frac{1}{2}}^{n,-}-A_{\jmh}^{n,-}\right) \right.\\ & \spc \spc \left. + \frac{\lambda}{2}\partial_{A}f(\rho_{\jmtbt}^{n,+}, \bar{A}_{j-1}^{n,+})  \left(A_{\jph}^{n,+}-A_{\jmtbt}^{n,+}\right)\right]\\

 where 
 \begin{nalign}\label{eq:L1L2}
     \mathcal{L}^1_{j} &:= \frac{1}{2}\left(\alpha-\lambda\partial_{\rho}f(\bar{\rho}_{j}^{\nph,+},A_{j+\frac{1}{2}}^{\nph})\right) \left[- \frac{\lambda}{2}\partial_{A}f(\rho_{\jph}^{n,-}, \bar{A}_{j+1}^{n,-})  \left(A_{j+\frac{3}{2}}^{n,-}-A_{\jph}^{n,-}\right)\right] \\ 
     & \spc -\frac{1}{2}\left(\alpha-\lambda\partial_{\rho}f(\bar{\rho}_{j-1}^{\nph,+},A_{j-\frac{1}{2}}^{\nph})\right) \left[- \frac{\lambda}{2}\partial_{A}f(\rho_{\jmh}^{n,-}, \bar{A}_{j}^{n,-})  \left(A_{j+\frac{1}{2}}^{n,-}-A_{\jmh}^{n,-}\right)\right],\\
     \mathcal{L}^{j}_2 &:= \frac{1}{2}\left(\alpha-\lambda\partial_{\rho}f(\bar{\rho}_{j}^{\nph,+},A_{j+\frac{1}{2}}^{\nph})\right) \left[ \frac{\lambda}{2}\partial_{A}f(\rho_{\jmh}^{n,+}, \bar{A}_{j}^{n,+})  \left(A_{\jph}^{n,+}-A_{\jmh}^{n,+}\right)\right] \\ 
     & \spc -\frac{1}{2}\left(\alpha-\lambda\partial_{\rho}f(\bar{\rho}_{j-1}^{\nph,+},A_{j-\frac{1}{2}}^{\nph})\right) \left[ \frac{\lambda}{2}\partial_{A}f(\rho_{\jmtbt}^{n,+}, \bar{A}_{j-1}^{n,+})  \left(A_{\jph}^{n,+}-A_{\jmtbt}^{n,+}\right)\right].
\end{nalign}
Turning to the term  $\mathcal{L}^1_{j},$ by adding and subtracting suitable terms, we write  
\begin{nalign}\label{eq:L1exp}
    \mathcal{L}^1_{j}  &= \frac{1}{2}\left(\alpha-\lambda\partial_{\rho}f(\bar{\rho}_{j}^{\nph,+},A_{j+\frac{1}{2}}^{\nph})\right) \left(- \frac{\lambda}{2}\partial_{A}f(\rho_{\jph}^{n,-}, \bar{A}_{j+1}^{n,-})\right)  \left(A_{j+\frac{3}{2}}^{n,-}-2A_{\jph}^{n,-}+A_{\jmh}^{n,-}\right)\\ & \spc +\left(A_{j+\frac{1}{2}}^{n,-}-A_{\jmh}^{n,-}\right) \left[\frac{1}{2}\left(\alpha-\lambda\partial_{\rho}f(\bar{\rho}_{j}^{\nph,+},A_{j+\frac{1}{2}}^{\nph})\right) \left(- \frac{\lambda}{2}\partial_{A}f(\rho_{\jph}^{n,-}, \bar{A}_{j+1}^{n,-})\right) \right. \\
    &\spc \spc \left. -\frac{1}{2}\left(\alpha-\lambda\partial_{\rho}f(\bar{\rho}_{j-1}^{\nph,+},A_{j-\frac{1}{2}}^{\nph})\right) \left(- \frac{\lambda}{2}\partial_{A}f(\rho_{\jmh}^{n,-}, \bar{A}_{j}^{n,-})  \right) \right].
    %%%%%%%%%%%%%%%%%%
\end{nalign}
Further, appropriately adding and subtracting the terms $\mathcal{T}_{1} := \partial_{A}f(\rho_{\jph}^{n,-}, \bar{A}_{j}^{n,-}),$ $\mathcal{T}_{2} :=  \partial_{\rho}f(\bar{\rho}_{j}^{\nph,+},A_{j-\frac{1}{2}}^{\nph})$ and $\mathcal{T}_{3} := \partial_{A}f(\rho_{\jmh}^{n,-}, \bar{A}_{j}^{n,-})$ to \eqref{eq:L1exp}, we obtain
\begin{nalign}
     \mathcal{L}^1_{j} &=  \frac{1}{2}\left(\alpha-\lambda\partial_{\rho}f(\bar{\rho}_{j}^{\nph,+},A_{j+\frac{1}{2}}^{\nph})\right) \left(- \frac{\lambda}{2}\partial_{A}f(\rho_{\jph}^{n,-}, \bar{A}_{j+1}^{n,-})\right)  \left(A_{j+\frac{3}{2}}^{n,-}-2A_{\jph}^{n,-}+A_{\jmh}^{n,-}\right)\\ & \spc +\left(A_{j+\frac{1}{2}}^{n,-}-A_{\jmh}^{n,-}\right) \frac{1}{2}\left(\alpha-\lambda\partial_{\rho}f(\bar{\rho}_{j}^{\nph,+},A_{j+\frac{1}{2}}^{\nph})\right) \frac{\lambda}{2}\left(- \partial_{A}f(\rho_{\jph}^{n,-}, \bar{A}_{j+1}^{n,-})+ \mathcal{T}_{1}-   \mathcal{T}_{1} + \mathcal{T}_{3} \right)  \\
    &\spc  + \left(A_{j+\frac{1}{2}}^{n,-}-A_{\jmh}^{n,-}\right) \left(- \frac{\lambda}{2}\mathcal{T}_{3}  \right) \frac{1}{2}\left[ \left(\alpha-\lambda\partial_{\rho}f(\bar{\rho}_{j}^{\nph,+},A_{j+\frac{1}{2}}^{\nph})\right) \right.\\ 
    & \spc \spc \left. + \lambda\mathcal{T}_{2}  -\lambda\mathcal{T}_{2} - \left(\alpha-\lambda\partial_{\rho}f(\bar{\rho}_{j-1}^{\nph,+},A_{j-\frac{1}{2}}^{\nph})\right) \right], 
\end{nalign}
which we split upon applying the mean value theorem as
\begin{nalign}
   \mathcal{L}^{1}_j & = \mathcal{L}_j^{a} + \mathcal{L}_j^{b} + \mathcal{L}_j^{c} + \mathcal{L}_j^{d},
\end{nalign}
where
\begin{nalign}\label{eq:L1abcd}
    \mathcal{L}_j^{a}& :=\frac{1}{2}\left(\alpha-\lambda\partial_{\rho}f(\bar{\rho}_{j}^{\nph,+},A_{j+\frac{1}{2}}^{\nph})\right) \left(- \frac{\lambda}{2}\partial_{A}f(\rho_{\jph}^{n,-}, \bar{A}_{j+1}^{n,-})\right)  \left(A_{j+\frac{3}{2}}^{n,-}-2A_{\jph}^{n,-}+A_{\jmh}^{n,-}\right),\\
    %%%%%%%%%%%%
    \mathcal{L}_j^{b} &:= \left(A_{j+\frac{1}{2}}^{n,-}-A_{\jmh}^{n,-}\right) \frac{1}{2}\left(\alpha-\lambda\partial_{\rho}f(\bar{\rho}_{j}^{\nph,+},A_{j+\frac{1}{2}}^{\nph})\right) \frac{\lambda}{2}\left(\partial_{\rho A}^{2}f(\bar{\rho}_{j}^{n,-}, \bar{A}_{j}^{n,-})(\rho_{\jmh}^{n,-}-\rho_{\jph}^{n,-}) \right.\\ & \spc \spc + \left.\partial_{AA}^{2}f(\rho_{\jph}^{n,-}, \bar{A}_{\jph}^{n,-})(\bar{A}_{j}^{n,-}- \bar{A}_{j+1}^{n,-})\right),\\
    \mathcal{L}_j^{c} &:= \left(A_{j+\frac{1}{2}}^{n,-}-A_{\jmh}^{n,-}\right) \left(- \frac{\lambda}{2}\mathcal{T}_3 \right)\frac{\lambda}{2}\partial_{\rho\rho}^{2}f(\bar{\bar{\rho}}_{\jmh}^{\nph,+}, A_{\jmh}^\nph)(\bar{\rho}_{j-1}^{\nph,+}- \bar{\rho}_{j}^{\nph,+}),\\
    %%%%%%%%%%%%%%%%%%%%
   \mathcal{L}_j^{d} &:=  \left(A_{j+\frac{1}{2}}^{n,-}-A_{\jmh}^{n,-}\right) \left(- \frac{\lambda}{2}\mathcal{T}_3\right)\frac{\lambda}{2}\partial_{\rho A}^{2}f(\bar{\rho}_{j}^{\nph,+},\bar{A}_{j}^{\nph})(A_{\jmh}^{\nph}-A_{\jph}^{\nph}),
\end{nalign}
where $\bar{\rho}_{j}^{\nph,\pm} \in \mathcal{I}({\rho}_{\jmh}^{\nph,\pm}, {\rho}_{\jph}^{\nph,\pm}), \bar{\bar{\rho}}_{\jmh}^{\nph,+} \in \mathcal{I}(\bar{\rho}_{j-1}^{\nph,+}, \bar{\rho}_{j}^{\nph,+})$ and $ \bar{A}_{j}^{n,-} \in \mathcal{I}({A}_{\jmh}^{n,-},{A}_{\jph}^{n,-}),$ for $\jinz.$
\par
Next, expanding $
    A_{j+\frac{3}{2}}^{n,-}-2A_{\jph}^{n,-}+A_{\jmh}^{n,-} = \left(A_{j+1}^{n}-2A_{j}^{n}+A_{j-1}^{n}\right) + \frac{1}{2}\left(s_{j+1}^n -2s_{j}^n+ s_{j-1}^n\right),$
and using the definition of the discrete convolution from \eqref{eq:conv_appr} together with Theorems \ref{theorem:positivity} and \ref{theorem:L1stability}, we obtain
\begin{nalign}\label{eq:conv_3_dif}
     \abs{A_{j+1}^{n}-2A_{j}^{n}+A_{j-1}^{n}} & =\abs{\Delta x  \sum_{\jinz}(\mu_{j+1-l}- 2\mu_{j-l}+ \mu_{j-1-l})\rho_{l}^n}\\
   & = \Delta x^3\abs{\sum_{\jinz}\mu^{\prime\prime}(\bar{x}_{j-l})\rho_{l}^n} \leq \Delta x^2 \norm{\mu^{\prime\prime}}\norm{\rho_{0}}_{\mathrm{L}^1(\mathbb{R})},
\end{nalign}
for some $\bar{x}_{j-l} \in (x_{j-l-1}, x_{j-l+1}).$
Further, noticing that
\begin{align*}
    \frac{1}{2}\left(s_{j+1}^n -2s_{j}^n+ s_{j-1}^n\right) &= \frac{1}{2}\theta \left[(A_{j+2}^n - A_{j+1}^n)- (A_{j+1}^n - A_{j}^n)-(A_{j}^n - A_{j-1}^n)+ (A_{j-1}^n - A_{j-2}^n)\right]  \\ & = \frac{1}{2}\theta \left[(A_{j+2}^n - 2A_{j+1}^n + A_{j}^n)-(A_{j}^n - 2A_{j-1}^n+ A_{j-2}^n)\right], 
\end{align*}
 yields
\begin{align}\label{eq:slope_3_dif}
    \frac{1}{2}\abs{s_{j+1}^n -2s_{j}^n+ s_{j-1}^n} & \leq \theta \Delta x^2 \norm{\mu^{\prime\prime}}\norm{\rho_{0}}_{\mathrm{L}^1(\mathbb{R})}.   
\end{align}
In view of \eqref{eq:slope_3_dif} and \eqref{eq:conv_3_dif}, we have the estimate
\begin{align}\label{eq:leftconv_3_bd}
\abs{A_{j+\frac{3}{2}}^{n,-}-2A_{\jph}^{n,-}+A_{\jmh}^{n,-}} & \leq (1+\theta)\Delta x^2 \norm{\mu^{\prime\prime}}\norm{\rho_{0}}_{\mathrm{L}^1(\mathbb{R})}.   
\end{align}

Now, applying \eqref{eq:leftconv_3_bd} together with Lemma \ref{lemma:reconvaluebd} and hypotheses \eqref{hyp:H1} and \eqref{hyp:H2}, we obtain
\begin{nalign}\label{eq:L1a_bd}
    \sum_{\jinz}\abs{\mathcal{L}_j^{a}} &\leq \frac{1}{2}\left(\alpha+ \lambda \norm{\partial_{\rho}f}\right)\frac{\lambda}{2}M(1+\theta)\Delta x^2 \norm{\mu^{\prime\prime}}\norm{\rho_0}_{\mathrm{L}^1(\mathbb{R})} \sum_{\jinz}\abs{\rho_{\jph}^{n,-}}\\ &
    \leq \frac{1}{2}\left(\alpha+ \lambda \norm{\partial_{\rho}f}\right)\frac{\lambda}{2}M(1+\theta)^{2}\Delta x \norm{\mu^{\prime\prime}}\norm{\rho_0}_{\mathrm{L}^1(\mathbb{R})} \Delta x \sum_{\jinz} \rho_{j}^{n} \\
    & \leq \Delta t \frac{1}{4}\left(\alpha+ \lambda \norm{\partial_{\rho}f}\right) M(1+\theta)^{2} \norm{\mu^{\prime\prime}}\norm{\rho_0}_{\mathrm{L}^1(\mathbb{R})}^{2}.
\end{nalign}

Moving to the estimation of the sum $\displaystyle \sum_{\jinz}\abs{\mathcal{L}_{j}^{b}},$   we bound the difference $\bar{A}^{n,-}_{j}- \bar{A}^{n,-}_{j+1}$ using the estimate \eqref{eq:lrvaluediff_bd} as 
\begin{nalign}\label{eq:barAdiff}
    \abs{\bar{A}^{n,-}_{j}- \bar{A}^{n,-}_{j+1}} & = \abs{\gamma_{1}A_{\jmh}^{n,-}+ (1-\gamma_1)A_{\jph}^{n,-}-(1-\gamma_{2})A^{n,-}_{\jph}-\gamma_{2} A_{j+\frac{3}{2}}^{n,-}}\\
    & \leq \gamma_{1}\abs{A_{\jmh}^{n,-}-A_{\jph}^{n,-}}+ \gamma_{2}\abs{A^{n,-}_{\jph}- A_{j+\frac{3}{2}}^{n,-}}\\
    & \leq 2(1+\theta)\Delta x \norm{\mu^{\prime}}\norm{\rho_{0}}_{\mathrm{L}^{1}(\mathbb{R})},
\end{nalign} for some $\gamma_1, \gamma_2 \in (0,1).$
Now, invoking \eqref{eq:lrvaluediff_bd}, hypothesis \eqref{hyp:H2}, property \eqref{eq:tv_recomstruction}, \eqref{eq:barAdiff} and Lemma \ref{lemma:reconvaluebd}, we obtain
\begin{nalign}\label{eq:L1b_bd}
\sum_{\jinz}\abs{\mathcal{L}_j^{b}} & \leq  (1+\theta)\Delta x \norm{\mu^{\prime}}\norm{\rho_{0}}_{\mathrm{L}^{1}(\mathbb{R})}  \frac{1}{4}(\alpha+\lambda\norm{\partial_{\rho}f}) \lambda\left(\norm{\partial_{\rho A}^{2}f}\sum_{\jinz}\abs{\rho_{\jmh}^{n,-}-\rho_{\jph}^{n,-}}+ M\sum_{\jinz}\abs{\bar{A}^{n,-}_{j}- \bar{A}^{n,-}_{j+1}}\abs{\rho^{n,-}_{\jph}}\right)\\
& \leq \Delta t (1+\theta)\norm{\mu^{\prime}}\norm{\rho_{0}}_{\mathrm{L}^{1}(\mathbb{R})}  \frac{1}{4}(\alpha+\lambda\norm{\partial_{\rho}f}) \left(\norm{\partial_{\rho A}^{2}f} \sum_{\jinz}\abs{\rho^n_{j} - \rho_{j-1}^n}+ 2M(1+\theta)^{2}\norm{\mu^\prime}\norm{\rho_0}^{2}_{\mathrm{L}^{1}(\mathbb{R})}\right).
\end{nalign}

Next, aiming to estimate the sum $\displaystyle \sum_{\jinz}\abs{\mathcal{L}_j^{c}},$ we first consider the difference $\bar{\rho}_{j-1}^{\nph,+}- \bar{\rho}_{j}^{\nph,+}$ and write
\begin{nalign}\label{eq:rhobarmidtimediff}
    \bar{\rho}_{j}^{\nph,+}- \bar{\rho}_{j-1}^{\nph,+} & = \tilde{\gamma}_{1} \rho_{\jph}^{\nph,+} + (1-\tilde{\gamma}_{1})\rho_{\jmh}^{\nph,+} - \tilde{\gamma_{2}}\rho_{\jmh}^{\nph,+} - (1-\tilde{\gamma_{2}})\rho_{\jmtbt}^{\nph,+}\\
    & = \tilde{\gamma}_{1}(\rho_{\jph}^{\nph,+}- \rho_{\jmh}^{\nph,+}) +  (1- \tilde{\gamma_{2}})(\rho_{\jmh}^{\nph,+}- \rho_{\jmtbt}^{\nph,+}),
\end{nalign}
for some $\tilde{\gamma}_{1}, \tilde{\gamma}_{2} \in (0,1).$ 
Now, proceeding analogously to  \eqref{eq:leftmidvldif}, we write
\begin{nalign}\label{eq:midtimerval_dif}
     \rho_{j+\frac{1}{2}}^{\nph,+}-\rho_{j-\frac{1}{2}}^{\nph,+} 
    & = \left(\rho_{j+\frac{3}{2}}^{n,-}-\rho_{j+\frac{1}{2}}^{n,-}\right)\left(-\frac{\lambda}{2}\partial_{\rho}f(\bar{\rho}_{j+1}^{n,-}, A_{j+\frac{3}{2}}^{n,-})\right) \\
    & \spc + \left(\rho_{\jph}^{n,+}- \rho_{\jmh}^{n,+}\right)\left(1+\frac{\lambda}{2}\partial_{\rho}f(\bar{\rho}^{n,+}_{j}, A_{\jph}^{n,+})\right)\\ & \spc - \frac{\lambda}{2}\partial_{A}f(\rho_{\jph}^{n,-}, \bar{A}_{j+1}^{n,-})  \left(A_{j+\frac{3}{2}}^{n,-}-A_{\jph}^{n,-}\right) + \frac{\lambda}{2}\partial_{A}f(\rho_{\jmh}^{n,+}, \bar{A}_{j}^{n,+})  \left(A_{\jph}^{n,+}-A_{\jmh}^{n,+}\right),
\end{nalign}
and subsequently apply hypothesis \eqref{hyp:H2} and the estimate \eqref{eq:lrvaluediff_bd}, to yield
\begin{nalign}\label{eq:rhormidtimediffbd}
    \abs{\rho_{j+\frac{1}{2}}^{\nph,+}-\rho_{j-\frac{1}{2}}^{\nph,+} } 
    & \leq \frac{\lambda}{2}\norm{\partial_{\rho}f}\abs{\rho_{j+\frac{3}{2}}^{n,-}-\rho_{j+\frac{1}{2}}^{n,-}}  + \abs{\rho_{\jph}^{n,+}- \rho_{\jmh}^{n,+}}\left(1+\frac{\lambda}{2}\norm{\partial_{\rho}f}\right)\\ &\spc + \frac{\lambda}{2}M(\abs{\rho_{\jph}^{n,-}}+\abs{\rho_{\jmh}^{n,+}})(1+\theta)\Delta x \norm{\mu^{\prime}}\norm{\rho_{0}}_{\mathrm{L}^{1}(\mathbb{R})}.
\end{nalign}

% \begin{nalign}
%     \abs{A_{j+1}^n - A_{j}^n} &= \abs{\Delta x \sum_{l \in \mathbb{Z}}(\mu_{j+1-l}-\mu_{j-l})        \ \rho^{n}_{l}}\\
%     & \leq \Delta x \norm{\mu^\prime}\norm{\rho^n}_{\mathrm{L}^1(\mathbb{R})} \leq \Delta x \norm{\mu^\prime}\norm{\rho_0}_{\mathrm{L}^1(\mathbb{R})}.
% \end{nalign}
% This, combined with \eqref{eq:lrconv} implies that
% \begin{align}
%     \abs{s_{j}^n -s^{n}_{j-1}} &= \theta \abs{(A_{j+1}^n - A_{j}^n) - (A_{j-1}^n - A_{j-2}^n)} \leq 2 \theta \Delta x \norm{\mu^\prime}\norm{\rho_0}_{\mathrm{L}^1(\mathbb{R})}.
% \end{align}
% Therefore, for $\jinz,$ 
% \begin{nalign}
%    \abs{A_{\jph}^{n,-}-A_{\jmh}^{n,-}} & \leq  \abs{A_{j}^n - A_{j-1}^n} + \frac{1}{2}\abs{s_{j}^n -s^{n}_{j-1}}\\
%    & \leq \Delta x \norm{\mu^\prime}\norm{\rho_0}_{\mathrm{L}^1(\mathbb{R})} +\theta \Delta x \norm{\mu^\prime}\norm{\rho_0}_{\mathrm{L}^1(\mathbb{R})} = \Delta x(1+\theta) \norm{\mu^\prime}\norm{\rho_0}_{\mathrm{L}^1(\mathbb{R})}
% \end{nalign}
% Similarly,
% \begin{align}
%     \abs{A_{\jmh}^{n,+}-A_{\jmtbt}^{n,+}} \leq \Delta x(1+\theta) \norm{\mu^\prime}\norm{\rho_0}_{\mathrm{L}^1(\mathbb{R})} \quad \mbox{for} \, \jinz.
% \end{align}
% Using the above, it is clear that 
% \begin{align}
%     \abs{\hat{k}_{\jph}^n}, \abs{\hat{\ell}_{\jph}^n} &\leq \frac{1}{2} \Delta t(\alpha+\lambda \norm{\partial_{\rho}f})\norm{\partial_{A}f}(1+\theta)\norm{\mu^\prime}\norm{\rho_0}_{\mathrm{L^1(\mathbb{R})}} \quad \mbox{for} \, \, \jinz.
% \end{align}
In \eqref{eq:rhobarmidtimediff}, taking the absolute values  and summing over $\jinz,$ and subsequently invoking the estimate \eqref{eq:rhormidtimediffbd}, Theorems \ref{theorem:positivity} and \ref{theorem:L1stability}, hypothesis \eqref{hyp:H2}, 
Lemma \ref{lemma:reconvaluebd} and property \eqref{eq:tv_recomstruction}, it follows that
\begin{nalign}\label{eq:rhobardif}
\sum_{\jinz}\abs{\bar{\rho}_{j-1}^{\nph,+}- \bar{\rho}_{j}^{\nph,+}} 
    & \leq \sum_{\jinz}  \abs{\rho_{\jph}^{\nph,+}- \rho_{\jmh}^{\nph,+}} + \sum_{\jinz} \abs{\rho_{\jmh}^{\nph,+}- \rho_{\jmtbt}^{\nph,+}} \\
    & \leq \lambda\norm{\partial_{\rho}f}\sum_{\jinz}\abs{\rho_{j+\frac{1}{2}}^{n,-}-\rho_{j-\frac{1}{2}}^{n,-}}  + \left(2+{\lambda}\norm{\partial_{\rho}f}\right)\sum_{\jinz}\abs{\rho_{\jph}^{n,+}- \rho_{\jmh}^{n,+}}\\ &\spc + 2\lambda M (1+\theta)^{2}\norm{\mu^{\prime}}\norm{\rho_{0}}_{\mathrm{L}^{1}(\mathbb{R})}  \Delta x\sum_{\jinz}\rho_{j}^{n}\\
    & \leq 2\left(1+{\lambda}\norm{\partial_{\rho}f}\right)\sum_{\jinz}\abs{\rho^n_{j} - \rho_{j-1}^n}  + 2 \lambda M(1+\theta)^{2}\norm{\mu^{\prime}}\norm{\rho_{0}}_{\mathrm{L}^{1}(\mathbb{R})}^{2}   ,
\end{nalign}
Finally, in view of the estimates \eqref{eq:lrvaluediff_bd} and \eqref{eq:rhobardif}, we arrive at
\begin{nalign}\label{eq:L1c_bd}
    \sum_{\jinz}\abs{\mathcal{L}_j^{c}} & \leq \frac{\lambda^2}{4}(1+\theta)\Delta x \norm{\mu^{\prime}}\norm{\rho_{0}}_{\mathrm{L}^{1}(\mathbb{R})}  \norm{\partial_{A}f}  \norm{\partial_{\rho\rho}^{2}f}\sum_{\jinz}\abs{\bar{\rho}_{j-1}^{\nph,+}- \bar{\rho}_{j}^{\nph,+}}\\
    & \leq \frac{\lambda}{2}\Delta t (1+\theta) \norm{\mu^{\prime}}\norm{\rho_{0}}_{\mathrm{L}^{1}(\mathbb{R})}  \norm{\partial_{A}f}  \norm{\partial_{\rho\rho}^{2}f}\left(1+{\lambda}\norm{\partial_{\rho}f}\right)\sum_{\jinz}\abs{\rho^n_{j} - \rho_{j-1}^n} \\ & \spc + \frac{\lambda^2}{2}\Delta t M(1+\theta)^{3} \norm{\mu^{\prime}}^{2}\norm{\rho_{0}}_{\mathrm{L}^{1}(\mathbb{R})}^{3}  \norm{\partial_{A}f}  \norm{\partial_{\rho\rho}^{2}f}.  
\end{nalign}
Next, applying the estimates \eqref{eq:lrvaluediff_bd} and \eqref{eq:midtimeconvdiff}, we obtain a bound on the sum $
\displaystyle \sum_{\jinz}\abs{\mathcal{L}_j^{d}}$ as follows:
\begin{nalign}\label{eq:L1d_bd}
\sum_{\jinz}\abs{\mathcal{L}_j^{d}} & \leq \frac{1}{4}\Delta t^{2}(1+\theta)^{2} \norm{\mu^{\prime}}^{2}\norm{\rho_{0}}^{2}_{\mathrm{L}^{1}(\mathbb{R})} \norm{\partial_{A}f}  \norm{\partial_{\rho A}^{2}f} (1+ {\lambda}\norm{\partial_{\rho}f})\\ & \leq \frac{1}{4}\Delta t(1+\theta)^{2} \norm{\mu^{\prime}}^{2}\norm{\rho_{0}}^{2}_{\mathrm{L}^{1}(\mathbb{R})} \norm{\partial_{A}f}  \norm{\partial_{\rho A}^{2}f} (1+ {\lambda}\norm{\partial_{\rho}f}),
\end{nalign}
where the last inequality holds for $\Delta t \leq 1.$
Collecting the estimates \eqref{eq:L1a_bd}, \eqref{eq:L1b_bd}, \eqref{eq:L1c_bd} and \eqref{eq:L1d_bd} we arrive at
\begin{nalign}\label{eq:L1sumbd}
\sum_{\jinz}\abs{\mathcal{L}^{1}_{j}} &\leq \sum_{\jinz}\abs{\mathcal{L}_j^{a}}+ \sum_{\jinz}\abs{\mathcal{L}_j^{b}}+ \sum_{\jinz}\abs{\mathcal{L}_j^{c}}+ \sum_{\jinz}\abs{\mathcal{L}_j^{d}}\\
& \leq \mathcal{K}_{1}\Delta t + \mathcal{K}_{2} \Delta t \sum_{\jinz}\abs{\rho^n_{j} - \rho_{j-1}^n},
\end{nalign}
where \begin{nalign}\label{eq:K1K2}
    \mathcal{K}_{1}
    &= (1+\theta)^{2} \norm{\mu^{\prime\prime}}\norm{\rho_0}_{\mathrm{L}^1(\mathbb{R})}^{2}\frac{1}{4} \left(\alpha+ \lambda \norm{\partial_{\rho}}\right) M    \\& 
   \spc +\frac{1}{2}(1+\theta)^{3}\norm{\mu^{\prime}}^{2}\norm{\rho_{0}}_{\mathrm{L}^{1}(\mathbb{R})}^{3} \left( \alpha+\lambda\norm{\partial_{\rho}f}  + {\lambda^2}\norm{\partial_{A}f}  \norm{\partial_{\rho\rho}^{2}f}  \right)M\\
    & \spc + (1+\theta)^{2} \norm{\mu^{\prime}}^{2}\norm{\rho_{0}}^{2}_{\mathrm{L}^{1}(\mathbb{R})} \frac{1}{4}\norm{\partial_{A}f}  \norm{\partial_{\rho A}^{2}f} (1+ {\lambda}\norm{\partial_{\rho}f}),\\
    \mathcal{K}_{2}&:=  (1+\theta)\norm{\mu^{\prime}}\norm{\rho_{0}}_{\mathrm{L}^{1}(\mathbb{R})}  \left( \frac{1}{4}(\alpha+\lambda\norm{\partial_{\rho}f}) \norm{\partial_{\rho A}f}+ \frac{\lambda}{2}\norm{\partial_{A}f}  \norm{\partial_{\rho\rho}^{2}f}\left(1+{\lambda}\norm{\partial_{\rho}f}\right)  \right).
\end{nalign}
An analogous treatment of the term $\mathcal{L}^2_{j}$ yields
\begin{align}\label{eq:L2sumbd}
\sum_{\jinz}\abs{\mathcal{L}^2_{j}} &\leq \mathcal{K}_{1}\Delta t + \mathcal{K}_{2} \Delta t \sum_{\jinz}\abs{\rho^n_{j} - \rho_{j-1}^n}.  
\end{align}
Now, the estimates \eqref{eq:L1sumbd} and \eqref{eq:L2sumbd} on \eqref{eq:hatldif}, it follows that
\begin{nalign}\label{eq:hatldifbd}
    \sum_{\jinz}\abs{\hat{\ell}_{j+\frac{1}{2}}^n - \hat{\ell}_{j-\frac{1}{2}}^n} &\leq \sum_{\jinz}\abs{\mathcal{L}^1_{j}}+\sum_{\jinz}\abs{\mathcal{L}^2_{j}}\leq 2\mathcal{K}_{1}\Delta t + 2\mathcal{K}_{2} \Delta t \sum_{\jinz}\abs{\rho^n_{j} - \rho_{j-1}^n}.
\end{nalign}
In a similar way, we obtain a bound
\begin{nalign}\label{eq:hatkdif_bd}
    \sum_{\jinz}\abs{\hat{k}_{j+\frac{1}{2}}^n - \hat{k}_{j-\frac{1}{2}}^n} & \leq 2\mathcal{K}_{1}\Delta t + 2\mathcal{K}_{2} \Delta t \sum_{\jinz}\abs{\rho^n_{j} - \rho_{j-1}^n}.
\end{nalign}

Thus, in view of \eqref{eq:hatldifbd} and \eqref{eq:hatkdif_bd},  we conclude from \eqref{eq:tildehatCs} that 
\begin{nalign}\label{eq:hatCsumbd}
    \sum_\jinz \abs{\hat{C}_{\jph}^{n}} & \leq \sum_{\jinz}\abs{\hat{k}_{j+\frac{1}{2}}^n - \hat{k}_{j-\frac{1}{2}}^n} + \sum_{\jinz}\abs{\hat{\ell}_{j+\frac{1}{2}}^n - \hat{\ell}_{j-\frac{1}{2}}^n}\\
    & \leq 4\mathcal{K}_{1}\Delta t + 4\mathcal{K}_{2} \Delta t \sum_{\jinz}\abs{\rho^n_{j+1} - \rho_{j}^n}.
\end{nalign}

\subsection{Bound on $\lambda\sum_{\jinz} \abs{D_{\jph}^n}$ }\label{subsec:Djphbd}
Plugging in the definition of the numerical flux \eqref{eq:midtimenumflux} in the expression  \eqref{eq:D_defn}, rearranging the terms and subsequently applying the mean value theorem, we write
\begin{nalign}\label{eq:Das_Dabcde}
    D^n_{\jph} 
& = \frac{1}{2}\left(f(\rho^{\nph,-}_{\jph}, A_{j+\frac{3}{2}}^{\nph}) - f(\rho^{\nph,-}_{\jph}, A_{\jph}^{\nph})\right) + \frac{1}{2} \left(f(\rho^{\nph,+}_{\jph}, A_{j+\frac{3}{2}}^{\nph}) - f(\rho^{\nph,+}_{\jph}, A_{\jph}^{\nph})\right)\\
%%%%%%%%%%%%%%%%%%%%%%%%%%
& \spc -\frac{1}{2}\left(f(\rho^{\nph,-}_{\jmh}, A_{j+\frac{1}{2}}^{\nph}) - f(\rho^{\nph,-}_{\jmh}, A_{\jmh}^{\nph})\right) - \frac{1}{2} \left(f(\rho^{\nph,+}_{\jmh}, A_{j+\frac{1}{2}}^{\nph}) - f(\rho^{\nph,+}_{\jmh}, A_{\jmh}^{\nph})\right)\\
& = \frac{1}{2} (A_{j+\frac{3}{2}}^{\nph}-A_{j+\frac{1}{2}}^{\nph})\left[\partial_{A}f(\rho_{\jph}^{\nph,-}, \bar{A}_{\jpo}^{\nph}) + \partial_{A}f(\rho_{\jph}^{\nph,+}, \tilde{A}_{\jpo}^{\nph}) \right] \\& \spc -\frac{1}{2} (A_{j+\frac{1}{2}}^{\nph}-A_{j-\frac{1}{2}}^{\nph})\left[ \partial_{A}f(\rho_{\jmh}^{\nph,-}, \bar{A}_{j}^{\nph})+\partial_{A}f(\rho_{\jmh}^{\nph,+}, \tilde{A}_{j}^{\nph})\right]
\end{nalign}
for $\displaystyle\bar{A}_{j}^{\nph}, \tilde{A}_{j}^{\nph} \in \mathcal{I}({A}_{\jmh}^{\nph}, {A}_{\jph}^{\nph}), \jinz.$
\par
Next, we add and subtract the term $\frac{1}{2} (A_{j+\frac{1}{2}}^{\nph}- A_{j-\frac{1}{2}}^{\nph})\left[\partial_{A}f(\rho^{\nph,-}_{\jph}, \bar{A}_{\jpo}^{\nph})+ \partial_{A}f(\rho^{\nph,+}_{\jph}, \tilde{A}_{\jpo}^{\nph})\right]$ to \eqref{eq:Das_Dabcde} and write 
\begin{nalign}\label{eq:Djph_refor}
    \mathcal{D}_{\jph}&= \frac{1}{2} \left((A_{j+\frac{3}{2}}^{\nph}-A_{j+\frac{1}{2}}^{\nph})- (A_{j+\frac{1}{2}}^{\nph}-A_{j-\frac{1}{2}}^{\nph})\right)\left[\partial_{A}f(\rho^{\nph,-}_{\jph}, \bar{A}_{\jpo}^{\nph})+ \partial_{A}f(\rho^{\nph,+}_{\jph}, \tilde{A}_{\jpo}^{\nph})\right] \\
& \spc + \frac{1}{2}(A^{\nph}_{\jph}-A^{\nph}_{\jmh}) \left[\partial_{A}f(\rho_{\jph}^{\nph,-}, \bar{A}_{\jpo}^{\nph})-\mathcal{S}_1+ \mathcal{S}_1- \partial_{A}f(\rho_{\jmh}^{\nph,-}, \bar{A}_{j}^{\nph})\right]\\
& \spc + \frac{1}{2}(A^{\nph}_{\jph}-A^{\nph}_{\jmh}) \left[\partial_{A}f(\rho_{\jph}^{\nph,+}, \tilde{A}_{\jpo}^{\nph})-\mathcal{S}_2+\mathcal{S}_2- \partial_{A}f(\rho_{\jmh}^{\nph,+}, \tilde{A}_{j}^{\nph})\right],
\end{nalign}
where $\mathcal{S}_{1}:= \partial_{A}f(\rho_{\jmh}^{\nph,-}, \bar{A}_{\jpo}^{\nph})$ and $\mathcal{S}_2:=\partial_{A}f(\rho_{\jmh}^{\nph,+}, \tilde{A}_{\jpo}^{\nph}).$  Once again applying the mean value theorem  in \eqref{eq:Djph_refor} yields
\begin{align}
    \mathcal{D}_{\jph} = \mathcal{D}_{\jph}^a + \mathcal{D}_{\jph}^b + \mathcal{D}_{\jph}^c+ \mathcal{D}_{\jph}^d + \mathcal{D}_{\jph}^e,
\end{align}
where
\begin{nalign}\label{eq:D_abcde}
    \mathcal{D}_{\jph}^a &:=  \frac{1}{2} (A_{j+\frac{3}{2}}^{\nph}-2A_{j+\frac{1}{2}}^{\nph}+ A_{j-\frac{1}{2}}^{\nph})\left[\partial_{A}f(\rho^{\nph,-}_{\jph}, \bar{A}_{\jpo}^{\nph})+ \partial_{A}f(\rho^{\nph,+}_{\jph}, \tilde{A}_{\jpo}^{\nph})\right], \\
    %%%%%%%%%%
    \mathcal{D}_{\jph}^b &:= \frac{1}{2}(A_{\jph}^\nph- A_{\jmh}^\nph) \left[\partial_{\rho A}^{2}f(\bar{\rho}_{j}^{\nph,-}, \bar{A}_{\jpo}^{\nph})(\rho^{\nph,-}_{\jph} - \rho^{\nph,-}_{\jmh}) \right],\\
    %%%%%%%%%%%%%%%%%%%
    \mathcal{D}_{\jph}^c &:=  \frac{1}{2}(A_{\jph}^\nph- A_{\jmh}^\nph)\left[ \partial^{2}_{AA}f(\rho_{\jmh}^{\nph,-}, \bar{\bar{A}}_{\jph}^{\nph})(\bar{A}^{\nph}_{\jpo} - \bar{A}^{\nph}_{j})\right],\\
    \mathcal{D}_{\jph}^d & := \frac{1}{2}(A_{\jph}^\nph- A_{\jmh}^\nph)\left[\partial_{\rho A}^{2}f(\bar{\rho}_{j}^{\nph,+}, \tilde{A}_{\jpo}^{\nph})(\rho^{\nph,+}_{\jph} - \rho^{\nph,+}_{\jmh}) \right],\\
   \mathcal{D}_{\jph}^e &:= \frac{1}{2}(A_{\jph}^\nph- A_{\jmh}^\nph)\left[ \partial_{AA}^{2}f(\rho_{\jmh}^{\nph,+}, \tilde{\tilde{A}}_{\jph}^{\nph})(\tilde{A}^{\nph}_{\jpo} - \tilde{A}^{\nph}_{j}) \right],
\end{nalign}
for some $\bar{\bar{A}}_{\jph}^{\nph} \in \mathcal{I}({\bar{A}}_{j}^{\nph}, {\bar{A}}_{\jpo}^{\nph})$ and $\bar{\rho}^{\nph, \pm}_{j} \in \mathcal{I}({\rho}^{\nph, \pm}_{\jmh}, {\rho}^{\nph, \pm}_{\jph}), \jinz.$

\par
Next our goal is to estimate the sum over $\jinz$ of absolute values of all the terms in \eqref{eq:Das_Dabcde}. To estimate the term $\displaystyle\sum_{\jinz}\abs{\mathcal{D}_{\jph}^a}$, we first expand $A_{j+\frac{3}{2}}^{\nph}-2A_{j+\frac{1}{2}}^{\nph}+ A_{j-\frac{1}{2}}^{\nph},$ sum its absolute value over $\jinz$ and  successively apply the mean value theorem twice to write
\begin{nalign}\label{eq:sum_conv_3_mt}    \sum_{\jinz}\abs{A_{j+\frac{3}{2}}^{\nph}-2A_{j+\frac{1}{2}}^{\nph}+ A_{j-\frac{1}{2}}^{\nph}} &\leq \frac{\Delta x}{2} \sum_{\jinz} \abs{\mu_{j+2-l}-2\mu_{j+1-l}+ \mu_{j-l}}\abs{\rho_{l-\half}^{\nph,+}}\\  & \spc  + \frac{\Delta x}{2} \sum_{\jinz} \abs{\mu_{j+1-l}-2\mu_{j-l}+ \mu_{j-1-l}}\abs{\rho_{l+\half}^{\nph,-}} \\
    & \leq  {\Delta x^{2}}\norm{\mu^{\prime\prime}}(1+\theta)(1+ {\lambda}\norm{\partial_{\rho}f}) \norm{\rho_{0}}_{\mathrm{L}^{1}(\mathbb{R})},
\end{nalign}
where the last inequality follows from the estimate \eqref{eq:midvalue_bd}.
Further, invoking \eqref{eq:sum_conv_3_mt}, hypothesis \eqref{hyp:H2} and \eqref{eq:midvalue_bd}, we obtain \begin{nalign}\label{eq:Dasumbd}
    \sum_{\jinz}\abs{\mathcal{D}_{\jph}^a} & \leq \frac{1}{2}{\Delta x^{2}}\norm{\mu^{\prime\prime}}(1+\theta)(1+ {\lambda}\norm{\partial_{\rho}f}) \norm{\rho_{0}}_{\mathrm{L}^{1}(\mathbb{R})} M\sum_{\jinz}( \abs{\rho^{\nph,-}_{\jph}}+ \abs{\rho^{\nph,+}_{\jph}})\\
    & \leq {\Delta x^{2}}\norm{\mu^{\prime\prime}}(1+\theta)(1+ {\lambda}\norm{\partial_{\rho}f}) \norm{\rho_{0}}_{\mathrm{L}^{1}(\mathbb{R})} M(1+\theta)\left(1+{\lambda}\norm{\partial_{\rho}f}\right)\sum_{\jinz}\rho_{j}^n  \\
    & \leq {\Delta x^{2}}\norm{\mu^{\prime\prime}}(1+\theta)^{2}(1+ {\lambda}\norm{\partial_{\rho}f})^{2} \norm{\rho_{0}}_{\mathrm{L}^{1}(\mathbb{R})}^{2} M.
\end{nalign}
Next, to derive a bound on the term $\displaystyle \sum_{\jinz}\abs{\mathcal{D}_{\jph}^c}$, we use the estimate \eqref{eq:midtimeconvdiff}  to write
\begin{nalign}\label{eq:tildebardiffs}
    \abs{\tilde{A}_{\jpo}^{\nph}-\tilde{A}_{j}^{\nph}}, \abs{\bar{A}_{\jpo}^{\nph}-\bar{A}_{j}^{\nph}} &\leq \abs{\gamma_{1} A_{j+\frac{3}{2}}^{\nph} + (1-\gamma_{1})A_{j+\frac{1}{2}}^{\nph} - \gamma_{2} A_{j+\frac{1}{2}}^{\nph} - (1-\gamma_{2})A_{j-\frac{1}{2}}^{\nph} }\\
    & \leq \gamma_{1}\abs{A_{j+\frac{3}{2}}^{\nph} - A_{j+\frac{1}{2}}^{\nph}} + (1-\gamma_{2})\abs{A^{\nph}_{\jph} - A^{\nph}_{\jmh}}\\
    & \leq 2{\Delta x}\norm{\mu^{\prime}}(1+\theta)  (1+ {\lambda}\norm{\partial_{\rho}f})\norm{\rho_{0}}_{\mathrm{L}^{1}(\mathbb{R})}, 
\end{nalign}
for some $\gamma_{1}, \gamma_{2} \in (0,1).$ Now, invoking the estimates \eqref{eq:midtimeconvdiff} and \eqref{eq:tildebardiffs}, hypothesis \eqref{hyp:H2} and Lemma \ref{lemma:midtimebd}, we obtain
\begin{nalign}\label{eq:Dcsumbd}
\sum_{\jinz}\abs{\mathcal{D}_{\jph}^c} &\leq{\Delta x}^2 M \norm{\mu^{\prime}}^{2}(1+\theta)^{2}  (1+ {\lambda}\norm{\partial_{\rho}f})^{2}\norm{\rho_{0}}^{2}_{\mathrm{L}^{1}(\mathbb{R})}  \sum_{\jinz} \abs{\rho_{\jmh}^{\nph,-}} \\
&\leq{\Delta x}M\norm{\mu^{\prime}}^{2}(1+\theta)^{3}  (1+ {\lambda}\norm{\partial_{\rho}f})^{3}\norm{\rho_{0}}^{2}_{\mathrm{L}^{1}(\mathbb{R})}  {\Delta x}\sum_{\jinz} \rho_{j-1}^n\\
%%%%%%%%%%%%%%%%%%%%%%%
& \leq {\Delta x}M\norm{\mu^{\prime}}^{2}(1+\theta)^{3}  (1+ {\lambda}\norm{\partial_{\rho}f})^{3}\norm{\rho_{0}}^{3}_{\mathrm{L}^{1}(\mathbb{R})} ,
\end{nalign} where the last inequality holds for  $\Delta x \leq 1.$
Analogously, we can estimate the term $\displaystyle\sum_{\jinz}\abs{\mathcal{D}_{\jph}^e}$ as follows
\begin{align}\label{eq:Desumbd}
\sum_{\jinz}\abs{\mathcal{D}_{\jph}^e} &\leq  {\Delta x}M\norm{\mu^{\prime}}^{2}(1+\theta)^{3}  (1+ {\lambda}\norm{\partial_{\rho}f})^{3}\norm{\rho_{0}}^{3}_{\mathrm{L}^{1}(\mathbb{R})} .
\end{align}
\par
Next, we proceed to estimate the terms $\displaystyle \sum_{\jinz}\abs{\mathcal{D}_{\jph}^b} \, \mbox{and} \, \sum_{\jinz}\abs{\mathcal{D}_{\jph}^d}.$ To this end, we sum  \eqref{eq:rhormidtimediffbd} over $\jinz$
and subsequently apply property \eqref{eq:tv_recomstruction} and Lemma \ref{lemma:reconvaluebd} to yield 
\begin{nalign}\label{eq:sum_midtimerdif}
    \sum_\jinz\abs{\rho_{j+\frac{1}{2}}^{\nph,+}-\rho_{j-\frac{1}{2}}^{\nph,+}} & \leq (1+\lambda \norm{\partial_\rho f}) \sum_{\jinz}\abs{\rho_{\jpo}^n - \rho_{j}^n} + {\lambda}(1+\theta)^{2} \norm{\mu^{\prime}}\norm{\rho_{0}}_{\mathrm{L}^{1}(\mathbb{R})}M\Delta x\sum_{\jinz}\rho_{j}^{n} \\
    & \leq (1+\lambda \norm{\partial_\rho f}) \sum_{\jinz}\abs{\rho_{\jpo}^n - \rho_{j}^n} + {\lambda}(1+\theta)^{2} \norm{\mu^{\prime}}\norm{\rho_{0}}^{2}_{\mathrm{L}^{1}(\mathbb{R})}M.
\end{nalign}
Similarly, we obtain
\begin{align}\label{eq:sum_midtimeldif}
    \sum_\jinz\abs{\rho_{j+\frac{1}{2}}^{\nph,-}-\rho_{j-\frac{1}{2}}^{\nph,-}} &\leq (1+\lambda \norm{\partial_\rho f}) \sum_{\jinz}\abs{\rho_{\jpo}^n - \rho_{j}^n} + {\lambda}(1+\theta)^{2} \norm{\mu^{\prime}}\norm{\rho_{0}}^{2}_{\mathrm{L}^{1}(\mathbb{R})}M. 
\end{align}
Now, invoking the estimates \eqref{eq:sum_midtimerdif}, \eqref{eq:sum_midtimeldif} and  \eqref{eq:midtimeconvdiff}, we obtain 
\begin{nalign}\label{eq:DbDdsumbd}
    \sum_{\jinz}\abs{\mathcal{D}_{\jph}^b}, \sum_{\jinz}\abs{\mathcal{D}_{\jph}^d} &\leq \frac{1}{2}{\Delta x}\norm{\mu^{\prime}}(1+\theta)  (1+ {\lambda}\norm{\partial_{\rho}f})^{2}\norm{\rho_{0}}_{\mathrm{L}^{1}(\mathbb{R})}\norm{\partial_{\rho A}^{2}f} \sum_{\jinz}\abs{\rho_{\jpo}^n - \rho_{j}^n} \\
    &  \spc + {\Delta t}\frac{1}{2}\norm{\mu^{\prime}}^{2}(1+\theta)^{3}  (1+ {\lambda}\norm{\partial_{\rho}f})\norm{\rho_{0}}^{3}_{\mathrm{L}^{1}(\mathbb{R})}\norm{\partial_{\rho A}^{2}f}M.
    \end{nalign}
Finally, the estimates \eqref{eq:Dasumbd}, \eqref{eq:DbDdsumbd}, \eqref{eq:Dcsumbd} and \eqref{eq:Desumbd} together yield the desired estimate:
\begin{nalign}\label{eq:Dsumbd}
    \lambda\sum_{\jinz} \abs{D_{\jph}^n} & \leq \Delta t \mathcal{K}_{7} + \Delta t \mathcal{K}_{8}\sum_{\jinz}\abs{\rho_{j+1}^n-\rho_{j}^n}, 
\end{nalign}
where 
\begin{nalign}\label{eq:K7K8}
    \mathcal{K}_{7} &:= \norm{\mu^{\prime\prime}}(1+\theta)^{2}(1+ {\lambda}\norm{\partial_{\rho}f})^{2} \norm{\rho_{0}}_{\mathrm{L}^{1}(\mathbb{R})}^{2} M \\ & \spc + \norm{\mu^{\prime}}^{2}(1+\theta)^{3}  (1+ {\lambda}\norm{\partial_{\rho}f})\norm{\rho_{0}}^{3}_{\mathrm{L}^{1}(\mathbb{R})}\norm{\partial_{\rho A}^{2}f}M\\ & \spc +  2{\Delta x}\norm{\mu^{\prime}}^{2}(1+\theta)^{3}  (1+ {\lambda}\norm{\partial_{\rho}f})^{3}\norm{\rho_{0}}^{3}_{\mathrm{L}^{1}(\mathbb{R})} M,\\
    \mathcal{K}_{8} &:=\norm{\mu^{\prime}}(1+\theta)  (1+ {\lambda}\norm{\partial_{\rho}f})^{2}\norm{\rho_{0}}_{\mathrm{L}^{1}(\mathbb{R})}\norm{\partial_{\rho A}^{2}f}.
\end{nalign}

% \section{Comparison with other problem setups in literature}
% \begin{itemize}
%     \item Amorim's case \cite{amorim2015}, \cblue{with the extra compact kernel support assumption (H3)}, fits into our framework.
%     \item Aggarwal's recent case \cite{aggarwal_holden_vaidya2024} fits into our case except for the fact that the flux is required only to be Lipschitz w.r.to the variable $\rho$ in \cite{aggarwal_holden_vaidya2024}, in contrast to the differentiability we assume. (Can we also assume the weaker Lipschitz condition? Depends on whether the well posedness of Amorim \cite{amorim2015} works without differentiability.) Further, we have \cblue{the extra compact kernel support assumption (H3).}
%     \item In the problem setup $(f(\rho, \rho *\mu)= g(\rho)v(\rho * \mu))$ of \cite{aggarwal_holden_vaidya2024}, one has the option of using the Godunov flux - by fixing $V_{\jph}$ and finding the usual Godunov flux for the part $g.$  In contrast, for our case (general $f(\rho, \rho * \mu)$) we do not have a Godunov type flux and we use the Lax-Friedrichs flux. As a remark, we can write that the MH scheme and its convergence analysis can be done also with the Godunov flux, for the problem in Aggarwal \cite{aggarwal2015}.
% \end{itemize}

\section{A second-order MUSCL scheme with Runge-Kutta time stepping}\label{section:rkscheme}
We combine the MUSCL-type reconstruction \eqref{eq:reconstruction_pl} and a
Runge-Kutta time-stepping (see \cite{gottlieb1998otalVD, shu1988}) to obtain a two-stage second-order method to approximate \eqref{eq:problem}. Given the cell-average solutions $\{\rho_{j}^n\}_{\jinz}$ at $t^n,$ we apply two consecutive Euler-forward stages to obtain a fully discrete second-order scheme  as follows:\\
{\bf Step 1.} Define
\begin{nalign}
    &\rho^{(1)}_{j} = \rho^{n}_{j} - \lambda[F(\rho_{j+\frac{1}{2}}^{n,-}, \rho_{j+\frac{1}{2}}^{n, +}, A^{n}_{j+\frac{1}{2}}) - F(\rho_{j-\frac{1}{2}}^{n,-}, \rho_{j-\frac{1}{2}}^{n,+}, A^{n}_{j-\frac{1}{2}})] \;\;\;\mbox{for each }\; j \in \mathbb{Z}, \no
\end{nalign} 
where $F$ is given by (\ref{eq:lxf}),  $\rho_{j\mp\frac{1}{2}}^{n,\pm}, \jinz,$ are obtained from \eqref{eq:reconvalues} and the discrete convolutions are computed using the trapezoidal quadrature rule as follows: \begin{align*}
    A^{n}_{j+\frac{1}{2}} &:= \frac{\Delta x}{2}  \sum_{l \in \mathbb{Z}}\left[\mu_{j+1-l}        \ \rho^{n,+}_{l-\frac{1}{2}}+ \mu_{j-l} \ \rho^{n,-}_{l+\frac{1}{2}}\right], \quad \jinz.
\end{align*}\\
{\bf Step 2.} Next, we apply the slope limiter \eqref{slope-1} to $\{\rho^{(1)}_{j}\}_{\jinz}$  to obtain the face values
$
\rho_{j+\frac{1}{2}}^{(1),\pm}, \jinz.$
Now, define 
\begin{nalign}\label{step-3}
    &\rho^{(2)}_{j} = \rho^{(1)}_{j} - \lambda[F(\rho_{j+\frac{1}{2}}^{(1),-}, \rho_{j+\frac{1}{2}}^{(1), +}, A^{(1)}_{j+\frac{1}{2}}) - F(\rho_{j-\frac{1}{2}}^{(1),-}, \rho_{j-\frac{1}{2}}^{(1),+}, A^{(1)}_{j-\frac{1}{2} })]\no 
 \end{nalign} 
where
$$A^{(1)}_{j+\frac{1}{2}} := \frac{\Delta x}{2}  \sum_{l \in \mathbb{Z}}\left[\mu_{j+1-l}        \ \rho^{(1),+}_{l-\frac{1}{2}}+ \mu_{j-l} \ \rho^{(1),-}_{l+\frac{1}{2}}\right], \quad \jinz.$$
{\bf Step 3.} Finally, the updated solution at the time level $t^{n+1}$ is computed as
\begin{align} \label{RK-2d}
\rho_{j}^{n+1}=\frac{\rho^{n}_{j}+\rho^{(2)}_{j}}{2}, \quad \jinz.
\end{align} 

\begin{remark}
The convergence of the scheme \eqref{RK-2d} to the entropy solution of problem \eqref{eq:problem} can be established using calculations analogous to those for the scheme \eqref{eq:mh}, within the general framework developed in \cite{gowda2023}.
\end{remark}

% \section{On the convergence of MUSCL-Hancock scheme for local conservation laws}

\bibliographystyle{siam}
\bibliography{ref}

\def\ocirc#1{\ifmmode\setbox0=\hbox{$#1$}\dimen0=\ht0 \advance\dimen0
  by1pt\rlap{\hbox to\wd0{\hss\raise\dimen0
  \hbox{\hskip.2em$\scriptscriptstyle\circ$}\hss}}#1\else {\accent"17 #1}\fi}
\begin{thebibliography}{10}

\bibitem{aggarwal2015}
{\sc A.~Aggarwal, R.~M. Colombo, and P.~Goatin}, {\em Nonlocal systems of
  conservation laws in several space dimensions}, SIAM J. Numer. Anal., 53
  (2015), pp.~963--983.

\bibitem{aggarwal_holden_vaidya2024}
{\sc A.~Aggarwal, H.~Holden, and G.~Vaidya}, {\em On the accuracy of the finite
  volume approximations to nonlocal conservation laws}, Numer. Math., 156
  (2024), pp.~237--271.

\bibitem{amorim2015}
{\sc P.~Amorim, R.~M. Colombo, and A.~Teixeira}, {\em On the numerical
  integration of scalar nonlocal conservation laws}, ESAIM Math. Model. Numer.
  Anal., 49 (2015), pp.~19--37.

\bibitem{bayen2021}
{\sc A.~Bayen, J.-M. Coron, N.~De~Nitti, A.~Keimer, and L.~Pflug}, {\em
  Boundary controllability and asymptotic stabilization of a nonlocal traffic
  flow model}, Vietnam J. Math., 49 (2021), pp.~957--985.

\bibitem{berthon2006}
{\sc C.~Berthon}, {\em Why the {MUSCL}-{H}ancock scheme is {$L^1$}-stable},
  Numer. Math., 104 (2006), pp.~27--46.

\bibitem{betancourt2011}
{\sc F.~Betancourt, R.~B\"{u}rger, K.~H. Karlsen, and E.~M. Tory}, {\em On
  nonlocal conservation laws modelling sedimentation}, Nonlinearity, 24 (2011),
  pp.~855--885.

\bibitem{blandin2016}
{\sc S.~Blandin and P.~Goatin}, {\em Well-posedness of a conservation law with
  non-local flux arising in traffic flow modeling}, Numer. Math., 132 (2016),
  pp.~217--241.

\bibitem{burgergoatin2020}
{\sc R.~B\"{u}rger, P.~Goatin, D.~Inzunza, and L.~M. Villada}, {\em A non-local
  pedestrian flow model accounting for anisotropic interactions and domain
  boundaries}, Math. Biosci. Eng., 17 (2020), pp.~5883--5906.

\bibitem{chalons2018}
{\sc C.~Chalons, P.~Goatin, and L.~M. Villada}, {\em High-order numerical
  schemes for one-dimensional nonlocal conservation laws}, SIAM J. Sci.
  Comput., 40 (2018), pp.~A288--A305.

\bibitem{chandrasekhar2020}
{\sc P.~Chandrashekar, B.~Nkonga, A.~K. Meena, and A.~Bhole}, {\em A path
  conservative finite volume method for a shear shallow water model}, J.
  Comput. Phys., 413 (2020), pp.~109457, 29.

\bibitem{chiarello-globalentropy}
{\sc F.~A. Chiarello and P.~Goatin}, {\em Global entropy weak solutions for
  general non-local traffic flow models with anisotropic kernel}, ESAIM Math.
  Model. Numer. Anal., 52 (2018), pp.~163--180.

\bibitem{chiarello2018}
\leavevmode\vrule height 2pt depth -1.6pt width 23pt, {\em {Non-local
  multi-class traffic flow models}}.
\newblock Networks and Heterogeneous Media, to appear, Aug. 2018.

\bibitem{chiarellotosin2023}
{\sc F.~A. Chiarello and A.~Tosin}, {\em Macroscopic limits of non-local
  kinetic descriptions of vehicular traffic}, Kinet. Relat. Models, 16 (2023),
  pp.~540--564.

\bibitem{ciotirfayad2021}
{\sc I.~Ciotir, R.~Fayad, N.~Forcadel, and A.~Tonnoir}, {\em A non-local
  macroscopic model for traffic flow}, ESAIM Math. Model. Numer. Anal., 55
  (2021), pp.~689--711.

\bibitem{colomboGaravelloMercier2012}
{\sc R.~M. Colombo, M.~Garavello, and M.~L\'{e}cureux-Mercier}, {\em A class of
  nonlocal models for pedestrian traffic}, Math. Models Methods Appl. Sci., 22
  (2012), pp.~Paper No. 1150023, 34.

\bibitem{ColomboHertyMercier2011}
{\sc R.~M. Colombo, M.~Herty, and M.~Mercier}, {\em Control of the continuity
  equation with a non local flow}, ESAIM Control Optim. Calc. Var., 17 (2011),
  pp.~353--379.

\bibitem{ColomboMercier2012}
{\sc R.~M. Colombo and M.~L\'{e}cureux-Mercier}, {\em Nonlocal crowd dynamics
  models for several populations}, Acta Math. Sci. Ser. B (Engl. Ed.), 32
  (2012), pp.~177--196.

\bibitem{friedrich2019CWENO}
{\sc J.~Friedrich and O.~Kolb}, {\em Maximum principle satisfying {CWENO}
  schemes for nonlocal conservation laws}, SIAM J. Sci. Comput., 41 (2019),
  pp.~A973--A988.

\bibitem{friedrich2018}
{\sc J.~Friedrich, O.~Kolb, and S.~G\"{o}ttlich}, {\em A {G}odunov type scheme
  for a class of {LWR} traffic flow models with non-local flux}, Netw. Heterog.
  Media, 13 (2018), pp.~531--547.

\bibitem{friedrich2023}
{\sc J.~Friedrich, S.~Sudha, and S.~Rathan}, {\em Numerical schemes for a class
  of nonlocal conservation laws: a general approach}, Netw. Heterog. Media, 18
  (2023), pp.~1335--1354.

\bibitem{godlewski1991hyperbolic}
{\sc E.~Godlewski and P.-A. Raviart}, {\em Hyperbolic systems of conservation
  laws}, vol.~3/4 of Math\'{e}matiques \& Applications (Paris) [Mathematics and
  Applications], Ellipses, Paris, 1991.

\bibitem{gottlieb1998otalVD}
{\sc S.~Gottlieb and C.-W. Shu}, {\em Total variation diminishing runge-kutta
  schemes}, Math. Comput., 67 (1998), pp.~73--85.

\bibitem{guinot2012}
{\sc V.~Guinot and C.~Delenne}, {\em M{USCL} schemes for the shallow water
  sensitivity equations with passive scalar transport}, Comput. \& Fluids, 59
  (2012), pp.~11--30.

\bibitem{keimer2017}
{\sc A.~Keimer and L.~Pflug}, {\em Existence, uniqueness and regularity results
  on nonlocal balance laws}, J. Differential Equations, 263 (2017),
  pp.~4023--4069.

\bibitem{keimer2019}
\leavevmode\vrule height 2pt depth -1.6pt width 23pt, {\em On approximation of
  local conservation laws by nonlocal conservation laws}, J. Math. Anal. Appl.,
  475 (2019), pp.~1927--1955.

\bibitem{Keimer2018}
{\sc A.~Keimer, L.~Pflug, and M.~Spinola}, {\em Existence, uniqueness and
  regularity of multi-dimensional nonlocal balance laws with damping}, J. Math.
  Anal. Appl., 466 (2018), pp.~18--55.

\bibitem{leroux1981}
{\sc A.~Y. Le~Roux}, {\em Convergence of an accurate scheme for first order
  quasilinear equations}, RAIRO Anal. Num\'{e}r., 15 (1981), pp.~151--170.

\bibitem{manoj2024}
{\sc N.~Manoj, G.~V. Gowda, and S.~K. Kenettinkara}, {\em A positivity
  preserving second-order scheme for multi-dimensional system of non-local
  conservation laws}, arXiv preprint arXiv:2412.18475v2,  (2024).

\bibitem{shu1988}
{\sc C.-W. Shu and S.~Osher}, {\em Efficient implementation of essentially
  nonoscillatory shock-capturing schemes}, J. Comput. Phys., 77 (1988),
  pp.~439--471.

\bibitem{SopasakisKatsoulakis2006}
{\sc A.~Sopasakis and M.~A. Katsoulakis}, {\em Stochastic modeling and
  simulation of traffic flow: asymmetric single exclusion process with
  {A}rrhenius look-ahead dynamics}, SIAM J. Appl. Math., 66 (2006),
  pp.~921--944.

\bibitem{tong2023}
{\sc W.~Tong, R.~Yan, and G.~Chen}, {\em On a class of robust bound-preserving
  {MUSCL}-{H}ancock schemes}, J. Comput. Phys., 474 (2023), pp.~Paper No.
  111805, 21.

\bibitem{tong2024}
\leavevmode\vrule height 2pt depth -1.6pt width 23pt, {\em A class of
  bound-preserving {MUSCL}-{H}ancock schemes in two dimensions}, J. Comput.
  Phys., 498 (2024), pp.~Paper No. 112668, 28.

\bibitem{van1979towards}
{\sc B.~van Leer}, {\em Towards the ultimate conservative difference scheme. v.
  a second-order sequel to godunov's method}, J. Comput. Phys., 32 (1979),
  pp.~101--136.

\bibitem{vanleer1984}
\leavevmode\vrule height 2pt depth -1.6pt width 23pt, {\em On the relation
  between the upwind-differencing schemes of {G}odunov, {E}ngquist-{O}sher and
  {R}oe}, SIAM J. Sci. Statist. Comput., 5 (1984), pp.~1--20.

\bibitem{gowda2023}
{\sc G.~D. Veerappa~Gowda, K.~Sudarshan~Kumar, and N.~Manoj}, {\em Convergence
  of a second-order scheme for non-local conservation laws}, ESAIM Math. Model.
  Numer. Anal., 57 (2023), pp.~3439--3481.

\bibitem{viallon1991}
{\sc M.-C. Viallon}, {\em Convergence of the two-point upstream weighting
  scheme}, Math. Comp., 57 (1991), pp.~569--584.

\bibitem{vila1988}
{\sc J.-P. Vila}, {\em High-order schemes and entropy condition for nonlinear
  hyperbolic systems of conservation laws}, Math. Comp., 50 (1988), pp.~53--73.

\bibitem{vila1989}
\leavevmode\vrule height 2pt depth -1.6pt width 23pt, {\em An analysis of a
  class of second-order accurate {G}odunov-type schemes}, SIAM J. Numer. Anal.,
  26 (1989), pp.~830--853.

\bibitem{waagan2009}
{\sc K.~Waagan}, {\em A positive {MUSCL}-{H}ancock scheme for ideal
  magnetohydrodynamics}, J. Comput. Phys., 228 (2009), pp.~8609--8626.

\end{thebibliography}
\end{document}